\documentclass[a4paper,11pt,twoside]{article}
\usepackage[
    top=2.75cm, 
    bottom=2.5cm, 
    inner=2.5cm,
    outer=2.5cm
]{geometry}

\usepackage{times}
\usepackage[british]{babel}

\title{\LARGE \textsc{Ricci-flat string algebroids} }
\author{
    Agnaldo A. da Silva Jr.\footnote{\texttt{agnaldosilvajr@ime.unicamp.br}},
\quad 
     Henrique N. Sá Earp\footnote{\texttt{henrique.saearp@ime.unicamp.br}}
\qandq
    Bernardo S. Vieira\footnote{\texttt{bernardo.vieira@ime.unicamp.br}}
\\[0.5em]
     Universidade Estadual de Campinas (UNICAMP)
}
\date{}

\usepackage[utf8]{inputenc}

\usepackage{amsmath,amssymb,amsfonts,amsthm,mathtools}
\usepackage{mathrsfs}
\usepackage{stmaryrd}
\usepackage{bbm}
\usepackage{xfrac}
\usepackage{colonequals}
\usepackage{cancel}
\usepackage{slashed}
\usepackage{accents}

\usepackage{graphicx}
\usepackage{tikz}
\usepackage{tikz-cd}
\usepackage{xcolor}
\usepackage{rotating}

\usepackage{abstract}
\usepackage{caption}
\usepackage{titlesec}
\usepackage{multicol}
\usepackage[shortlabels]{enumitem}
\usepackage{framed}

\usepackage{anyfontsize}
\usepackage{textgreek}
\usepackage{soul}
\usepackage{blindtext}

\usepackage{listings}
\usepackage{sidenotes}
\usepackage[size=\tiny]{todonotes}
\usepackage{verbatim}

\usepackage{import}
\usepackage{xifthen}
\usepackage{xstring}
\usepackage{pdfpages}
\usepackage{transparent}
\usepackage{calc}

\usepackage[pageanchor, backref=page, colorlinks=true, linkcolor = blue]{hyperref}
\hypersetup{colorlinks=true, citecolor=blue, urlcolor = red}
\renewcommand*{\backref}[1]{}
\renewcommand*{\backrefalt}[4]
{%$\Rsh$
    \ifcase #1%
        \or        (p.~#2)% cite of one page
        \else      (pp.~#2)% cite of several pages 
    \fi
}

\setuptodonotes{
  color=white,          % fundo do box
  bordercolor=red!60!black, % borda do box
  linecolor=blue!60!black,   % linha até a margem
  size=\small                % tamanho da fonte nas notas
}

\makeatletter
\@addtoreset{paragraph}{subsection}
\makeatother

\renewcommand{\theparagraph}{\S\thesubsection.\arabic{paragraph}}

\titleformat{\paragraph}[runin]
    {\normalfont\normalsize\bfseries}
    {\theparagraph}
    {0.5em}
    {}
    [:]
\usepackage{fancyhdr}
\newcommand{\shortauthors}{Agnaldo A. da Silva Jr., Henrique N. Sá Earp, Bernardo S. Vieira}
\newcommand{\shorttitulo}{Ricci-flat string algebroids}
\begin{document}
\emergencystretch 3em 

\numberwithin{equation}{section}

%comands
%% comandos de simplificação de texto matematico
\newcommand{\bb}{\mathbb}
\newcommand{\mcal}{\mathcal}
\newcommand{\vet}{\mathbf}
\newcommand{\rmm}{\mathrm}
\newcommand{\rmend}{\rmm{End}}
\newcommand{\rmhom}{\rmm{Hom}}
\newcommand{\Lg}{\mathfrak{g}}
\newcommand{\Lh}{\mathfrak{h}}
\newcommand{\brackss}[2]{\langle #1,#2\rangle}
\newcommand{\fsp}{\fs\fp }
\newcommand{\im}{\rmm{im}\,}
\newcommand{\bracks}[1]{\left\langle #1 \right\rangle}
\newcommand{\pp}[1]{\left( {#1} \right)}
\newcommand{\chaves}[1]{\left\{#1\right\}}
\newcommand{\colchetes}[1]{\left[#1\right]}
\newcommand{\largemath}[1]{\textrm{\Large\(#1\)}}
\newcommand{\smallmath}[1]{\textrm{\footnotesize\(#1\)}}
\newcommand{\hint}[1]{\\ \textit{{\footnotesize Dica: #1 }}}
\newcommand{\ppp}[1]{( #1 )}
\newcommand{\Fr}{\rmm{Fr}}
\newcommand{\al}{\alph*)}
\newcommand{\noit}{\textnormal{\roman*)}}
\newcommand{\llangle}{\langle\!\langle}
\newcommand{\rrangle}{\rangle\!\rangle}
\newcommand{\proj}{\rmm{proj}}
%% comandos de atalho de texto
\newcommand{\id}{\rmm{Id}}
\newcommand{\rank}{\rmm{rank}}
\newcommand{\actionL}{\circlearrowright}
\newcommand{\actionR}{\circlearrowleft}
\newcommand{\iprod}{\mathbin{\lrcorner}}
\newcommand{\adjointbundle}{\rmm{ad}}
\newcommand{\adjointbundleP}{\rmm{ad}P}
\newcommand{\rmmtr}{\rmm{tr}}
\newcommand{\ricci}{\rmm{Ric}}
\newcommand{\vol}{\rmm{vol}}
\newcommand{\dv}{\rmm{div}}
\newcommand{\ida}{\((\Rightarrow)\)\;}
\newcommand{\volta}{\((\Leftarrow)\)\;}
\newcommand{\isomorphic}{\cong}
\renewcommand{\r}{\rmm}
\newcommand{\SO}{\mathrm{SO}}
\newcommand{\SU}{\mathrm{SU}}
\newcommand{\Sp}{\mathrm{Sp}}
\renewcommand{\sp}{\Sp}
\newcommand{\spksp}{\sp(k)\sp(1)}
\newcommand{\ad}{\adjointbundle}
\newcommand{\rG}{\mathrm{G}}
\newcommand{\Cl}{\rmm{Cl}}
\newcommand{\cl}{\Cl}
\newcommand{\la}{\langle}
\newcommand{\ra}{\rangle}
\newcommand{\GRic}{\rmm{GRic}}
\newcommand{\Ric}{\rmm{Ric}}
\newcommand{\Spin}{\rmm{Spin}}
\newcommand{\Gl}{\rmm{GL}}
\newcommand{\framebundle}{\rmm{Fr}}
\newcommand{\spin}{\Spin}
\newcommand{\GL}{\Gl}
\newcommand{\rg}{\r G}
\newcommand{\embeds}{\hookrightarrow}
\newcommand{\U}{\r U}
\newcommand{\simast}{\overset{*}{\sim}}
\newcommand{\leftbrack}{[\![}
\newcommand{\rightbrack}{]\!]}
\newcommand{\transitive}{TM\oplus \adjointbundleP\oplus T^*M}
\renewcommand{\star}{*}
\newcommand{\lambad}{\lambda}
\newcommand{\Lambad}{\Lambda}

%%comandos qed
\newcommand{\qedbarra}{\hfill{/}}
\newcommand{\qedduasbarra}{\hfill{//}}
\newcommand{\qedpreto}{\hfill{\blacksquare}}
\newcommand{\qedstar}{\hfill{\star}}
\newcommand{\qedball}{\hfill{$\bigcirc$}}
\newcommand{\qedtriangle}{\hfill{$\triangle$}}

%< LETTERS >====================================================================

% ROMAN
%\newcommand{\ra}{{\rm a}}
\newcommand{\rb}{{\rm b}}
\newcommand{\rc}{{\rm c}}
\newcommand{\rd}{{\rm d}}
\newcommand{\re}{{\rm e}}
\newcommand{\rf}{{\rm f}}
\newcommand{\rh}{{\rm h}}
\newcommand{\ri}{{\rm i}}
\newcommand{\rj}{{\rm j}}
\newcommand{\rmk}{{\rm k}}
\newcommand{\rl}{{\rm l}}
\newcommand{\rn}{{\rm n}}
\newcommand{\ro}{{\rm o}}
\newcommand{\rp}{{\rm p}}
\newcommand{\rmq}{{\rm q}}
\newcommand{\rr}{{\rm r}}
\newcommand{\rs}{{\rm s}}
\newcommand{\rt}{{\rm t}}
\newcommand{\ru}{{\rm u}}
\newcommand{\rv}{{\rm v}}
\newcommand{\rw}{{\rm w}}
\newcommand{\rx}{{\rm x}}
\newcommand{\ry}{{\rm y}}
\newcommand{\rz}{{\rm z}}
\newcommand{\rA}{{\rm A}}
\newcommand{\rB}{{\rm B}}
\newcommand{\rC}{{\rm C}}
\newcommand{\rD}{{\rm D}}
\newcommand{\rE}{{\rm E}}
\newcommand{\rF}{{\rm F}}
\newcommand{\rH}{{\rm H}}
\newcommand{\rI}{{\rm I}}
\newcommand{\rJ}{{\rm J}}
\newcommand{\rK}{{\rm K}}
\newcommand{\rL}{{\rm L}}
\newcommand{\rM}{{\rm M}}
\newcommand{\rN}{{\rm N}}
\newcommand{\rO}{{\rm O}}
\newcommand{\rP}{{\rm P}}
\newcommand{\rQ}{{\rm Q}}
\newcommand{\rR}{{\rm R}}
\newcommand{\rS}{{\rm S}}
\newcommand{\rT}{{\rm T}}
\newcommand{\rU}{{\rm U}}
\newcommand{\rV}{{\rm V}}
\newcommand{\rW}{{\rm W}}
\newcommand{\rX}{{\rm X}}
\newcommand{\rY}{{\rm Y}}
\newcommand{\rZ}{{\rm Z}}

% SANS SERIF
\newcommand{\sfa}{{\sf a}}
\newcommand{\sfb}{{\sf b}}
\newcommand{\sfc}{{\sf c}}
\newcommand{\sfd}{{\sf d}}
\newcommand{\sfe}{{\sf e}}
\newcommand{\sff}{{\sf f}}
\newcommand{\sfg}{{\sf g}}
\newcommand{\sfh}{{\sf h}}
\newcommand{\sfi}{{\sf i}}
\newcommand{\sfj}{{\sf j}}
\newcommand{\sfk}{{\sf k}}
\newcommand{\sfl}{{\sf l}}
\newcommand{\sfm}{{\sf m}}
\newcommand{\sfn}{{\sf n}}
\newcommand{\sfo}{{\sf o}}
\newcommand{\sfp}{{\sf p}}
\newcommand{\sfq}{{\sf q}}
\newcommand{\sfr}{{\sf r}}
\newcommand{\sfs}{{\sf s}}
\newcommand{\sft}{{\sf t}}
\newcommand{\sfu}{{\sf u}}
\newcommand{\sfv}{{\sf v}}
\newcommand{\sfw}{{\sf w}}
\newcommand{\sfx}{{\sf x}}
\newcommand{\sfy}{{\sf y}}
\newcommand{\sfz}{{\sf z}}
\newcommand{\sfA}{{\sf A}}
\newcommand{\sfB}{{\sf B}}
\newcommand{\sfC}{{\sf C}}
\newcommand{\sfD}{{\sf D}}
\newcommand{\sfE}{{\sf E}}
\newcommand{\sfF}{{\sf F}}
\newcommand{\sfG}{{\sf G}}
\newcommand{\sfH}{{\sf H}}
\newcommand{\sfI}{{\sf I}}
\newcommand{\sfJ}{{\sf J}}
\newcommand{\sfK}{{\sf K}}
\newcommand{\sfL}{{\sf L}}
\newcommand{\sfM}{{\sf M}}
\newcommand{\sfN}{{\sf N}}
\newcommand{\sfO}{{\sf O}}
\newcommand{\sfP}{{\sf P}}
\newcommand{\sfQ}{{\sf Q}}
\newcommand{\sfR}{{\sf R}}
\newcommand{\sfS}{{\sf S}}
\newcommand{\sfT}{{\sf T}}
\newcommand{\sfU}{{\sf U}}
\newcommand{\sfV}{{\sf V}}
\newcommand{\sfW}{{\sf W}}
\newcommand{\sfX}{{\sf X}}
\newcommand{\sfY}{{\sf Y}}
\newcommand{\sfZ}{{\sf Z}}

% UNDERLINED
\newcommand{\ua}{{\underline a}}
\newcommand{\ub}{{\underline b}}
\newcommand{\uc}{{\underline c}}
\newcommand{\ud}{{\underline d}}
\newcommand{\ue}{{\underline e}}
\newcommand{\uf}{{\underline f}}
\newcommand{\ug}{{\underline g}}
\newcommand{\uh}{{\underline h}}
\newcommand{\ui}{{\underline i}}
\newcommand{\uj}{{\underline j}}
\newcommand{\uk}{{\underline k}}
\newcommand{\um}{{\underline m}}
\newcommand{\un}{{\underline n}}
\newcommand{\uo}{{\underline o}}
\newcommand{\up}{{\underline p}}
\newcommand{\uq}{{\underline q}}
\newcommand{\ur}{{\underline r}}
\newcommand{\us}{{\underline s}}
\newcommand{\ut}{{\underline t}}
\newcommand{\uu}{{\underline u}}
\newcommand{\uv}{{\underline v}}
\newcommand{\uw}{{\underline w}}
\newcommand{\ux}{{\underline x}}
\newcommand{\uy}{{\underline y}}
\newcommand{\uz}{{\underline z}}
\newcommand{\uA}{{\underline A}}
\newcommand{\uB}{{\underline B}}
\newcommand{\uC}{{\underline C}}
\newcommand{\uD}{{\underline D}}
\newcommand{\uE}{{\underline E}}
\newcommand{\uF}{{\underline F}}
\newcommand{\uG}{{\underline G}}
\newcommand{\uH}{{\underline H}}
\newcommand{\uI}{{\underline I}}
\newcommand{\uJ}{{\underline J}}
\newcommand{\uK}{{\underline K}}
\newcommand{\uL}{{\underline L}}
\newcommand{\uM}{{\underline M}}
\newcommand{\uN}{{\underline N}}
\newcommand{\uO}{{\underline O}}
\newcommand{\uP}{{\underline P}}
\newcommand{\uQ}{{\underline Q}}
\newcommand{\uR}{{\underline R}}
\newcommand{\uS}{{\underline S}}
\newcommand{\uT}{{\underline T}}
\newcommand{\uU}{{\underline U}}
\newcommand{\uV}{{\underline V}}
\newcommand{\uW}{{\underline W}}
\newcommand{\uX}{{\underline X}}
\newcommand{\uY}{{\underline Y}}
\newcommand{\uZ}{{\underline Z}}

% BOLD
\newcommand{\ba}{{\bf a}}
\newcommand{\bc}{{\bf c}}
\newcommand{\bd}{{\bf d}}
\newcommand{\be}{{\bf e}}
\newcommand{\bff}{{\bf f}}
\newcommand{\bg}{{\bf g}}
\newcommand{\bh}{{\bf h}}
\newcommand{\bi}{{\bf i}}
\newcommand{\bj}{{\bf j}}
\newcommand{\bk}{{\bf k}}
\newcommand{\bl}{{\bf l}}
\newcommand{\bm}{{\bf m}}
\newcommand{\bn}{{\bf n}}
\newcommand{\bo}{{\bf o}}
\newcommand{\bp}{{\bf p}}
\newcommand{\bq}{{\bf q}}
\newcommand{\br}{{\bf r}}
\newcommand{\bs}{{\bf s}}
\newcommand{\bt}{{\bf t}}
\newcommand{\bu}{{\bf u}}
\newcommand{\bv}{{\bf v}}
\newcommand{\bw}{{\bf w}}
\newcommand{\bx}{{\bf x}}
\newcommand{\by}{{\bf y}}
\newcommand{\bz}{{\bf z}}
\newcommand{\bA}{{\bf A}}
\newcommand{\bB}{{\bf B}}
\newcommand{\bC}{{\bf C}}
\newcommand{\bD}{{\bf D}}
\newcommand{\bE}{{\bf E}}
\newcommand{\bF}{{\bf F}}
\newcommand{\bG}{{\bf G}}
\newcommand{\bH}{{\bf H}}
\newcommand{\bI}{{\bf I}}
\newcommand{\bJ}{{\bf J}}
\newcommand{\bK}{{\bf K}}
\newcommand{\bL}{{\bf L}}
\newcommand{\bM}{{\bf M}}
\newcommand{\bN}{{\bf N}}
\newcommand{\bO}{{\bf O}}
\newcommand{\bP}{{\bf P}}
\newcommand{\bQ}{{\bf Q}}
\newcommand{\bR}{{\bf R}}
\newcommand{\bS}{{\bf S}}
\newcommand{\bT}{{\bf T}}
\newcommand{\bU}{{\bf U}}
\newcommand{\bV}{{\bf V}}
\newcommand{\bW}{{\bf W}}
\newcommand{\bX}{{\bf X}}
\newcommand{\bY}{{\bf Y}}
\newcommand{\bZ}{{\bf Z}}

% CALLIGRAPHIC
\newcommand{\cA}{\mathcal{A}}
\newcommand{\cB}{\mathcal{B}}
\newcommand{\cC}{\mathcal{C}}
\newcommand{\cD}{\mathcal{D}}
\newcommand{\cE}{\mathcal{E}}
\newcommand{\cF}{\mathcal{F}}
\newcommand{\cG}{\mathcal{G}}
\newcommand{\cH}{\mathcal{H}}
\newcommand{\cI}{\mathcal{I}}
\newcommand{\cJ}{\mathcal{J}}
\newcommand{\cK}{\mathcal{K}}
\newcommand{\cL}{\mathcal{L}}
\newcommand{\cM}{\mathcal{M}}
\newcommand{\cN}{\mathcal{N}}
\newcommand{\cO}{\mathcal{O}}
\newcommand{\cP}{\mathcal{P}}
\newcommand{\cQ}{\mathcal{Q}}
\newcommand{\cR}{\mathcal{R}}
\newcommand{\cS}{\mathcal{S}}
\newcommand{\cT}{\mathcal{T}}
\newcommand{\cU}{\mathcal{U}}
\newcommand{\cV}{\mathcal{V}}
\newcommand{\cW}{\mathcal{W}}
\newcommand{\cX}{\mathcal{X}}
\newcommand{\cY}{\mathcal{Y}}
\newcommand{\cZ}{\mathcal{Z}}

% SCRIPT
\newcommand{\sA}{\mathscr{A}}
\newcommand{\sB}{\mathscr{B}}
\newcommand{\sC}{\mathscr{C}}
\newcommand{\sD}{\mathscr{D}}
\newcommand{\sE}{\mathscr{E}}
\newcommand{\sF}{\mathscr{F}}
\newcommand{\sG}{\mathscr{G}}
\newcommand{\sH}{\mathscr{H}}
\newcommand{\sI}{\mathscr{I}}
\newcommand{\sJ}{\mathscr{J}}
\newcommand{\sK}{\mathscr{K}}
\newcommand{\sL}{\mathscr{L}}
\newcommand{\sM}{\mathscr{M}}
\newcommand{\sN}{\mathscr{N}}
\newcommand{\sO}{\mathscr{O}}
\newcommand{\sP}{\mathscr{P}}
\newcommand{\sQ}{\mathscr{Q}}
\newcommand{\sR}{\mathscr{R}}
\newcommand{\sS}{\mathscr{S}}
\newcommand{\sT}{\mathscr{T}}
\newcommand{\sU}{\mathscr{U}}
\newcommand{\sV}{\mathscr{V}}
\newcommand{\sW}{\mathscr{W}}
\newcommand{\sX}{\mathscr{X}}
\newcommand{\sY}{\mathscr{Y}}
\newcommand{\sZ}{\mathscr{Z}}

% FRAKTUR
\newcommand{\fa}{{\mathfrak a}}
\newcommand{\fb}{{\mathfrak b}}
\newcommand{\fc}{{\mathfrak c}}
\newcommand{\fd}{{\mathfrak d}}
\newcommand{\fe}{{\mathfrak e}}
\newcommand{\ff}{{\mathfrak f}}
\newcommand{\fg}{{\mathfrak g}}
\newcommand{\frg}{{\mathfrak g}}
\newcommand{\fh}{{\mathfrak h}}
\newcommand{\fri}{{\mathfrak i}}
\newcommand{\fj}{{\mathfrak j}}
\newcommand{\fk}{{\mathfrak k}}
\newcommand{\fl}{{\mathfrak l}}
\newcommand{\fm}{{\mathfrak m}}
\newcommand{\fn}{{\mathfrak n}}
\newcommand{\fo}{{\mathfrak o}}
\newcommand{\fp}{{\mathfrak p}}
\newcommand{\fq}{{\mathfrak q}}
\newcommand{\fr}{{\mathfrak r}}
\newcommand{\fs}{{\mathfrak s}}
\newcommand{\ft}{{\mathfrak t}}
\newcommand{\fu}{{\mathfrak u}}
\newcommand{\fv}{{\mathfrak v}}
\newcommand{\fw}{{\mathfrak w}}
\newcommand{\fx}{{\mathfrak x}}
\newcommand{\fy}{{\mathfrak y}}
\newcommand{\fz}{{\mathfrak z}}
\newcommand{\fA}{{\mathfrak A}}
\newcommand{\fB}{{\mathfrak B}}
\newcommand{\fC}{{\mathfrak C}}
\newcommand{\fD}{{\mathfrak D}}
\newcommand{\fE}{{\mathfrak E}}
\newcommand{\fF}{{\mathfrak F}}
\newcommand{\fG}{{\mathfrak G}}
\newcommand{\fH}{{\mathfrak H}}
\newcommand{\fI}{{\mathfrak I}}
\newcommand{\rII}{{\rm II}}
\newcommand{\rIII}{{\rm III}}
\newcommand{\rIV}{{\rm IV}}
\newcommand{\fJ}{{\mathfrak J}}
\newcommand{\fK}{{\mathfrak K}}
\newcommand{\fL}{{\mathfrak L}}
\newcommand{\fM}{{\mathfrak M}}
\newcommand{\fN}{{\mathfrak N}}
\newcommand{\fO}{{\mathfrak O}}
\newcommand{\fP}{{\mathfrak P}}
\newcommand{\fQ}{{\mathfrak Q}}
\newcommand{\fR}{{\mathfrak R}}
\newcommand{\fS}{{\mathfrak S}}
\newcommand{\fT}{{\mathfrak T}}
\newcommand{\fU}{{\mathfrak U}}
\newcommand{\fV}{{\mathfrak V}}
\newcommand{\fW}{{\mathfrak W}}
\newcommand{\fX}{{\mathfrak X}}
\newcommand{\fY}{{\mathfrak Y}}
\newcommand{\fZ}{{\mathfrak Z}}

%< NUMBERS >====================================================================
\newcommand{\K}{\mathbb{K}}
\newcommand{\N}{\mathbb{N}}
\newcommand{\Z}{\mathbb{Z}}
\newcommand{\Q}{\mathbb{Q}}
\newcommand{\R}{\mathbb{R}}
\newcommand{\C}{\mathbb{C}}
\renewcommand{\H}{\mathbb{H}}
\newcommand{\Oc}{\mathbb{O}}

%< Slashed >=============================
\newcommand{\corte}{\slashed}
\newcommand{\Sc}{\corte{S}}
\newcommand{\Scc}{\Sc_{\C}}
\newcommand{\Scp}{\Sc^{+}}
\newcommand{\Scn}{\Sc^{-}}
\newcommand{\Dc}{\corte{D}}
\newcommand{\Dca}{\Dc_{A}}

%< MISC >=======================================================================
\newcommand{\mfrak}{\mathfrak}

\newcommand{\incl}{\hookrightarrow}
\newcommand{\lto}{\longrightarrow}

\newcommand{\homot}{\simeq}
\newcommand{\homott}{\sim}
\newcommand{\homol}{\sim}
\newcommand{\defeq}{\coloneqq}
\newcommand{\iso}{\cong}
\newcommand{\tensor}{\otimes}
\newcommand{\hodge}{{*}}

\newcommand{\chapeu}{\widehat}
\newcommand{\fecho}{\overline}
\newcommand{\til}{\Tilde}
\newcommand{\wtil}{\widetilde}
\newcommand{\interior}{\mathring}
\newcommand{\barra}[1]{\underline #1}

\newcommand{\menos}{\setminus}
\newcommand{\cont}{\subset}
\newcommand{\bordo}{\partial}

\newcommand{\Ad}{\operatorname{Ad}}

\newcommand{\Span}{\mathrm{span}}
\newcommand{\spanR}{\Span_{\R}}
\newcommand{\spanC}{\Span_{\C}}
\newcommand{\spanK}{\Span_{\K}}
\newcommand{\spanQ}{\Span_{\Q}}
\newcommand{\supp}{\mathop{\mathrm{supp}}}
\newcommand{\codim}{\mathop{\mathrm{codim}}}
\newcommand{\coker}{\mathop{\mathrm{coker}}}
\newcommand{\Ker}{\mathop{\mathrm{Ker}}}
\newcommand{\Ima}{\mathop{\mathrm{Im}}}
\newcommand{\End}{\mbox{End}}
\newcommand{\Hom}{\mbox{Hom}}
\newcommand{\Aut}{\mathrm{Aut}}
\newcommand{\Fred}{\mathrm{Fred}}
\newcommand{\ind}{\mathop{\mathrm{index}}}
\newcommand{\Lie}{\mathrm{Lie}}

\newcommand{\norm[1]}{||#1||}
\newcommand{\pts}[1]{\left( #1 \right)} % parêntesis com \left \right
\newcommand{\Mod}[1]{\ (\mathrm{mod}\ #1)}
\newcommand{\floor}[1]{\left\lfloor #1 \right\rfloor}
\newcommand{\pdi}[1]{\left\langle #1 \right\rangle} %produto interno sem duas entradas
\newcommand{\bracl}[1]{ [\![ #1 ]\!]}
\newcommand{\prodint}[2]{\left\langle #1,#2\right\rangle}
\newcommand{\Reprodint}[2]{\operatorname{Re}\left\langle #1,#2\right\rangle}

\newcommand{\nablac}{\nabla^{*}\nabla}
\newcommand{\grad}{\mbox{grad}}
\newcommand{\divv}{\mbox{div}}
\newcommand{\del}{\partial}
\newcommand{\dLie}{\mathcal{L}}

\newcommand{\vphi}{\varphi}
\newcommand{\vepsilon}{\varepsilon}
\newcommand{\vtheta}{\vartheta}

\newcommand{\PP}{\mathbb{P}}
\newcommand{\Ccl}{\mathbb{C}\ell}
\newcommand{\Ricc}{\mbox{Ricc}}
\newcommand{\ric}{\mbox{ric}}
\newcommand{\tr}{\text{tr}}
\newcommand{\Ss}{\mathbb{S}}
\newcommand{\Rr}{\mathfrak{R}}
\newcommand{\Scal}{\mathrm{Scal}}

\newcommand{\stab}{\rmm{Stab}}
\newcommand{\Stab}{\stab}
\newcommand{\G}{\mcal{G}}
\newcommand{\M}{\mcal{M}}
\newcommand{\Mirr}{\M_{irr}}
\newcommand{\B}{\mcal{B}}
\newcommand{\Birr}{\mcal{B}_{irr}}
\newcommand{\SW}{\textit{\textbf{SW}}}
\newcommand{\Cirr}{\mcal{C}_{irr}}
\newcommand{\Zirr}{\mcal{Z}_{irr}}
\newcommand{\rfrak}{\mfrak{r}}
\newcommand{\rfrakk}{\mfrak{r}^{*}}
\newcommand{\PSW}{\mcal{P}\SW }
\newcommand{\perturb}{W^{3,2}(i\Lambda^{2}_{+}T^{*}M)}
\newcommand{\SWM}{\textit{\textbf{sw}}_{M}}
\newcommand{\starseven}{\overset{\scriptscriptstyle 7}{*}}
\newcommand{\stareight}{\overset{\scriptscriptstyle 8}{*}}

\newcommand{\qeq}{\quad\text{e}\quad}
\newcommand{\qandq}{\quad\text{and}\quad}
\newcommand{\qwithq}{\quad\text{with}\quad}
\newcommand{\qforq}{\quad\text{for}\quad}
\newcommand{\qwhereq}{\quad\text{where}\quad}

\newcommand{\Hol}{\mathrm{Hol}}

%< LIE ALGEBRAS AND LIE GROUPS >================================================
\newcommand{\su}{\mathfrak{su}}
\renewcommand{\so}{\mathfrak{so}}
\newcommand{\spinc}{\mathfrak{spin}^{c}}
\newcommand{\gl}{\mathfrak{gl}}
\newcommand{\fsl}{\mathfrak{sl}}

\newcommand{\Spinc}{\mathrm{Spin}^{c}}
\newcommand{\pin}{\mathrm{Pin}}
%\newcommand{\SU}{\mathrm{SU}}
%\newcommand{\rG}{\mathrm{G}}
%\newcommand{\Sp}{\mathrm{Sp}}
%\newcommand{\SL}{\mathrm{SL}}
%\newcommand{\U}{{\mathrm U}}

% Derivatives =========================================================================================
\newcommand{\delum}{\dfrac{\partial}{\partial x^1}}
\newcommand{\deldois}{\dfrac{\partial}{\partial x^2}}
\newcommand{\delj}{\dfrac{\partial}{\partial x^j}}
\newcommand{\deli}{\dfrac{\partial}{\partial x^i}}
\newcommand{\delk}{\dfrac{\partial}{\partial x^k}}
\newcommand{\deln}{\dfrac{\partial}{\partial x^n}}
\newcommand{\delyum}{\dfrac{\partial}{\partial y^1}}
\newcommand{\delydois}{\dfrac{\partial}{\partial y^2}}
\newcommand{\delyj}{\dfrac{\partial}{\partial y^j}}
\newcommand{\delyi}{\dfrac{\partial}{\partial y^i}}
\newcommand{\delyk}{\dfrac{\partial}{\partial y^k}}
\newcommand{\delyn}{\dfrac{\partial}{\partial y^n}}
\renewcommand{\indent}{\hspace{0.3cm}}
\renewcommand{\cos}{\operatorname{cos}}
\renewcommand{\sin}{\operatorname{sen}}
\newcommand{\ddt}[1]{\left.\frac{d}{dt}\left(#1\right)\right\rvert_{t=0}}
\newcommand{\dddt}[1]{\left.\frac{d^2}{dt^2}\left(#1\right)\right\rvert_{t=0}}
\newcommand{\ddto}[1]{\left.\frac{d}{dt}\left(#1\right)\right\rvert_{t=s}}
\newcommand{\pd}[3]{\left.\frac{\partial #1}{\partial #2}\right\rvert_{#3}}
\newcommand{\pdd}[2]{\frac{\partial #1}{\partial #2}}
\newcommand{\ddts}[1]{\left.\frac{d}{dt}#1\right\rvert_{t=0}}

%todos
\newcommand{\todoinagn}[1]{\todo[inline]{\color{blue}AGN: #1}}
\newcommand{\todoinber}[1]{\todo[inline]{\color{red} BER: #1}}
\newcommand{\todoinhqs}[1]{\todo[inline]{\color{green}HQS: #1}}
\newcommand{\todoagn}[1]{\todo{\color{blue} AGN: #1}}
\newcommand{\todober}[1]{\todo{\color{red} BER: #1}}
\newcommand{\todohqs}[1]{\todo{\color{green} HQS: #1}}

    %%%%%%ENUMERAÇÕES

    \newtheorem{theorem}{Theorem}[section]
    \newtheorem{theoremA}{Theorem}
    \renewcommand{\thetheoremA}{\Alph{theoremA}} 
    \newtheorem*{theorem*}{Theorem}
    \newtheorem*{proposition*}{Proposition}
    \newtheorem{lemma}[theorem]{Lemma}
    \newtheorem{proposition}[theorem]{Proposition}
    \newtheorem{corollary}[theorem]{Corollary}
	
    \theoremstyle{definition} %%sem italico
    
    \newtheorem{definition}[theorem]{Definition}

    \newtheorem{example2}[theorem]{Example}
        \newenvironment{example}{\begin{example2}}{\qedtriangle \end{example2}}
        
    \theoremstyle{remark}
    
    \newtheorem{remark2}[theorem]{Remark}
        \newenvironment{remark}{\begin{remark2}}
    {$\hfill\bigcirc $ \end{remark2}\smallskip}

\maketitle
\begin{abstract}
We investigate geometric structures defined by 4-forms, encompassing important classes of structures with torsion, including almost Hermitian, $\mathrm{G}_2$, $\mathrm{Spin}(7)$, almost quaternion-Hermitian, and almost hyper-Hermitian structures. Within this unified framework, the underlying 4-form naturally determines geometric data that define a string algebroid over the manifold. In particular, it gives rise to a canonical notion of instanton, thereby linking these geometries to gauge-theoretic structures in generalized geometry and heterotic string theory. We further prove that these data induce generalized Ricci-flat metrics. Finally, we examine the main examples in the literature and construct several explicit ones.

\end{abstract}

\tableofcontents{}

\newpage

\phantomsection
\addcontentsline{toc}{section}{Introduction}
\section*{Introduction}

The development of gauge theory beyond four dimensions was strongly motivated by theoretical physics \cite{Strominger1986,Hull1986}. A recurring feature is the study of instantons on $n$-manifolds endowed with a distinguished $(n-4)$-form \cite{Donaldson1998,Tian2000} reflecting the symmetries of a $G$-structure \cite{Carrion1998}, or equivalently with its Hodge-dual 4-form. Heterotic supergravity provides a source of gauge-theoretic equations \cite{Garcia-Fernandez2014,delaOssa2014} governing compactifications of the form $\mathbb R^{1,9-n}\times M^n$, where $\mathbb R^{1,9-n}$ is a Lorentzian spacetime and $M^n$ is a compact Riemannian spin manifold encoding the extra dimensions of a supersymmetric vacuum \cite{Gauntlett2004}. In many cases, $M^n$ carries a special structure associated with $\SU(m)$, $\Sp(k)$, $\rG_2$, or $\Spin(7)$ \cite{Clarke2022}.

For $\SU(3)$-structures, the relevant equations form the \emph{Hull--Strominger system}, a non-Kähler extension of Calabi--Yau geometry whose mathematical study was motivated in part by Yau's interpretation of it as a generalisation of the Calabi problem \cite{Fu2008,LiYau2006}. Its seven-dimensional analogue, the \emph{heterotic $\rG_2$-system}, goes back to Friedrich and Ivanov \cite{Friedrich2003b,Friedrich2003} and was subsequently developed in \cite{Fernandez2011,delaOssa2018a,Clarke2022,Lotay2023,deLazari2025}. Studies of heterotic compactifications and their infinitesimal moduli spaces by de la Ossa et al. \cite{delaOssa2014,delaOssa2016,delaOssa2018a,delaOssa2018} inspired the notions of coupled $\SU(m)$- and $\rG_2$-instantons introduced in \cite{Garcia-Fernandez2025} and \cite{daSilva2024a}, respectively. The latter work develops a general theory of coupled $G$-instantons using generalized geometry \cite{Gualtieri2003}, encompassing both systems and relating their Killing spinor equations to generalized Ricci flatness on \emph{string algebroids}. This raises the question of how these results extend to other dimensions and to $\Sp(k)$- and $\Spin(7)$-structures.

Many of the findings in \cite{daSilva2024a} concerning $\rG_2$-geometry depend only on the role played by the coassociative 4-form. This motivates the study of $G$-structures carrying a distinguished invariant 4-form, which we call \emph{4-form geometries}. We show that, in this setting, the distinguished 4-form canonically determines a flux 3-form and a Lee form, which are central for expressing the generalized Ricci tensor on string algebroids, while the $G$-structure determines the instanton condition. Our aim is to identify the conditions under which these data yield generalized Ricci-flat metrics on string algebroids.

Our main result states that, under a suitable algebraic condition relating the 4-form to the Lie algebra $\fg\subset\Lambda^2T^*M$, the gravitino equation for the tensors defining the structure implies generalized Ricci flatness on string algebroids. Under the same algebraic hypothesis, when the coupled instanton equations hold, we also express the generalized Ricci tensor as a contraction of the flux with the gravitino defect. For the natural 4-forms considered here, the algebraic hypothesis holds for $\SU(m)$, $\rG_2$, $\Spin(7)$, and $\Sp(k)$, but not for $\U(m)$ or $\Sp(k)\Sp(1)$, a distinction that parallels the classical Berger picture. The framework recovers known generalized Ricci-flatness results for $\SU(m)$-structures and coupled $\SU(m)$-instantons on Hermitian manifolds \cite{Garcia-Fernandez2019,Garcia-Fernandez2025,garciafernandez2024pluriclosedflowhullstrominger}, as well as the $\rG_2$ result of \cite{daSilva2024a}. These results thus follow from a common algebraic mechanism that also applies to $\Spin(7)$- and $\Sp(k)$-geometries.

\paragraph*{4-form geometries}
A \emph{4-form geometry} is a triple $(M,G,\psi)$ consisting of a connected, oriented Riemannian manifold $(M^n,g)$ endowed with a $G$-structure and a non-zero 4-form $\psi\in\Omega^4(M)$ pointwise modelled on a fixed form $\psi_0\in\Lambda^4(\mathbb R^n)^*$, such that
\[
G\subset \stab_{\SO(n)}(\psi_0).
\]
See Definition~\ref{def:4:form:geometry}. The form $\psi$ induces self-adjoint, $G$-equivariant operators obtained by contracting a pair of indices with $\psi$:
\[
\bH_\psi^q:\Omega^q(M)\longrightarrow\Omega^q(M),
\qquad
\bH_\psi^q(\gamma)\defeq\gamma\iprod^2\psi;
\]
the operator $\bH_\psi^3$ then determines a unique form \(
H_\psi\in({\ker\bH_\psi^3})^\perp\subset \Omega^3(M)
\), called the \emph{flux 3-form}, satisfying
\[
\proj^\perp(d^*\psi)=\bH_\psi^3(H_\psi),
\]
where $\proj^\perp$ denotes the orthogonal projection onto $({\ker\bH_\psi^3})^\perp$; see Definition~\ref{def:flux:3:form}. Its contraction with $\psi$ defines the \emph{Lee form}, see Definition~\ref{def:lee:form},
\[
\zeta_\psi\defeq H_\psi\iprod\psi.
\]
The first result identifies $H_\psi$ directly from the irreducible components of $d^*\psi$, and therefore from the intrinsic torsion of the underlying $G$-structure. 

\begin{theoremA}[Theorem~\ref{theorem:flux:3:form:expression:texto}]
\label{theorem:flux:3:form:expression}
Let $(M,G,\psi)$ be a 4-form geometry and consider the decomposition
\[
    \Omega^3(M)=\ker\vet H_\psi^3\oplus \bigoplus_dV_d,
    \]
where $V_d$ is the sum of all irreducible $G$-subbundles of rank $d$. Let
\(
\bh_d\defeq\bH_\psi^3|_{V_d}:V_d\to V_d
\)
and define the $d$-torsion forms by $\tau^d\defeq\proj_d(d^*\psi)$. Then
\[
H_\psi=\sum_d\bh_d^{-1}\pts{\tau^d}.
\]
In particular, suppose that
\(
({\ker\bH_\psi^3})^\perp
=\Omega^3_1\oplus\cdots\oplus\Omega^3_m
\)
with $\Omega^3_j$ pairwise non-isomorphic irreducible $G$-modules. Then
$\bH_\psi^3|_{\Omega^3_j}=a_j\cdot \id_{\Omega^3_j}$ for non-zero constants $a_j$, and
\[
H_\psi=\sum_{j=1}^m\frac1{a_j}\tau^j,
\qquad
\tau^j:=\proj_{\Omega^3_j}(d^*\psi).
\]
\end{theoremA}

Another relevant ingredient is the notion of a $G$-instanton. The $G$-structure determines an associated Lie algebra subbundle $\fg\subset\Lambda^2T^*M$, whose fibres are isomorphic to $\Lie(G)$. A connection $\vtheta$ on a principal $K$-bundle $P\to M$ is called a \emph{$G$-instanton} (Definition~\ref{def:instanton}) if
\[
F_\vtheta\in\Gamma(\fg\tensor\adjointbundleP)
\subset \Omega^2(M,\adjointbundleP).
\]

\paragraph*{String algebroids and Ricci flatness}
Given a 4-form geometry $(M,G,\psi)$, let $P\to M$ be a principal $K$-bundle with a connection $\theta\in\Omega^1(P,\fk)$, and such that $\fk=\Lie(K)$ is endowed with a non-degenerate, symmetric, bi-invariant pairing $\bracks{\cdot,\cdot}_\fk$. If the flux 3-form $H_\psi\in\Omega^3(M)$ and the connection $\theta$ satisfy the heterotic Bianchi identity
\begin{equation*}
\label{eq: hBi4}
\tag{HBI4}
 dH_\psi=\bracks{F_\theta\wedge F_\theta}_\fk,
\end{equation*}
then these data determine a string algebroid as in Example~\ref{example:string:alegbroid}, with underlying vector bundle
\[
E=TM\oplus\adjointbundleP\oplus T^*M.
\]

\medskip
Under the condition $\vet H_\psi^2|_\fg=b\cdot\id_\fg$ for some $b\neq 0$, the metric $g$ and the Lee form $\zeta_\psi$ induce, respectively, a generalized metric $V_+^{G,\psi}$ and a divergence operator $\dv^{G,\psi}$; see \eqref{eq:generalized:metric:transitive} and \eqref{eq:divergence:4:form:geometry}. These two objects determine the generalized Ricci tensors
\[
    \GRic^\pm
    \pp{V_+^{G,\psi},\dv^{G,\psi}}
    \in \Gamma\pp{\pp{V^{G,\psi}_\mp}^*\otimes \pp{V^{G,\psi}_\pm}^*},
\]
as presented in \ref{parag:ricci:tensor:string:algebroids}. Throughout this article, \emph{generalized Ricci flatness} means
\(
\GRic^+=0.
\)
When the 1-form determining the divergence is closed, this is equivalent to the vanishing of both generalized Ricci tensors; cf. Remark~\ref{remark:Gric+=Gric-}.

\bigskip 
To state our main generalized Ricci-flatness result, we impose two compatibility conditions on the induced string algebroid. First, the flux operator must act on the associated Lie algebra subbundle $\fg\subset\Lambda^2T^*M$ by a non-zero scalar $b$:
\[
\bH_\psi^2|_\fg=b\cdot \id_\fg,
\qquad b\neq0.
\]
Second, the tensors defining the $G$-structure must satisfy the \emph{gravitino equation}; see Definition~\ref{def:gravitino:equation}. Writing $\nabla^+$ for the metric connection with skew-symmetric torsion $H_\psi$, this equation for a defining tensor $\xi$ takes the form
\begin{equation*}
D^+_-\xi=0
\quad\Longleftrightarrow\quad
\nabla^+\xi=0
\qandq
F_\theta\diamond\xi=0;
\end{equation*}
see \eqref{par:gravitino:ricci:flatness}. Thus, when the tensors defining the $G$-structure satisfy the gravitino equation, $\nabla^+$ preserves the $G$-structure and $\theta$ is a $G$-instanton. One of our main results concerning generalized Ricci flatness is as follows.

\begin{theoremA}[Theorem~\ref{theorem:ricci:flatness:simple:group:texto}]
\label{theorem:ricci:flatness:simple:group}
Let $E:(P,K;\theta,\pdi{\cdot,\cdot}_\fk)\to(M,G,\psi)$ be the string algebroid induced by a 4-form geometry for which
\[
\bH_\psi^2|_\fg=b\cdot \id_\fg,
\qquad b\neq0.
\]
Suppose that $G$ is the common stabiliser of $\psi$ and a multi-tensor $\xi=(\xi_1,\ldots,\xi_m)$, and that both $\psi$ and $\xi$ satisfy the gravitino equation:
\[
G=\stab_{\SO(n)}(\psi,\xi)
\qandq
D^+_-\psi=0=D^+_-\xi.
\]
Then the induced pair \eqref{eq:generalized:induced:pair} is generalized Ricci-flat:
\[
\GRic^+\pp{V_+^{G,\psi},\dv^{G,\psi}}=0.
\]
\end{theoremA}

In the $\rG_2$ case, this recovers the generalized Ricci flatness theorem proved in \cite[Theorem~4.9]{daSilva2024a}; the point here is that the argument is a consequence of the 4-form formalism and therefore applies uniformly to other geometric structures.

\paragraph*{Coupled instantons}

Let $N(G)$ denote the normaliser of $G$ in $\SO(n)$. The $G$-instanton condition extends to 4-form geometries $(M,N(G),\psi)$; see Remark~\ref{remark:instantons:normaliser}. This allows us, for instance, to consider $\SU(m)$-instantons on manifolds with a $\U(m)$-structure. In this broader setting, a \emph{coupled $G$-instanton} is a string algebroid on which $D^-_+$ is a $G$-instanton; see \ref{par:coupled:instantons}. If a tensor $\xi$ satisfies the gravitino equation, then $F_{D^-_+}\diamond\xi=0$. Consequently, if $G$ is the common stabiliser of a collection of tensors, each satisfying the gravitino equation, then the coupled $G$-instanton equation holds. However, the coupled $G$-instanton condition alone does not guarantee generalized Ricci flatness. The following theorem clarifies this relationship by expressing the generalized Ricci tensor for coupled $G$-instantons as a contraction of the flux 3-form with the gravitino defect.

\begin{theoremA}[Theorem~\ref{theorem:ricci:flatness:coupled:equations:texto}]
\label{theorem:ricci:flatness:coupled:equations}
Let $E:(P,K;\theta,\pdi{\cdot,\cdot}_\fk)\to(M,N(G),\psi)$ be the string algebroid induced by a 4-form geometry for which
\[
\bH_\psi^2|_{\fg}=b\cdot \id_{\fg},
\qquad b\neq0.
\]
If the coupled $G$-instanton equation holds, then
\[
\GRic^+\pp{V_+^{G,\psi},\dv^{G,\psi}}
=-\frac1bH_\psi\iprod D^+_-\psi.
\]
\end{theoremA}

\paragraph*{Special geometries and Berger's list}

In \S\ref{sec:examples:4-form:geometry}, we apply the general theory to the natural 4-form geometries associated with $\U(m)$, $\SU(m)$, $\rG_2$, $\Spin(7)$, $\Sp(k)\Sp(1)$, and $\Sp(k)$. The flux operator acts by a non-zero scalar on $\fg$ for $\SU(m)$, $\rG_2$, $\Spin(7)$, and $\Sp(k)$, so the eigenvalue hypothesis of Theorem~\ref{theorem:ricci:flatness:simple:group} is satisfied. By contrast, the relevant Lie algebra summands have different eigenvalues for $\U(m)$ and $\Sp(k)\Sp(1)$.
This distinction mirrors the classical Berger picture: among the corresponding irreducible torsion-free structures, the groups $\SU(m)$, $\Sp(k)$, $\rG_2$, and $\Spin(7)$ induce Ricci-flat metrics, whereas $\U(m)$ and $\Sp(k)\Sp(1)$ do not impose Ricci flatness in general \cite{Berger1955,Salamon1989}. Thus, in the presence of skew torsion and generalized geometry, the 4-form formalism exhibits an algebraic division that strikingly parallels the torsion-free Berger picture.

The failure of the eigenvalue hypothesis in Theorem~\ref{theorem:ricci:flatness:simple:group} does not rule out generalized Ricci-flat metrics on $\U(m)$- and $\Sp(k)\Sp(1)$-structures. For example, suitable coupled $\SU(m)$-instantons still produce such metrics on $\U(m)$-structures, and an analogous phenomenon occurs for $\Sp(k)\Sp(1)$-structures; see \ref{parag:string:algebroids:u:m} and the discussion in \S\ref{sec:examples:4-form:geometry}.
The principal applications of Theorem~\ref{theorem:ricci:flatness:simple:group} can be summarised as follows.

\begin{theoremA}
\label{theorem:ricci:flatness:compilation}
Let $E:(P,K;\theta,\pdi{\cdot,\cdot}_\fk)\to(M,G,\psi)$ be a string algebroid induced by one of the 4-form geometries below, and let $\nabla^+$ denote the Bismut connection \eqref{eq:bismut:connection} with skew-torsion $H_\psi$. Then the induced generalized pair is generalized Ricci-flat,
\[
\GRic^+\pp{V_+^{G,\psi},\dv^{G,\psi}}=0,
\]
in each of the following cases:

    \begin{itemize}

    \item       (Corollary~\ref{corollary:GRIC:dimensao:quatro}) 
    $M^4$ is a Calabi--Yau manifold, and $\theta$ is an anti-self-dual instanton connection. The 4-form geometry is $(M^4,G=\SU(2),\psi=\vol_M)$; in particular,
\[
H_{\vol_M}=0\qandq\zeta_{\vol_M}=0.
\]
    
    \item (Proposition~\ref{prop:ricci:flatness:U(m)}) $M^{2m}$, for $m\ge 3$, is endowed with a $G=\SU(m)$-structure $(\omega,\Psi)$ whose Nijenhuis tensor $N_\omega\in \Omega^3(M)$ is totally skew-symmetric and for which $\nabla^+\Psi=0$, while $\theta$ is an $\SU(m)$-instanton. The 4-form geometry is $(M^{2m},\SU(m),\psi=\tfrac12\omega^2)$; in particular,
        \[
        H_\omega=-d^c\omega+N_\omega
        =\frac{1}{m-2}\tau_1\iprod\psi +\frac{1}{3}\tau_3^{3,0}-\tau_3^{2,1} 
        \qandq
        \zeta_\omega=Jd^*\omega 
        =-\frac{m-1}{m-2}\tau_1 ,
        \]  
        where $\tau_1,\tau_3^{3,0},\tau_3^{2,1}$ are the torsion forms defined in \eqref{eq:torsion:forms:u:m}.
        
    \item (Proposition~\ref{prop:ricci:flatness:G2}) 
    $M^7$ is endowed with an integrable $G=\rg_2$-structure $\varphi$, i.e., one for which $\tau_2=0$, and $\theta$ is a $\rg_2$-instanton. The 4-form geometry $(M^{7},\rg_2,\psi=\star\varphi)$ is induced by the coassociative 4-form; in particular,
        \[
        H_\varphi=\frac16\tau_0\vphi-\tau_1\iprod\psi-\tau_3\qandq
        \zeta_\varphi=4\tau_1,
        \]
    where $\tau_j\in \Omega^j$ are the torsion forms defined in \eqref{eq:torsion:forms:G2}. 

    \item (Proposition~\ref{prop:ricci:flatness:Spin7})
    $M^8$ is endowed with a $G=\spin(7)$-structure $\Omega$, and $\theta$ is a $\spin(7)$-instanton. The 4-form geometry $(M^{8},\spin(7),\psi=\Omega)$ is induced by the Cayley 4-form; in particular,
        \[
        H_\Omega =-\frac16\tau_1\iprod\Omega-\tau_3\qandq
        \zeta_\Omega = \frac76\tau_1,
        \]
    where $\tau_j\in \Omega^j$ are torsion forms defined in \eqref{eq:torsion:forms:Spin(7)}. 

    \item (Proposition~\ref{prop:ricci:flatness:sp(k)}) 
    $M^{4k}$, for $k\ge 2$, is endowed with a $G=\sp(k)$-structure $(\omega_I,\omega_J,\omega_K)$ satisfying $\nabla^+\omega_I=\nabla^+\omega_J=\nabla^+\omega_K=0$, and $\theta$ is a $\sp(k)$-instanton. The 4-form geometry $(M^{4k},\sp(k),\psi=\kappa)$ is induced by the Kraines 4-form, where $\kappa=\tfrac16(\omega_I^2+\omega_J^2+\omega_K^2)$; in particular,
        \[
        H_\kappa = \frac{3}{2k-2}\tau_1\iprod\kappa +\frac3{2k+1}\tau_3^{8k}-\tau_3^K+\tau_3^G\qandq 
        \zeta_\kappa =-\frac{2k+1}{2k-2}\tau_1,
        \]
    where $\tau_1\in \Omega^1(M)$, $\tau_3^{8k},\tau_3^K,\tau_3^G\in \Omega^3(M)$ are the torsion forms defined in \eqref{eq:torsion:forms:sp:k:sp:1}.
\end{itemize}
\end{theoremA}

Besides these applications, in \ref{parag:su(4):structures} we discuss \(\SU(4)\)-structures equipped with an alternative 4-form, instead of the usual choice \(\psi=\frac12\omega^2\), namely the real part of the complex volume form $\Psi_+$.
We also study in \ref{parag:lifting:G2:to:spin7} a lift of a $\rG_2$-structure to a $\Spin(7)$-structure using the language of 4-form geometry, which can be considered a generalisation of \cite[Theorem~5.1]{Ivanov2005}.

\bigskip
The article is organised as follows. In \S~\ref{section:theory:4:forms}, we develop the general theory of 4-form geometries. We introduce the flux and Lee forms, derive the Yang--Mills and Ricci--Bismut identities that enter the generalized Ricci tensor, recall the required background on string algebroids and generalized geometry, and prove Theorems~\ref{theorem:ricci:flatness:simple:group} and \ref{theorem:ricci:flatness:coupled:equations}. In \S~\ref{sec:examples:4-form:geometry}, we apply the theory to $\U(m)$-, $\SU(m)$-, $\rG_2$-, $\Spin(7)$-, $\Sp(k)\Sp(1)$-, and $\Sp(k)$-structures, compute the corresponding flux operators, flux 3-forms, and Lee forms, and discuss explicit examples and product constructions. 

\paragraph*{Notation and conventions}
Unless otherwise stated, all manifolds are assumed to be connected and oriented. We use the following notation throughout:

\begin{itemize}
    \item A 4-form geometry is denoted by $(M,G,\psi)$. The Riemannian metric $g$ and volume form $\vol_M$ determined by the $G$-structure are usually left implicit; cf. Definition~\ref{def:4:form:geometry}.

    \item A principal $K$-bundle $P\to M$, together with a non-degenerate, bi-invariant, symmetric pairing $\pdi{\cdot,\cdot}_\fk$ on $\fk=\Lie(K)$, is denoted by $(P,K;\pdi{\cdot,\cdot}_\fk)$. When a connection $\theta$ is fixed, we write $(P,K;\theta,\pdi{\cdot,\cdot}_\fk)$; cf. Example~\ref{example:transitive:Courant}.

    \item We denote by
    \(
        E:(P,K;\theta,\pdi{\cdot,\cdot}_\fk)\to(M,H)
    \)
    the string algebroid
    \(
        E=TM\oplus\adjointbundle P\oplus T^*M
    \)
    determined by the connection $\theta$ and the 3-form $H$, subject to the heterotic Bianchi identity \eqref{eq: hBi}; cf. Example~\ref{example:transitive:Courant}. A Riemannian metric $g$ on $M$ induces a generalized metric $V_+$, while $\vol_M$ induces the Riemannian divergence $\dv^{V_+}$; cf. \eqref{eq:riemannian:divergence}.

    \item If the string algebroid $E$ is induced by a 4-form geometry, we write
    \(
        E:(P,K;\theta,\pdi{\cdot,\cdot}_\fk)\to(M,G,\psi).
    \)
    In this case, $H=H_\psi$ is the flux 3-form, and the heterotic Bianchi identity becomes \eqref{eq: hBi4}. When the Lie algebra $\fg$ is an eigenspace of the flux operator with non-zero eigenvalue $b$, the 4-form geometry determines the generalized pair $(V_+^{G,\psi},\dv^{G,\psi})$ defined in \eqref{eq:generalized:induced:pair}; cf. Definition~\ref{def:generalized:pair:4:form}.
\end{itemize}

\paragraph*{Acknowledgements}
The authors thank Artur Blois, Udhav Fowdar, Mario Garcia-Fernandez, Lino Grama, Jason Lotay, Andr\'es Moreno, Paul Schwahn, and Jakob Stein, for helpful discussions. We would like to thank Jeffrey Streets, in particular, for suggesting a number of improvements to the original manuscript.

\noindent\textbf{AdSJ} is supported by the S\~ao Paulo Research Foundation (FAPESP) grant 2024/06658-2.

\noindent\textbf{HSE} is supported by FAPESP grants 2020/09838-0, with BI0S--Brazilian Institute of Data Science;
2021/04065-6, with BRIDGES--Brazil--France Interplays in Gauge Theory, Extremal Structures and Stability;
2024/00923-6, with CBG--Brazilian Centre for Geometry; and by the Brazilian National Council for Scientific and Technological Development (CNPq), through the PQ-A Productivity Grant 307145/2025-5.

\noindent\textbf{BV} is supported by FAPESP grant 2024/09514-1.

\paragraph*{Declaration of AI use}
The authors used OpenAI Codex, powered by GPT-5, for English-language editing, \LaTeX{} and bibliography formatting, compilation diagnostics, and drafting the cover letter. It also helped phrase brief explanations based on mathematical reasoning supplied by the authors, but was not used to originate the principal mathematical ideas or results, perform calculations, or construct or verify proofs. The authors reviewed and approved all AI-assisted changes, independently checked the mathematical content, references, and final source, and take full responsibility for the manuscript's accuracy, originality, and integrity.

\newpage
\section{String algebroids over 4-form geometries}
\label{section:theory:4:forms}
\label{sec1:generalized:ricci:4:forms:theory}

We now formulate in detail the theoretical framework for this text. In \S\ref{subsec1:geometry:4:forms}, we introduce the main objects, namely \emph{4-form geometries}. These are geometric structures, viewed as reductions of the oriented orthogonal frame bundle of a Riemannian manifold from $\SO(n)$ to $G$, endowed with a distinguished 4-form $\psi \in \Omega^4(M)$ whose pointwise stabiliser contains $G$. We then present quantities naturally associated with such structures. First, we introduce the \emph{flux 3-form} $H_\psi \in \Omega^3(M)$, which corresponds to the skew-symmetric component of the intrinsic torsion of the $G$-structure and is closely related to compatible connections with totally skew-symmetric torsion, as extensively studied by Friedrich, Ivanov, Agricola, and others \cite{Friedrich2003b,Agricola2004}. We also consider the \emph{Lee form} $\zeta_\psi \in \Omega^1(M)$, which captures the vector component of the intrinsic torsion, as well as the notion of an \emph{instanton} connection, defined by algebraic restrictions on the curvature analogous to those introduced by Reyes-Carrión \cite{Carrion1998}. This notion is also related to the higher-dimensional instantons of Donaldson and Thomas \cite{Donaldson1998,Donaldson2011}, defined in terms of $(n-4)$-forms.
These quantities play a central role in what follows, since they provide the data used in generalized geometry, as discussed in \S\ref{subsec:generalized:geometry}; see \cite{Garcia-Fernandez2014,Garcia-Fernandez2019}.

In \S\ref{subsec:yang:mills+ricci:bismut}, we derive the equations that later appear in the expression for the generalized Ricci tensor associated with 4-form geometries, namely the \emph{Yang--Mills} term, which measures the failure of an instanton to satisfy the Yang--Mills condition, and the \emph{Ricci--Bismut} term, which measures the failure of the underlying metric to be Ricci-flat. These two quantities are naturally encoded by geometric structures with torsion and, in our setting, are expressed in terms of the distinguished 4-form.

Finally, in \S\ref{sec:ricci:curvature:4:forms}, we apply the theory developed in the previous sections to prove that, under suitable assumptions, Courant algebroids associated with 4-form geometries are generalized Ricci-flat. We also discuss a more general case in which the obstruction to generalized Ricci flatness is governed by the coupled instanton equations introduced in \cite{Garcia-Fernandez2017,Garcia-Fernandez2025,daSilva2024a}.

\subsection{Geometric structures admitting 4-forms}
\label{subsec1:geometry:4:forms}

We follow the conventions of \cite{Fadel2024} on geometric structures; see also \cite{Loubeau2023, Dwivedi2024}. A $G$-structure on a manifold $M^n$ is a reduction of the frame bundle
\[
    \GL(n,\R) \ \cdots \ P_{\GL}(M)\to M
\]
to a principal $G$-subbundle $P_G\to M$. 

In many important cases, such a reduction is equivalent to the existence of a finite collection of tensor fields pointwise modelled on a fixed collection of tensors on $\bb R^n$. More precisely, $G$ is the stabiliser of these model tensors under the natural right action of $\GL(n,\bb R)$ defined by
\begin{align*}
    (v,A)\longmapsto A^{-1}v,
    \qforq
    v\in \bb R^n \qandq
    (\alpha,A)\longmapsto A^*\alpha=\alpha\circ A,
    \qforq
    \alpha\in(\bb R^n)^*.
\end{align*}
Let $\{e_i\}\subset \bb R^n$ be the standard basis and let $\{e^i\}\subset (\bb R^n)^*$ be its dual basis. Denote the coordinates of a $(p,q)$-tensor $
    \xi_{\circ}\in \cT^{p,q}(\bb R^n)$ by
    \[
    \xi_{\circ}
    =
    \xi^{i_1\ldots i_p}{}_{j_1\ldots j_q}\ 
    e_{i_1}\otimes\cdots\otimes e_{i_p}
    \otimes
    e^{j_1}\otimes\cdots\otimes e^{j_q},
\]
where $
    \xi^{i_1\ldots i_p}{}_{j_1\ldots j_q}
    \defeq
    \xi_{\circ}
    (e^{i_1},\ldots,e^{i_p},e_{j_1},\ldots,e_{j_q})
    \in \bb R$, 
and the Einstein summation convention is used throughout. The induced right action of $A\in\GL(n,\bb R)$ on $\xi_{\circ}$, corresponding to arbitrary changes of frames, is given by
\begin{align*}
    \xi_{\circ}\cdot A
    \defeq
    \xi^{i_1\ldots i_p}{}_{j_1\ldots j_q}
    A^{-1}e_{i_1}\otimes\cdots\otimes A^{-1}e_{i_p}
    \otimes
    A^*e^{j_1}\otimes\cdots\otimes A^*e^{j_q}.
\end{align*}
We denote the stabiliser of $\xi_{\circ}$ under this right $\GL(n,\R)$-action by
\[
    \Stab(\xi_{\circ})
    \defeq
    \{A\in\GL(n,\bb R):\xi_{\circ}\cdot A=\xi_{\circ}\}.
\]
More generally, if $\xi_{\circ}=\pp{(\xi_{\circ})_1,\ldots,(\xi_{\circ})_k}$ is a finite collection of tensors, then $\GL(n,\R)$ acts componentwise, so that
\[
    \Stab(\xi_{\circ})
    =
    \bigcap_{i=1}^k \Stab\big((\xi_{\circ})_i\big).
\]
For our purposes, we assume throughout that $(M^n,g)$ is an oriented Riemannian manifold; equivalently, $M$ is endowed with an $\SO(n)$-structure. We consider only closed Lie subgroups $G\subseteq \SO(n)$ and, consequently, only the induced $\SO(n)$-action on tensors $\xi_0$ and their $\SO(n)$-stabilisers, which we denote by $\Stab_{\SO(n)}(\xi_0)$. For simplicity, we write $(M,G)$ for a manifold $M$ endowed with a $G$-structure and often refer to this pair simply as a \emph{$G$-structure}.

\bigskip

Throughout the text, we use the diamond operator $\diamond$, which denotes the infinitesimal action induced by the right $\GL(n,\R)$-action on tensors \cite{Fadel2024}. For $\xi\in \Omega^k(M)$, the action of $\fg\fl(n)$ is given by
\[
\pp{\alpha\otimes X}\diamond\xi =
\alpha\wedge i_X\xi
\qforq
\alpha\in\Omega^1(M)
\qandq
X\in\fX(M),
\]
which can be written in coordinates as
\( 
A\diamond \xi = A^j\wedge\xi_j
= (e^j\iprod A) \wedge (e_j\iprod\xi)\) for an orthonormal frame.

\paragraph{4-form geometries} Many special geometric structures on manifolds naturally give rise to associated 4-forms, including almost Hermitian, almost special Hermitian, almost hyper-Hermitian, almost quaternionic Hermitian, almost contact metric, \(\rG_2\), and \(\Spin(7)\) structures. A distinguished 4-form $\psi$ allows us to define a 3-form $H=H_\psi$ and a notion of instanton connection. Together, these provide the ingredients for a string algebroid structure whenever the heterotic Bianchi identity is satisfied.

\begin{definition}[4-form geometry]
\label{def:4:form:geometry}
    Let $(M,g)$ be a connected, oriented Riemannian manifold, and let $G$ be a closed Lie subgroup of $\SO(n)$. A \emph{4-form geometry} is a triple $(M,G,\psi)$ consisting of a $G$-structure on $M$ and a non-zero 4-form $\psi\in \Omega^4(M)$ pointwise modelled on $\psi_0\in\Lambda^4\pts{\R^n}^*$, whose stabiliser satisfies
    \begin{equation*}
        G\subset \mathrm{Stab}_{ \SO(n)}(\psi_0).
    \end{equation*}
\end{definition}

If $G$ is the stabiliser $\stab_{\SO(n)}(\psi)$, or the identity component of $\Stab_{\SO(n)}(\psi)$, then
\begin{equation*}
    \mathfrak g=\rmm{Lie}(G) 
    \iso \Lie(\Stab_{\SO(n)}(\psi))
    \subset \mathfrak{so}(n)\isomorphic \Lambda^2 .
\end{equation*}
Since $\psi$ has, pointwise, the same stabiliser as $\psi_0$, we can use \cite[Lemma 1.7]{Fadel2024} to conclude that
\begin{equation}
\label{eq:ker:diamond:=:lie:algebra}
    \ker ({}\cdot {}\diamond \psi)=
    \Lg.
    \end{equation}
In this case, a 2-form $\beta\in \Omega^2(M)$ belongs to $\Lg$ if and only if $\beta\diamond\psi=0$. In general, when $G\subset \stab_{\SO(n)}(\psi)$, we have $\fg\subset \ker\pts{\cdot\diamond\psi}$; that is, $\beta\in \mathfrak g$ implies $\beta\diamond\psi=0$, but the converse need not hold.

\paragraph{The flux operator and the flux 3-form}
\label{parag:flux:3:form}
On a 4-form geometry $(M, G, \psi)$, the structural form $\psi \in \Omega^4(M)$ induces invariant maps that encode significant geometric information about the manifold. In particular, one such map is the \emph{flux operator} $\bH_\psi$, which is used to define the \emph{flux 3-form} $H_\psi$ associated with the 4-form geometry. The latter plays a fundamental role in the study of string algebroids in \S\ref{sec:ricci:curvature:4:forms}. In this subsection, we introduce these constructions and prove Theorem~\ref{theorem:flux:3:form:expression}, which provides an expression for the flux 3-form.

\begin{definition}[Flux operator]
\label{def:flux:operator}
    Let $(M^n, G, \psi)$ be a 4-form geometry. Then there are $G$-equivariant maps $\vet{H}^q_\psi \colon \Omega^q(M) \to \Omega^q(M)$ defined by
\begin{equation*}
\vet{H}^q_\psi(\xi) = \xi \iprod^2 \psi 
\coloneqq\frac{1}{2}\sum_{i,j=1}^n\xi^{ij}\wedge\psi_{ij}
= \frac{1}{2!(q-2)!2!} \xi^{\mu\nu}{}_{i_1 \cdots i_{q-2}} \psi_{\mu\nu i_{q-1} i_q} e^{i_1 i_2 \cdots i_q},
\end{equation*}
    where $\{e_j\}$ is a local orthonormal frame of $TM$.
    For $q = 3$, we simply write $\vet{H}^3_\psi = \vet{H}_\psi$ and call it the \emph{flux operator}.
\end{definition}

In the next lemma, we collect some properties of the operators $\bH_\psi^q$. 

\begin{lemma}
    The operators $\bH_\psi^q$ have constant rank and are self-adjoint. In particular, $\ker \bH_\psi^q$ and $\Ima\bH_\psi^q$ are vector subbundles of $\Omega^q(M)$, and each $\bH_\psi^q$ is diagonalisable with real eigenvalues. Moreover,
    \begin{equation*}
        \Ima\bH_\psi^q = \pts{\ker\bH_\psi^q}^\perp
        \qandq
        \Omega^q(M)\iso \ker\bH_\psi^q\oplus\pts{\ker\bH_\psi^q}^\perp
        \iso \ker\bH_\psi^q\oplus\Ima\bH_\psi^q .
    \end{equation*}
\end{lemma}

\begin{proof}
    The fact that each $\bH_\psi^q$ has constant rank follows from its algebraic definition and the fact that $\psi$ is pointwise modelled on $\psi_0\in\Lambda^4(\R^n)^*$.

    To prove that $\bH_\psi^q$ is self-adjoint, let $\alpha,\beta \in \Omega^q(M)$. Then
\[
\bracks{\vet H^q_\psi \alpha,\beta}
=
\frac{1}{2!(q-2)!2!} \alpha_{ijK} \psi_{ijab} \beta_{Kab}
=
\bracks{\alpha,\vet H^q_\psi\beta}.
\]
Hence, $\Ima\bH^q_\psi\subset \pts{\ker\bH^q_\psi}^\perp$. From this, and since these subbundles have fibres of the same dimension, it follows that $\Ima\bH^q_\psi=(\ker\bH^q_\psi)^\perp$.
\end{proof}

\begin{definition}[Flux 3-form]\label{def:flux:3:form}
On a 4-form geometry $(M, G, \psi)$, the associated \emph{flux 3-form} $H_\psi \in \Omega^3(M)$ is the unique 3-form in $(\ker \vet{H}_\psi)^\perp$ such that
\begin{equation*}
\label{eq:def:flux:3:form}
\proj^\perp(d^* \psi) = \vet{H}_\psi(H_\psi) = H_\psi \iprod^2 \psi, 
\end{equation*}
where $\proj^\perp \colon \Omega^3(M) \to (\ker \vet{H}_\psi)^\perp$ is the orthogonal projection. The \emph{Bismut connection} $\nabla^+$ (respectively, the \emph{Hull connection} $\nabla^-$) is the metric connection on $TM$ with totally skew-symmetric torsion $H_\psi\in \Omega^3(M)$ (respectively, $-H_\psi$): 
\begin{equation*}
    \label{eq:bismut:connection}
    \tag{BisC}
        \nabla^\pm= \nabla^g\pm\frac12g^{-1}H_\psi.
    \end{equation*}
\end{definition}

\begin{remark}[Parallel 4-form] 
\label{remark:parallelism:d:psi:has:no:kernel:intersection}
A particular situation in which  $d^*\psi$ has no component in $\ker \vet H_\psi$ occurs when there exists a metric connection $\nabla = \nabla^g + \frac{1}{2}g^{-1}T$ with totally skew-symmetric torsion $T\in \Omega^3(M)$ satisfying $\nabla\psi=0$. In this case, one can check that $$d^*\psi=T\iprod^2\psi=\vet H_\psi(T).$$ 
Moreover, even when $\nabla^+\psi\neq 0$, in some of our examples we still find that $d^*\psi$ has no component in
$\ker \vet H_\psi$ for purely algebraic reasons stemming from the structure group $G$.
An instance is provided by $\rG_2$-structures, for which the identity
$d^*\psi=H\iprod^2\psi$
continues to hold even when $\nabla^+\psi\neq0$; see \cite[\S 2.1.2]{delaOssa2018a}.
\end{remark}

\bigskip

     From the definition of the flux 3-form $H_\psi$, obtaining an explicit expression for it requires inverting the restriction of the flux operator $\bH_\psi$ to $(\ker\bH_\psi)^\perp$. To this end, we first decompose $(\ker\bH_\psi)^\perp$ into irreducible $G$-submodules. We then group together all irreducible submodules of the same rank $d$, regardless of whether they are isomorphic, into a $G$-submodule $V_d$. Thus,
\[
(\ker\vet H_\psi)^\perp
=
\bigoplus_{d}V_d,
\]
where each $V_d$ is generally reducible. Since $V_d$ contains every irreducible component of rank $d$, including all copies of each isomorphism class, it is invariant under the flux operator. Hence, for each $d$, the restriction
\[
\bh_d\defeq \bH_\psi|_{V_d}:V_d\to V_d
\]
is a well-defined isomorphism. Let $
\proj_d:\Omega^3(M)\to V_d$ denote the orthogonal projection.

\begin{theorem}[Theorem~\ref{theorem:flux:3:form:expression}]
\label{theorem:flux:3:form:expression:texto}
    Let $(M, G,\psi)$ be a 4-form geometry and let $\tau^d \defeq \proj_d(d^*\psi)$, which we call the $d$-\emph{torsion form}. 
    Then the flux 3-form is given by
    \begin{equation*}
        H_\psi = \sum_{d}\bh_d^{-1}\pts{\tau^d} .
    \end{equation*}
    In particular, suppose that no two distinct irreducible submodules are isomorphic in the decomposition of $(\ker\vet H_\psi)^\perp$, so that
    \begin{equation*}
        (\ker\vet H_\psi^3)^\perp = \bigoplus_{j=1}^m\Omega^3_j ,
    \end{equation*}
    with $\Omega^3_j$ pairwise non-isomorphic irreducible $G$-modules. Then $\vet H_\psi^3|_{\Omega^3_j}=a_j\cdot \id_{\Omega^3_j}$ for non-zero constants $a_j\in \bb R$, so that
    \begin{equation*}
        H_\psi = \sum_{j=1}^m \dfrac{1}{a_j}\tau^j, 
        \quad \textrm{where} \;
        \tau^j=\proj_{\Omega^3_j}(d^*\psi) .
    \end{equation*}
\end{theorem}
\begin{proof}
Let $\proj_K:\Omega^3(M)\to\ker\bH_\psi$ be the orthogonal projection onto the kernel.
As each $V_d$ is $\vet H_\psi$-invariant, we have, for $\gamma\in \Omega^3(M)$,
\begin{align*}
    &\bH_\psi(\gamma)
    = \bH_\psi\pts{\proj_K(\gamma) + \sum_{d}\proj_d(\gamma)} 
    = \sum_{d}\bH_\psi\pts{\proj_d(\gamma)} 
    = \sum_{d} \bh_d(\proj_d(\gamma))
    \\
    \Rightarrow \quad
    &\proj_d(\gamma)=\vet h_d^{-1}\pp{\proj_d(\vet H_\psi(\gamma))}\\
    \Rightarrow \quad
    &\proj^\perp(\gamma) =\sum_{d}\vet h_d^{-1}\pp{\proj_d(\vet H_\psi(\gamma))}.
\end{align*}
Taking $\gamma=H_\psi$, using the definition $\proj^\perp(d^*\psi)=\vet H_\psi(H_\psi)$, and noting that $\proj_d\circ\proj^\perp=\proj_d$, we obtain
\begin{align*}
H_\psi&=\proj^\perp(H_\psi)=\sum_{d} \vet h_d^{-1}(\proj_d(\vet H_\psi(H_\psi)))
=\sum_{d} \vet h_d^{-1}(\proj_d(\proj^\perp(d^*\psi)))
\\&=
\sum_{d} \vet h_d^{-1}(\proj_d(d^*\psi)) .\qedhere
\end{align*}
\end{proof}

Although the decomposition of $\Omega^3$ contains no repeated isomorphism classes in most of our examples, such repetitions do occur in some cases, as illustrated by the alternative description of the flux operator for $\SU(4)$-structures in \ref{parag:su(4):structures}.

\begin{remark}[Flux 3-form and intrinsic torsion]
\label{remark:intrinsic:torsion}
If a $G$-structure is defined by a multi-tensor $\xi=(\xi_1,\dots,\xi_k)$, then its intrinsic torsion $\Gamma\in \Omega^1(M, \mathfrak g^\perp)\subset \Omega^1\otimes \Omega^2$ encodes the first-order variation of these tensors through the diamond action $\nabla^g_X\xi=:\Gamma(X)\diamond\xi$; see \cite[\S 1.1]{Fadel2024} and \cite[\S 2]{Friedrich2002}. We define the \emph{minimal} or \emph{characteristic connection}\footnote{The term \emph{characteristic} is typically used for a compatible connection with totally skew-symmetric torsion.} by
\begin{equation*}
\nabla^0_X=\nabla^g_X-\Gamma(X) \diamond{}.
\end{equation*}
This connection is $G$-compatible by construction: $\nabla^0\xi=0$ and, consequently, $\Hol(\nabla^0)\subset G$. When the intrinsic torsion $\Gamma\in \Omega^3(M)$ is totally skew-symmetric, the minimal connection also has totally skew-symmetric torsion $T_0=-2\Gamma$.
    The conditions for the existence of a $G$-compatible metric connection $\nabla$ with totally skew-symmetric torsion were extensively studied in  \cite{Friedrich2002,Friedrich2003b}. 
    
    For a 4-form geometry $(M,G,\psi)$, if the $\stab_{\SO(n)}(\psi)$-intrinsic torsion $\Gamma\in \Omega^1\otimes \mathfrak{stab}(\psi)^\perp$ is totally skew-symmetric, then $-2\Gamma=T_0=H_\psi$. Indeed,
\begin{align*}
    d^*\psi &= -e_j\iprod\nabla^g_{e_j}\psi=-e_j\iprod(\Gamma_j\diamond \psi)
    =-e_j\iprod (\Gamma_{jk}\wedge \psi_k)
    \\&= \cancel{\Gamma_{jjk}}\wedge\psi_k
    +\Gamma_{jk}\wedge \psi_{kj}
    =
    -\Gamma_{jk}\wedge \psi_{jk}
    =-2 (\Gamma\iprod^2\psi)\\
    &=-2\cdot \vet H_\psi^3(\Gamma),
    \end{align*}
    and the conclusion follows from the definition of $H_\psi$ in Definition~\ref{def:flux:3:form}.

Moreover, if the $\Stab_{\SO(n)}(\psi)$-intrinsic torsion is totally skew-symmetric, then there is at least one connection with totally skew-symmetric torsion with respect to which $\psi$ is parallel, namely the minimal connection $\nabla^0$. 
If $\nabla$ is another such connection with torsion $T\in \Omega^3(M)$, then 
    $$
    T=H_\psi+\gamma, \qquad \rmm{for\ some }\ \gamma\in \ker\vet H_\psi^3.
    $$
    In other words, the compatible connections with totally skew-symmetric torsion form an affine space modelled on $\ker\vet H_\psi$.
    In particular, $H_\psi$ is invariant under constant rescaling of $\psi$.
\end{remark}

\paragraph{Lee form and Lee space}
Here we introduce the Lee form, a natural 1-form associated with the 4-form geometry, which will be used to determine a canonical divergence operator in generalized geometry.

\begin{definition}
\label{def:lee:form}
Let $(M, G,\psi)$ be a 4-form geometry. We call $\zeta_\psi:=H_\psi\iprod\psi\in \Omega^1$ its \emph{Lee form}. 
\end{definition}

Consider a 4-form geometry $(M,G,\psi)$. If $\Omega^1$ is an irreducible $G$-module, then $\psi$ is \emph{non-degenerate}:
$$\ker\psi:=\{X\in TM:i_X\psi=0\}=\{0\}\subset TM.
$$
When $\Omega^1$ is not irreducible, we define the \emph{Lee space} by
\begin{equation*}
    \Omega^1_L\defeq \pts{\Omega^1\iprod\psi}\iprod\psi \subset\Omega^1.
\end{equation*}
By invariance of contractions, this space is a $G$-submodule of $\Omega^1$, and we denote the rank of $\Omega^1_L$ by $L\le n$.
In the next lemma, we give another characterisation of the Lee space.

\begin{lemma}
\label{lemma:characterization:Lee:space}
    Let $(M,G,\psi)$ be a 4-form geometry. Then its Lee space is 
    $
    \Omega^1_L =\pts{\ker\psi}^\perp
    $
    and $\zeta_\psi\in \Omega^1_L$.
\end{lemma}
\begin{proof}
    Let $A: \Omega^1\to \Omega^1$ be the map defined by $A(X)\defeq (X\iprod\psi)\iprod \psi$. By construction, we have $\Ima A=\Omega^1_L$. Moreover, a direct computation shows that $A$ is self-adjoint. Hence, $\Omega^1_L = \Ima A =(\ker A)^\perp$.
    
We now show that 
$$
\Omega^1_L=\pts{\ker\psi}^\perp. 
$$
If $X\in \Omega^1_L$, then $X=A(Y)$ for some $Y\in \Omega^1$. Therefore, for any $Z\in\ker\psi$,
\begin{equation*}
    \pdi{X,Z}
    =\pdi{A(Y),Z}
    =\pdi{Y, A(Z)}
    =\pdi{Y, (Z\iprod\psi)\iprod\psi}
    = 0.
\end{equation*}
Thus, $X\in\pts{\ker\psi}^\perp$, and consequently $\Omega^1_L\subset (\ker\psi)^\perp$. Conversely, suppose that $(X\iprod\psi)\iprod\psi=0$, i.e., $X\in \ker A$. Then
\[
0
=
\bracks{ (X \iprod \psi ) \iprod \psi , X}
=
-\bracks{X\iprod\psi,X\iprod\psi}
=
-|X\iprod\psi|^2,
\]
and hence $X\iprod\psi=0$. Therefore, $\ker A\subset \ker\psi$, which implies that $(\ker A)^\perp\supset (\ker\psi)^\perp$. Since $A$ is self-adjoint, we have $\Omega^1_L=\Ima A=(\ker A)^\perp\supset (\ker\psi)^\perp$, and equality follows.

Finally, the Lee form always belongs to the Lee space $\Omega^1_L$, because if $X\in \ker \psi$, we have
\[
    \bracks{\zeta_\psi, X}=
X\iprod (H_\psi\iprod\psi)=-H_\psi\iprod(X\iprod\psi)=0
\quad \Rightarrow \quad
\zeta_\psi\in (\ker \psi)^\perp = \Omega_L^1.\qedhere
\]
\end{proof}

\begin{remark}
\label{remark:Lee:form:2:irreps}
    Suppose that $(M^n,G,\psi)$ is such that $\Omega^1(M)$ is irreducible and $\Omega^3(M)$ has exactly two isomorphic irreducible components of dimension $n$ (as occurs with $\SU(4)$-structures, cf. \ref{parag:su(4):structures}), given by 
    $$\Omega^3_n = \Omega^1\iprod\psi \qandq \Omega^3_{n'} = \Omega^1\iprod\hat{\psi},
    $$ 
    where $\hat{\psi}\in\Omega^4(M)$ is another non-zero form whose stabiliser contains $G$, so that we have another 4-form geometry $(M,G,\hat{\psi})$. Since both components are identified with the irreducible $G$-module $\Omega^1(M)$ via contraction, we have a natural isomorphism
\begin{equation*}
    \Phi:\Omega^3_n\to \Omega^3_{n'}
    ,\quad
    X\iprod\psi\mapsto X\iprod\hat{\psi},
\end{equation*}
    and we may write $V_n=\Omega^3_n\oplus\Omega^3_{n'}$ and $\bh_n: V_n\to V_n$, following the notation of Theorem~\ref{theorem:flux:3:form:expression}/\ref{theorem:flux:3:form:expression:texto}. Therefore,
\begin{equation*}
    d^*\hat{\psi} = \tau^1\iprod\psi + \hat{\tau}^1\iprod\hat{\psi} + \eta,
    \qforq
    \eta\in V_n^\perp,
\end{equation*}
and the $n$-dimensional component of the flux 3-form $H_{\hat{\psi}}$ is given by $\bh_n^{-1}\pts{\tau^n}$, where $\bh_n = \bH_{\hat{\psi}}|_{V_n}$ and $\tau^n = \tau^1\iprod\psi + \hat{\tau}^1\iprod\hat{\psi}$. Writing
\begin{equation*}
    \bh_n = \begin{bmatrix}
         a\cdot\id_{\Omega^3_n} & b\cdot\Phi^{-1}\\
         c\cdot\Phi& d\cdot\id_{\Omega^3_{n'}}
     \end{bmatrix}
     \qforq
     a,b,c,d\in\R
     \quad \Longrightarrow\quad
     \bh_n^{-1}  = \frac{1}{ad-bc}\begin{bmatrix}
         d\cdot\id_{\Omega^3_n} & -b\cdot\Phi^{-1}\\
         -c\cdot\Phi& a\cdot\id_{\Omega^3_{n'}}
     \end{bmatrix},
\end{equation*}
we compute
\begin{align*}
    \bh_n^{-1}\pts{\tau^n}
    &=\frac{1}{ad-bc}\pts{
    d\pts{\tau^1\iprod\psi}
    -b\Phi^{-1}\pts{\hat{\tau}^1\iprod\hat{\psi}}
    -c\Phi\pts{\tau^1\iprod\psi}
    +a\pts{\hat{\tau}^1\iprod\hat{\psi}}
    }\\
    &= \frac{1}{ad-bc}\pts{\pts{d\tau^1-b\hat{\tau}^1}\iprod\psi + \pts{-c\tau^1+a\hat{\tau}^1}\iprod\hat{\psi}}.
\end{align*}
Now, since $\Omega^1$ is irreducible and $\hat{\psi}$ is non-zero, we have  
\begin{equation*}
    \Omega^1=\pts{\ker\hat{\psi}}^\perp= \pts{\Omega^1\iprod\hat{\psi}}\iprod\hat{\psi},
\end{equation*}
    where the last equality follows from Lemma~\ref{lemma:characterization:Lee:space}. Therefore, Schur's lemma implies that the double contraction map with $\hat{\psi}$ is a scalar multiple of the identity. Since this map is self-adjoint, the scalar is real:
\[
\pp{X\iprod\hat{\psi}}\iprod\hat{\psi}=\alpha\cdot X, \qquad \forall X\in \Omega^1.
\]
Thus, the Lee form associated with the 4-form geometry $(M,G,\hat{\psi})$ is given by
\begin{align*}
    \zeta_{\hat{\psi}} 
    &= H_{\hat{\psi}}\iprod\hat{\psi}
    = \frac{1}{ad-bc}\pts{\pts{d\tau^1-b\hat{\tau}^1}\iprod\psi}\iprod\hat{\psi} 
    + \frac{1}{ad-bc}\pts{\pts{-c\tau^1+a\hat{\tau}^1}\iprod\hat{\psi}}\iprod\hat{\psi} \\
    &= \frac{1}{ad-bc}\pts{\pts{d\tau^1-b\hat{\tau}^1}\iprod\psi}\iprod\hat{\psi}
    + \frac{\alpha}{ad-bc}\pts{-c\tau^1+a\hat{\tau}^1}.
\end{align*}
    For the choice of forms $\psi$ and $\hat{\psi}$ in \ref{parag:su(4):structures}, we show that the first term above vanishes.
\end{remark}

\paragraph{Instantons}
\label{parag:instantons}
    As discussed in \S\ref{subsec:generalized:geometry}, a string algebroid is determined by a 3-form $H\in\Omega^3(M)$ and a connection $\theta$ on a principal bundle over $M$, subject to the heterotic Bianchi identity \eqref{eq: hBi}. In the setting of 4-form geometries, we have identified a canonical choice of $H$, namely the flux 3-form $H_\psi\in\Omega^3(M)$. Although a 4-form geometry does not canonically determine a connection $\theta$, it imposes natural constraints on its choice. In particular, the instanton condition is one of the main hypotheses in our construction of generalized Ricci-flat metrics on string algebroids over 4-form geometries; see \S\ref{sec:ricci:curvature:4:forms}.
In this paragraph, we present the notion of a $G$-instanton used throughout this work.

\begin{definition}
\label{def:instanton}
    Let $(M,G)$ be a $G$-structure. We say that a connection $\vartheta$ on a principal $K$-bundle $P\to M$ is a $G$-instanton if its curvature 2-form $F_\vartheta\in \Omega^2(P,\fk)^G \isomorphic \Omega^2(M,\adjointbundle{P})$ satisfies
\begin{equation*}
    F_{\vartheta}\in \mathfrak g\otimes \adjointbundleP\subset \Omega^2(M,\adjointbundleP),
\end{equation*}
    where the Lie algebra $\mathfrak g=\rmm{Lie}(G)$ is identified with a subspace of 2-forms under the metric-induced isomorphism $\mathfrak{so}(n)\isomorphic \Omega^2$. In particular, if $(M, G,\psi)$ is a 4-form geometry and $\vtheta$ is a $G$-instanton, then $F_\vtheta\diamond\psi=0$, cf. the discussion below \eqref{eq:ker:diamond:=:lie:algebra}.
\end{definition}

\begin{remark}\label{remark:G-instanton:N(G)-structure}
\label{remark:instantons:normaliser}
    The notion of a $G$-instanton makes sense for a broader class of geometric structures; see \cite{Carrion1998,Madnick2024}. Let $G\subset \SO(n)$ be a closed subgroup, and denote its normaliser in $\SO(n)$ by $N(G)$. Suppose that $M^n$ carries an $N(G)$-structure. Since $G$ is closed, $N(G)$ is a closed Lie subgroup. Moreover, one can see that $\fg\subset \Lie(N(G))\defeq \fn(\fg)$ is invariant under the adjoint action of $N(G)$. Consequently, $\fg$ determines a well-defined subbundle, also denoted by $\fg$, of $\Lambda^2T^*M$. A connection $\vtheta$ on a principal $K$-bundle $P\to M$ is then called a $G$-instanton if its curvature $F_\vtheta\in \Omega^2(M,\ad P)$ satisfies
    \begin{equation*}
        F_\vtheta\in\fg\tensor\ad P\subset \Omega^2(M,\ad P).
    \end{equation*}
    This generalisation is relevant in several contexts \cite{Madnick2024}:
    \begin{itemize}
        \item If $G=\SU(2)=\sp(1)\subset \SO(4)$, then $N(G)=\SO(4)=\sp(1)\sp(1)$. Thus, the notion of an $\SU(2)$-instanton makes sense on any oriented Riemannian 4-manifold.

        \item If $G=\SU(m)\subset \SO(2m)$, with $m\geq 3$, then $N(G)=\U(m)$. Thus, the notion of an $\SU(m)$-instanton makes sense on any almost Hermitian manifold $(M^{2m},\U(m))$.

        \item If $G=\Sp(k)\subset \SO(4k)$, with $k\ge 2$, then $N(G)=\Sp(k)\Sp(1)$. Thus, the notion of an $\Sp(k)$-instanton makes sense on any almost quaternion-Hermitian manifold $(M^{4k},\Sp(k)\Sp(1))$.
    \end{itemize}
    For subgroups that are self-normalising in $\SO(n)$, the notion of $G$-instantons is allowed only for genuine $G$-structures. For example,
    \[
        N(\rG_2)=\rG_2\subset\SO(7)
        \qandq
        N(\Spin(7))=\Spin(7)\subset\SO(8),
    \]
    which will be discussed in \S\ref{sec:examples:4-form:geometry}.
\end{remark}

\begin{example}[Four dimensions]
\label{example:dimensao:quatro}
In the 4-form geometry $(M^4,\SO(4),\psi=\vol_M)$, the volume form is parallel, and hence $d^*\psi=0$ and $H_\psi=0$. The bundle of 2-forms decomposes as
\[
\Lambda^2T^*M=\Lambda^2_+T^*M\oplus\Lambda^2_-T^*M,
\qquad
\beta\in\Lambda^2_\pm T^*M
\iff
\star\beta=\pm\beta.
\]
Under the metric identification $\mathfrak{so}(4)\cong\Lambda^2(\bb R^4)^*$, these summands correspond to the two ideals
\[
\mathfrak{so}(4)=\mathfrak{su}(2)_+\oplus\mathfrak{su}(2)_-.
\]
Consequently, the corresponding self and anti-self dual $\SU(2)$-instanton conditions for a connection $\theta$ are
\[
F_\theta\in\Omega^2_\pm(M,\adjointbundleP)
\iff
\star F_\theta=\pm F_\theta
\iff
F_\theta\iprod\psi=\pm F_\theta.
\]
The minus sign gives the anti-self-duality equation used in Donaldson theory; reversing the orientation exchanges the self-dual and anti-self-dual equations \cite{Donaldson1990}. Our notion of instanton (Definition~\ref{def:instanton}) is motivated by this four-dimensional picture, together with the higher-dimensional framework of \cite{Donaldson1998,Donaldson2011,Tian2000} based on distinguished $(n-4)$-forms and the algebraic perspective developed in \cite{Carrion1998}.

The decomposition of $\mathfrak{so}(4)$ into two copies of $\fs\fu(2)$ is peculiar to four dimensions and gives rise to the two notions of instanton described above. No analogous phenomenon occurs for the higher-dimensional 4-form geometries considered in \S\ref{sec:examples:4-form:geometry}. If the structure group is reduced from $\SO(4)$ to $\SU(2)$ while retaining the 4-form $\psi=\vol_M$, this ambiguity disappears. With respect to the orientation induced by the $\SU(2)$-structure, $\Omega^2_+$ decomposes into three one-dimensional subbundles generated by the Hermitian forms, whereas the associated Lie algebra subbundle is the single component $\fg=\Omega^2_-$. Consequently, Definition~\ref{def:instanton} singles out the anti-self-duality equation as the $\SU(2)$-instanton condition.
\end{example}

\subsection{The Yang--Mills and Ricci--Bismut terms}
\label{subsec:yang:mills+ricci:bismut}

It is well known that the metrics induced by torsion-free
$\SU(m)$-, $\rG_2$-, $\spin(7)$-, and $\sp(k)$-structures are
Ricci-flat, $
    \Ric^g=0$ \cite{Berger1955,Salamon1989}. These structures provide important
examples of 4-form geometries, as discussed in greater detail
in Section~\ref{sec:examples:4-form:geometry}. Moreover, if
$\theta\in\Omega^1(P,\fk)$ is an instanton connection with respect to
one of these structures, then it satisfies the \emph{Yang--Mills
equation}, $d^{\theta *}F_\theta=0$  \cite{Tian2000,Fadel2019}.

For structures with torsion, the defining 4-form is generally not
parallel with respect to the Levi--Civita connection, i.e., $\nabla^g\psi\neq 0$.
Nevertheless, in many geometrically relevant cases, one does have $\nabla^+\psi=0$ with respect to the Bismut
connection \cite{Friedrich2003b,Ivanov2004}.
Since $\Ric^g$ does not generally vanish in the presence of torsion,
the more natural curvature tensor for our purposes is $\Ric^+$, the
Ricci curvature of the Bismut connection. Similarly, instantons in
this setting do not, in general, satisfy the ordinary Yang--Mills
equation.

Let $(P,K,\theta,\pdi{\cdot,\cdot}_\fk)\to (M,G,\psi)$ be a principal bundle with connection $\theta$ over a 4-form geometry. The \emph{Yang--Mills} and
\emph{Ricci--Bismut} equations associated with these data are defined,
respectively, by
\begin{align*}
\label{eq:YM}
    d^{\theta *}F_\theta
    -F_\theta\iprod H_\psi
    +\tilde \zeta_\psi\iprod F_\theta
    &=0,
    \tag{YM}
\\
\label{eq:RB}
    \Ric^+
    +F_\theta\circ F_\theta
    +\nabla^+\tilde \zeta_\psi
    &=0,
    \tag{RB}
\end{align*}
where $\tilde \zeta_\psi$ is the scalar multiple of the Lee form $\zeta_\psi$ specified in Definition~\ref{def:generalized:pair:4:form}, and $F_\theta\circ F_\theta:=\bracks{i_{e_\mu}F_\theta,i_{e_\mu}F_\theta}_{\fk}$, where the sum is over an orthonormal frame $\{e_\mu\}$ of the tangent bundle. The Ricci--Bismut equation usually appears decomposed into its symmetric and skew-symmetric parts
\[
  \Ric^g+F_\theta\circ F_\theta -\frac14H_\psi^2+\frac12\mcal L_{\tilde\zeta_\psi}g=0 
  \qandq
d^*H_\psi-d\tilde \zeta_\psi+\tilde \zeta_\psi\iprod H_\psi=0,
\]
where $H_\psi\circ H_\psi=H^{\mu\nu}\otimes H_{\mu\nu}$.

The equations \eqref{eq:YM} and \eqref{eq:RB} generalise, in the presence of torsion, the ordinary Yang--Mills equation and the Ricci flatness condition $\Ric^g=0$, respectively. 
As shown in \S\ref{subsec:generalized:geometry}, these equations arise as components of the generalized Ricci tensor of a string algebroid and are satisfied by the structures considered above. We refer to the left-hand sides of \eqref{eq:YM} and \eqref{eq:RB} as the \emph{Yang--Mills term} and the \emph{Ricci--Bismut term}, respectively:
\begin{align*}
\label{eq:YMt}
\tag{YMt}
    d^{\theta *}F_\theta
    -F_\theta\iprod H_\psi
    +\tilde \zeta_\psi\iprod F_\theta,
    \\
    \label{eq:RBt}
    \tag{RBt}
    \Ric^+
    +F_\theta\circ F_\theta
    +\nabla^+\tilde \zeta_\psi.
\end{align*}
Together, these equations can be regarded as a generalisation of the Ricci soliton equation and are referred to in the literature as the \emph{string generalized Ricci soliton equations} \cite{kennon2026}.

In this subsection, we consider general 4-form geometries and derive expressions for their Yang--Mills and Ricci--Bismut terms, with the aim of determining when \eqref{eq:YM} and \eqref{eq:RB} hold. In the setting of string algebroids, these conditions are equivalent to generalized Ricci flatness, as discussed in \S\ref{sec:ricci:curvature:4:forms}.

\paragraph{The Yang--Mills term \eqref{eq:YMt}}
Let $(M,G,\psi)$ be a 4-form geometry and consider the decomposition of the space of 2-forms into eigenspaces 
\begin{equation*}
\Omega^2(M)=\Omega^2_{1}\oplus\cdots\oplus\Omega^2_{p}
, \qquad \vet H^2_\psi|_{\Omega^2_{i}}=\lambda_i\cdot \id_{\Omega^2_{i}}.
\end{equation*}
Fix $i_0\in\{1,\dots,p\}$. Then, for any $\beta\in\Omega^2(M)$, we have
\begin{equation}
    \label{eq:projec:aux}
    \beta\iprod\psi-\lambda_{{i_0}}\beta = 
\sum_{i\neq i_0}^p(\lambda_i-\lambda_{i_0})\proj_{\Omega^2_{i}}(\beta).
\end{equation}
For each non-zero eigenvalue $\lambda_{i_0}$, define the map $\mathfrak A_{i_0}: \Omega^2\otimes \Omega^k\to \Omega^{1}\otimes \Omega^{k-1}$ by
\begin{equation*}
\mathfrak{A}_{i_0}(\beta\otimes \xi):=\frac1{\lambda_{i_0}}
 i_{e_\mu}(\beta\iprod\psi-\lambda_{i_0}\beta)\otimes i_{e_\mu}\xi
=
\sum_{i\neq i_0}^p
    \frac{\lambda_i-\lambda_{i_0}}{\lambda_{i_0}} i_{e_\mu}\pp{\proj_{\Omega^2_{i}}(\beta)}\otimes i_{e_\mu}\xi,
\end{equation*}
where summation over $\mu$ is understood with respect to an orthonormal basis $\{e_\mu\}$ of $T_pM$.
Note that if $\beta\in \Omega^2_{i_0}$ is an eigenvector, then $\mathfrak A_{i_0}(\beta \otimes \xi)=0$.

\begin{proposition}[Yang--Mills term]
\label{prop:yang:mills}
    Let $(M,G,\psi)$ be a 4-form geometry, and let
$(P,K,\theta,\pdi{\cdot,\cdot}_{\fk})$ be a principal bundle with a connection $\theta$. 
Then, for every eigenvalue $\lambda_{i_0}\neq 0$, the following identity holds:
{
\begin{equation*}
\begin{aligned}
{d^\theta}{^*} F_\theta -F_\theta\iprod H_\psi-\frac1{\lambda_{i_0}}\zeta_\psi\iprod F_\theta
&=
-\frac{1}{\lambda_{i_0}}
H_\psi\iprod \pp{F_\theta\diamond \psi }
+\mathfrak A_{i_0}
(\nabla^{\theta,+}F_\theta)
.
\end{aligned}
\end{equation*}
}
\end{proposition}
\begin{proof}
    Applying the characteristic codifferential $\delta^+:=-i_{e_j}\nabla^+_{e_j}=d^*-H_\psi\iprod^2{}$ to the curvature 2-form, we obtain
\begin{align*}
- i_{e_j} \nabla^{\theta,+}_{e_j}F_\theta&=\delta^{\theta,+}F_\theta
=d^{\theta*}F_\theta -F_\theta\iprod H_\psi.
\end{align*}
Since $\nabla^{\theta,+}F_\theta\iprod\psi=\nabla^{\theta,+}(F_\theta\iprod\psi)-F_\theta\iprod\nabla^+\psi$, taking the characteristic codifferential gives
\begin{align*}
   -i_{e_j}\pp{\nabla^{\theta,+}_{e_j} F_\theta \iprod \psi}
   &=
   \delta^{\theta,+}(F_\theta\iprod\psi)-F_\theta\iprod \delta^+\psi
   \\&=
   d^{\theta*}(F_\theta\iprod\psi)
   -H_\psi\iprod^2(F_\theta\iprod\psi)
   -F_\theta\iprod d^*\psi 
   +F_\theta\iprod(H_\psi\iprod^2\psi)
   \\&=
   \cancel{d^{\theta}F_\theta} \iprod\psi+ \bcancel{F_\theta\iprod d^*\psi}
   -(F_\theta\iprod\psi)\iprod H_\psi
   -\bcancel{F_\theta\iprod d^*\psi }
       +F_\theta\iprod(H_\psi \iprod^2\psi),
\end{align*}
where we have used the Bianchi identity $d^\theta F_\theta=0$.
Using \eqref{eq:projec:aux}, we compute
\begin{align*}
\lambda_{i_0}\cdot\mathfrak A_{i_0}(\nabla^{\theta,+}F_\theta)&=
\sum_{i\neq i_0}(\lambda_i-\lambda_{i_0})i_{e_\mu}\pp{\proj_{\Omega^2_{k_i}}\nabla^{\theta,+}_{e_\mu }F_\theta}
=
i_{e_\mu}\pp{\nabla^{\theta,+}_{e_\mu}F_\theta\iprod\psi-\lambda_{i_0}\nabla^{\theta,+}_{e_\mu}F_\theta}\\&=
(F_\theta\iprod\psi)\iprod H_\psi-F_\theta\iprod\pp{H_\psi\iprod^2\psi}
+\lambda_{i_0}\pp{d^{\theta *}F_\theta-F_\theta\iprod H_\psi}
\\
\Rightarrow\quad
{d^\theta}{^*} F_\theta -F_\theta\iprod H_\psi
&=
\frac{1}{\lambda_{i_0}}\pp{
F_\theta\iprod\pp{H_\psi\iprod^2\psi}
-\pp{F_\theta\iprod \psi }\iprod H_\psi}
+\mathfrak A_{i_0}
(\nabla^{\theta,+}F_\theta)
.
\end{align*}
We now consider the first term on the right-hand side. A direct computation in an orthonormal frame, using \(
e^{cd}\iprod e^{ijk}
=
(\delta_c^i\delta_d^j-\delta_c^j\delta_d^i)e^k
-(\delta_c^i\delta_d^k-\delta_c^k\delta_d^i)e^j
+(\delta_c^j\delta_d^k-\delta_c^k\delta_d^j)e^i\), gives
    \begin{align*}
        F_\theta\iprod (H_\psi\iprod^2\psi)
        &=\frac{1}{2!2!1!}
        F_\theta\iprod \pp{
        H_{abi}\psi_{abjk}e^{ijk}
        }
        =\frac18 F_{cd}H_{abi}\psi_{abjk} e^{cd}\iprod e^{ijk}
        \\&=
\frac14H_{abi}\psi_{abjk}
\left(
F_{ij}e^k-F_{ik}e^j+F_{jk}e^i
\right)
=
\frac14F_{ij}
\left(
2H_{abi}\psi_{abjk}
+
H_{abk}\psi_{abij}
\right)e^k
\\&=
(F_\theta\iprod \psi)\iprod H_\psi
+\frac12F_{ij}H_{abi}\psi_{abjk}e^k
=
(F_\theta\iprod \psi)\iprod H_\psi
-(F_j\iprod H_\psi)\iprod\psi_j.
\end{align*}
Next, we have
    \begin{align*}
    H_\psi\iprod (F_\theta\diamond\psi)=H_\psi \iprod(F_j\wedge \psi_j)
        =
        (F_j\iprod H_\psi)\iprod\psi_j-(H_\psi\iprod\psi_j)F_j,
    \end{align*}
    where the last term on the right-hand side above is related to the Lee form
    \[
    (H_\psi\iprod\psi_j)F_j=
    -e_j\iprod (H_\psi\iprod\psi )F_j=-(H_\psi\iprod\psi)_jF_j=
    -\zeta_jF_j=-\zeta_\psi\iprod F_\theta .
    \]
Combining the identities above, we obtain
    \begin{align*}
F_\theta\iprod\pp{H_\psi\iprod^2\psi}
-\pp{F_\theta\iprod \psi }\iprod H_\psi
&= -(F_j\iprod H_\psi)\iprod\psi_j
= -H_\psi\iprod(F_\theta\diamond\psi)-(H_\psi\iprod\psi_j)F_j
\\&= -H_\psi\iprod(F_\theta\diamond\psi) + \zeta_\psi\iprod F_\theta.\qedhere
\end{align*}
\end{proof}

\paragraph{The Ricci--Bismut term \eqref{eq:RBt}}
We introduce the \emph{Bismut form} $\mathfrak R^+\in \Omega^2(M,\Omega^2)$, 
obtained by interchanging the two pairs of indices in the curvature tensor $R_{\nabla^+}\in \Omega^2\otimes \Omega^2$ of the Bismut connection $\nabla^+$, cf. \eqref{eq:bismut:connection}:
    \begin{equation}
    \label{eq:definition:R^+}
    \mathfrak R^+
(X, Y )(Z, W) := g(R_{\nabla^+}(Z, W)X, Y),
    \end{equation}
    which in coordinates reads
\(
\mathfrak R^+_{ijkl}=R^+_{klij}.
\)

\begin{proposition}[Ricci--Bismut term]
\label{prop:ricci:bismut}
Let $(M,G,\psi)$ be a 4-form geometry, and let
$(P,K,\theta,\pdi{\cdot,\cdot}_{\fk})$ be a principal bundle with a connection $\theta$. Suppose that the flux
$3$-form $H_\psi$ satisfies the \emph{heterotic Bianchi identity} \eqref{eq: hBi4}.
Then, for every eigenvalue $\lambda_{i_0}\neq 0$, the following identity holds:
{
\begin{equation*}
\Ric^+ + F_\theta\circ F_\theta
-\frac1{\lambda_{i_0}}\nabla^+\zeta_\psi
=
-\frac{1}{\lambda_{i_0}}
H_\psi\iprod\nabla^+\psi
+
\mathfrak A_{i_0}
\pp{
\mathfrak R^+
-
\bracks{F_\theta\otimes F_\theta}_{\fk}
}.
\end{equation*}
}
\end{proposition}

    \begin{proof}
    We adapt the argument of Ivanov et al.\ (see \cite[Theorem 4.5]{Ivanov2023b} or \cite[Theorem 4.3]{Ivanov2023}) to express the Ricci curvature of $\nabla^+$ in terms of $dH_\psi$ and $\nabla^+H_\psi$.
    Using the decomposition of $\Omega^2$ and applying \eqref{eq:projec:aux}, we obtain 
    \begin{align*}
\frac12R^+_{ijab}{\psi^{ab}}_{kl}&=
\lambda_{i_0}\ R^+_{ijkl}
+\sum_{i\neq i_0}^p
(\lambda_i-\lambda_{i_0})\pp{\proj_{\Omega^2_{i}}\pp{R^+_{ij}}}_{kl}.
\end{align*}
For the Ricci curvature of $\nabla^+$, we have
\begin{align*}
    \lambda_{i_0}\cdot \rmm{Ric}^+_{ij} &:= \lambda_{i_0}\ R^+_{i\mu \mu j} 
= 
\frac12R^+_{i\mu ab}\psi_{ab \mu j} 
-\sum_{i\neq i_0}^p
(\lambda_i-\lambda_{i_0})\pp{\proj_{\Omega^2_{i}}\pp{R^{+}_{i\mu}}}_{\mu j}
\\&=
\frac{1}{6}\pp{
R^+_{i\mu ab}+
R^+_{i\mu ab}+
R^+_{i\mu ab}
}\psi_{ab \mu j} 
+\lambda_{i_0}\cdot\mathfrak A_{i_0}(\mathfrak R^+)_{ij}
\\&=
\frac{1}{6}\pp{
R^+_{i\mu ab}+
R^+_{i b\mu a}+
R^+_{i ab\mu}
}\psi_{ ab\mu j}
+\lambda_{i_0}\cdot\mathfrak A_{i_0}(\mathfrak R^+)_{ij} .
\end{align*}
Consider the Bianchi identity for connections $\nabla$ with skew-symmetric torsion $T\in \Omega^3(M)$ and curvature $R$, cf. \cite[\S 2]{Friedrich2003b} and \cite[Prop. 2.1]{Ivanov2024}:
\[
R(V, X, Y, Z) + R(V, Y, Z, X) + R(V, Z, X, Y ) = 
\frac12
dT(V,X, Y, Z) + (\nabla_V T)(X, Y, Z).
\]
Applying it to the Bismut connection $\nabla^+$, we have
\begin{align*}
\rmm{Ric}^+_{ij} &= \frac{1}{6\lambda_{i_0}}\pp{-\frac{1}{2}(dH_\psi)_{ab\mu i} + (\nabla^+_iH_\psi)_{ab\mu}}\psi_{ab\mu j} 
+\mathfrak A_{i_0}(\mathfrak R^+)_{ij}
\\
\label{eq:two:parts:ricci:ij:geral}
&=- \frac{1}{12\lambda_{i_0}}(dH_\psi)_{ab\mu i}\psi_{ab\mu j} + \frac{1}{6\lambda_{i_0}}(\nabla^+_iH_\psi)_{ab\mu}\psi_{ab\mu j}
+\mathfrak A_{i_0}(\mathfrak R^+)_{ij}.
\end{align*}
Choosing a basis $\{r_\alpha\}$ of $\mathfrak k$ with respect to the pairing $\langle\cdot,\cdot\rangle_\fk$, we have
\[
F_\theta= \frac{1}{2!}{F^\alpha}_{ij}e^{ij} \otimes r_\alpha
\Rightarrow
\langle F_\theta\wedge F_\theta \rangle_\fk  
= \frac{1}{4}{F^\alpha}_{ab}F_{\alpha kl}e^{abkl}.
\]
    Using the heterotic Bianchi identity \eqref{eq: hBi4}, we have 
$
    (dH_\psi)_i = 
{F{^\alpha}_{i\mu}F_{\alpha \nu\rho } e^{\mu\nu\rho}}$, and contracting with $\psi_j=\tfrac{1}{3!}\psi_{j\mu\nu\rho}e^{\mu\nu\rho}$, we obtain
\begin{align*}
    (dH_\psi)_i\iprod \psi_j 
    &= 
{F{^\alpha}{_{i}{^\mu}}F_{\alpha}{^{\nu\rho} }
}\psi_{j\mu\nu\rho} 
    = 2\lambda_{i_0}F{^\alpha}{_{i}{^\mu}}F_{\alpha j\mu}
+2\sum_{i\neq i_0}(\lambda_i-\lambda_{i_0}){F^\alpha}{_i}{^\mu}\pp{\proj_{\Omega^2_{i}}(F_\alpha)}{}_{j\mu}
\\&=
2\lambda_{i_0}\pp{\bracks{e_\mu\iprod F_\theta,e_\mu\iprod F_\theta}_\fk{}_{ij}
+\mathfrak A_{i_0}\pp{\bracks{F_\theta\otimes F_\theta}_\mathfrak k}}.
    \end{align*}
On the other hand, $
(dH_\psi)_i\iprod \psi_j=
\tfrac{1}{6}(dH_\psi)_{abc i}\psi_{abc  j}$. Combining this identity with the computations above, we obtain
\begin{align*}
    \Ric^++F_\theta\circ F_\theta &=\frac{1}{\lambda_{i_0}}
    \nabla^+H_\psi\iprod\psi +\mathfrak A_{i_0}\pp{\mathfrak R^+-\bracks{F_\theta\otimes F_\theta}_\mathfrak k},
\end{align*}
and the result follows from the Leibniz rule $\nabla^+\zeta_\psi=\nabla^+(H_\psi\iprod\psi)=\nabla^+H_\psi\iprod\psi+H_\psi\iprod\nabla^+\psi$.% and the definition of the Lee form $\zeta_\psi=H_\psi\iprod\psi$.
\end{proof}

\subsection{Background on generalized geometry}
\label{subsec:generalized:geometry}
In this section, we provide the terminology and results from generalized geometry needed for the construction of string algebroids over 4-form geometries. We refer the reader to \cite{Garcia-Fernandez2020a,Garcia-Fernandez2014,daSilva2024a} for details.

\paragraph{String algebroids and generalized connections}\label{parag:string:algebroids:and:generalized:connections}

Let us briefly review the main concepts from generalized geometry that will be used later to apply the theory of 4-form geometries.

\begin{definition}[Courant algebroid]
\label{def:courant:algebroid}
A \emph{Courant algebroid} \index{Courant algebroid} $(E, \langle \cdot, \cdot \rangle, [\cdot, \cdot], \pi)$ consists of a vector bundle $E \to M$ equipped with a non-degenerate symmetric bilinear form $\langle \cdot, \cdot \rangle$ (inducing an isomorphism $E^* \cong E$), a \emph{Dorfman bracket} $[\cdot, \cdot]$ on $\Gamma(E)$, and a bundle map $\pi \colon E \to TM$, called the \emph{anchor map},
satisfying the following axioms for all $a, b, c \in \Gamma(E)$ and $f \in C^\infty(M)$:
\begin{multicols}{2}
\begin{enumerate}[label=\noit, itemsep=1pt,parsep=1pt]
    \item $\left[a, [b, c]\right] = \left[[a, b], c\right] + \left[b, [a, c]\right]$,
    \item $\pi\left[a, b\right] = [\pi(a), \pi(b)]$,
    \item $\left[a, f b\right] = f \left[a, b\right] + \pi(a)(f) b$,
    \item $\pi(a) \langle b, c \rangle = \langle [a, b], c \rangle + \langle b, [a, c] \rangle$,
    \item $\left[a, b\right] + \left[b, a\right] = \pi^* d \langle a, b \rangle$.
\end{enumerate}
\end{multicols}
\noindent A Courant algebroid is called \emph{transitive}, \emph{heterotic}, or \emph{of string type} if $\pi$ is surjective.
\end{definition}

\begin{example}[String algebroids] 
\label{example:transitive:Courant}
\label{example:string:alegbroid}
Let $P \to M$ be a principal $K$-bundle, and suppose that $\mathfrak{k}=\Lie(K)$ is equipped with a non-degenerate bi-invariant symmetric pairing $\langle \cdot, \cdot \rangle_\mathfrak{k}$. Consider the vector bundle
\begin{equation*}
E = TM \oplus \adjointbundle{P} \oplus T^*M,
\end{equation*}
where $\adjointbundle{P} = P \times_{\text{Ad}} \mathfrak{k}$ is the adjoint bundle. Equip $E$ with the symmetric pairing
\begin{equation*}
\langle X + r + \xi, Y + t + \eta \rangle = \frac{1}{2} \left( \eta(X) + \xi(Y) \right) + \langle r, t \rangle_\mathfrak{k},
\end{equation*}
and the anchor map $\pi \colon X + r + \xi \mapsto X$.

Given a 3-form $H \in \Omega^3(M)$ and a connection $\theta \in \Omega^1(P, \mathfrak{k})$ with curvature 2-form $F_\theta \in \Omega^2(P, \mathfrak{k})$, define the Dorfman bracket on $\Gamma(E)$ by
\begin{equation*}
\begin{split}
[X + r + \xi, Y + t + \eta] &=  [X, Y] + \mathcal{L}_X \eta - i_Y d\xi + i_Y i_X H - [r, t] - F_\theta(X, Y) 
\\& \quad
+ d^\theta_X t 
 - d^\theta_Y r + 2 \langle d^\theta r, t \rangle_\mathfrak{k} + 2 \langle i_X F_\theta, t \rangle_\mathfrak{k} - 2 \langle i_Y F_\theta, r \rangle_\mathfrak{k}.
\end{split}
\end{equation*}
The data $(E, \langle \cdot, \cdot \rangle, [\cdot, \cdot], \pi)$ define a Courant algebroid if and only if the \emph{heterotic Bianchi identity}, 
\begin{equation}\label{eq: hBi}
    \tag{HBI}
    dH=\bracks{F_\theta\wedge F_\theta}_\fk,
\end{equation}
 holds. This construction defines a string algebroid, since $\pi$ is surjective.

\medskip
On the other hand, for any string algebroid $E \to M$, one can construct a 3-form $H \in \Omega^3(M)$, a principal $K$-bundle $P \to M$, and a connection form $\theta \in \Omega^1(P, \mathfrak{k})$ satisfying the heterotic Bianchi identity \eqref{eq: hBi}, such that $E$ is isomorphic to the Courant algebroid constructed above. This allows us to work exclusively with the latter in the context of string algebroids. See \cite[\S 1.3]{Chen2013} or \cite[\S 2]{Garcia-Fernandez2014} for details. Note that \eqref{eq: hBi} implies that the first Pontryagin class of $P$ associated with $\bracks{\cdot,\cdot}_\fk$ via Chern--Weil theory is trivial:
$p_1(P) = 0 \in H^4_{dR}(M, \bb R)$.
\end{example}

\begin{definition}[Generalized metric]
\label{def:generalized:metric}
Let $M$ be an oriented manifold with a Courant algebroid $E$. A \emph{generalized metric} \index{generalized metric} on $E$ is an $\bracks{\cdot,\cdot}$-orthogonal decomposition 
\begin{equation*}
E = V_+ \oplus V_-,
\end{equation*}
such that $\langle \cdot, \cdot \rangle$ is positive definite on $V_+$ and negative definite on $V_-$, and the restriction $\pi|_{V_+} \colon V_+ \to TM$ is an isometry.
\end{definition}

A generalized metric on a Courant algebroid $E\to M$ naturally induces a Riemannian metric $g$ on $M$ via the isometry $\pi|_{V_+} \colon V_+ \to TM$.
Conversely, if $M$ is endowed with a Riemannian metric $g$, then, for the construction in Example~\ref{example:transitive:Courant}, where $E = TM \oplus \adjointbundle{P} \oplus T^*M$, the following splitting defines a generalized metric:
\begin{equation*}
\label{eq:generalized:metric:transitive}
\tag{GM}
\begin{split}
V_+ &= \{X + gX : X \in TM\}, \quad 
V_- = \{X + r - gX : X \in TM, r \in \adjointbundle{P}\}.
\end{split}
\end{equation*}
Conversely, any generalized metric on a string algebroid induces a Riemannian metric on $M$ such that the generalized metric takes the form \eqref{eq:generalized:metric:transitive}. See \cite{daSilva2024a} for details. Consequently, the splitting \eqref{eq:generalized:metric:transitive} gives an isometry $V_+\isomorphic TM$ and an anti-isometry $V_-\isomorphic TM\oplus \adjointbundle{P}$, identifications that we use throughout the text.

\begin{definition}[Generalized connection]
\label{def:generalized:connection}
A \emph{generalized connection} \index{generalized connection} $D$ on a Courant algebroid $E \to M$ is a first-order differential operator
\(
D \colon \Gamma(E) \to \Gamma(E^* \otimes E)
\)
satisfying a Leibniz rule and compatible with $\bracks{\cdot,\cdot}$:
\begin{align*}
D_a(f b) &= f D_a b + \pi(a)(f) b
\qandq
\pi(e) \langle a, b \rangle = \langle D_e a, b \rangle + \langle a, D_e b \rangle,
\end{align*}
for $a,b,e\in \Gamma(E)$ and $f \in C^\infty(M)$. We say that $D$ is \emph{$V_+$-compatible} if
\(
D(\Gamma(V_\pm)) \subset \Gamma(E^* \otimes V_\pm)\). The space of all $V_+$-compatible generalized connections is denoted $\mathcal{D}({V_+})$.
\end{definition}

\begin{definition}[Generalized curvature]
\label{def:generalized:curvature}
Let $E$ be a Courant algebroid with a generalized metric $E = V_+ \oplus V_-$ and $D \in \mathcal{D}({V_+})$. The \emph{generalized curvatures} $\mathrm{GR}_D^\pm \in \Gamma(V_\pm^* \otimes V_\mp^* \otimes \mathrm{End}(V_\pm))$ are defined by
\begin{equation*}
\mathrm{GR}_D^\pm(e_1^\pm, e_2^\mp) = D_{e_1^\pm} D_{e_2^\mp} - D_{e_2^\mp} D_{e_1^\pm} - D_{[e_1^\pm, e_2^\mp]}.
\end{equation*}
The corresponding \emph{generalized Ricci tensors} $\mathrm{GRic}^\pm\in \Gamma(V_\mp^*\otimes V_\pm^*)$ are defined by
\begin{equation*}
\mathrm{GRic}_D^\pm(e_2^\mp, e_3^\pm) = \mathrm{tr} \left( e_1^\pm \mapsto \mathrm{GR}_D^\pm(e_1^\pm, e_2^\mp) e_3^\pm \right).
\end{equation*}
\end{definition}

The notion of generalized curvature is defined only in the presence of a generalized metric because such quantities are not well-behaved in the general setting \cite[\S4.1]{Garcia-Fernandez2014}.

\begin{definition}[Generalized torsion]
\label{def:generalized:torsion}
The \emph{generalized torsion} $T_D \in \Gamma(\Lambda^3(E^*))$ of a generalized connection $D$ is defined by
\begin{equation*}
T_D(a, b, c) = \langle D_a b - D_b a - [a, b], c \rangle + \langle D_c a, b \rangle.
\end{equation*}
The set of all generalized connections $D$ that have vanishing torsion $T_D=0$ is denoted by $\mcal D^0$.
\end{definition}

\paragraph{Divergence operators}
In the next paragraph, we carry out an explicit computation of the generalized Ricci curvature associated with torsion-free generalized connections on string algebroids. Besides providing an explicit formula, the main point is that the resulting expression does not depend on the choice of generalized connection, but only on the generalized metric and on a particular quantity associated with the connection, called the \emph{divergence operator}, which we now introduce.

\begin{definition}[Divergence operator]
\label{def:divergence}
A \emph{divergence operator} \index{divergence operator} on a Courant algebroid $E \to M$ is a bundle map $\mathrm{div} \colon \Gamma(E) \to C^\infty(M)$ satisfying the Leibniz rule
\begin{equation*}
    \mathrm{div}(f a) = f \mathrm{div}(a) + \pi(a)(f),
    \qforq f \in C^\infty(M)
    \qandq a \in \Gamma(E).
\end{equation*}
\end{definition}

Two examples of divergence operators will be useful below.
First, if $D$ is a generalized connection on a Courant algebroid $E\to M$, then
\begin{equation*}
\mathrm{div}^D(e) = \mathrm{tr}(D e)
\end{equation*}
defines a divergence operator on $E$, called the divergence induced by $D$. In particular, the set of divergence operators on $E$ is always non-empty. In the presence of a generalized metric on $E$, we denote the set of connections satisfying $\dv^D=\dv$ by $\mcal D(\dv)$. We also consider the following spaces of generalized connections:
\begin{align*}
\mathcal{D}(V_+, \mathrm{div}) &= \{D \in \mathcal{D}({V_+}) : \mathrm{div}^D = \mathrm{div}\},
\\
\mathcal{D}^0(V_+, \mathrm{div}) &= \{D \in \mathcal{D}(V_+) : \mathrm{div}^D = \mathrm{div}, T_D = 0\}.
\end{align*}
If, moreover, the manifold is oriented by the volume form $\vol_M\in \Omega^n(M)$ of the Riemannian metric induced by $V_+$ (see Definition~\ref{def:generalized:metric} and \eqref{eq:generalized:metric:transitive}), then a second important example is the \emph{Riemannian divergence}, given by
\begin{equation}
\label{eq:riemannian:divergence}
\mathrm{div}^{V_+}(e) := \frac{1}{\mathrm{vol}_M} \mathcal{L}_{\pi(e)} \mathrm{vol}_M,
\end{equation}
where $\mcal L$ denotes the Lie derivative.

\medskip
The divergence operators on a Courant algebroid $E \to M$ constitute an affine space modelled on $\Gamma(E)$. Specifically, if $\mathrm{div}$ is a fixed divergence operator, any other $\mathrm{div}'$ is determined by a unique section $e \in \Gamma(E)$:
\begin{equation*}
    \mathrm{div}' 
    = \mathrm{div} + \langle e, \cdot \rangle.
\end{equation*}
For our purposes, we fix the Riemannian divergence; then, for each generalized connection $D$, we have
\(
\dv^D=\dv^{V_+}-\bracks{e,\cdot}\), where the sign is chosen for convenience.

\paragraph{Generalized Ricci tensor of string algebroids}
\label{parag:ricci:tensor:string:algebroids}
Our primary focus is on the generalized Ricci tensor of string algebroids \(
E = TM \oplus \adjointbundle{P} \oplus T^*M
\) over special-structure manifolds $M$, endowed with a pair \((H, \theta)\) satisfying the \emph{heterotic Bianchi identity} \eqref{eq: hBi}, cf. Example~\ref{example:transitive:Courant}.

\begin{theorem}[{\cite[Prop. 4.2]{Garcia-Fernandez2014}}]
\label{theorem:explicit:formula:generalized:ricci:curvature}
Let $(V_+, \mathrm{div})$ be a pair consisting of a generalized metric $V_+$ and a divergence operator $\mathrm{div}$ on a string algebroid
$$E = TM \oplus \adjointbundle{P} \oplus T^*M.$$
Suppose that the string algebroid structure is determined by the pair $(H,\theta)$ as in Example~\ref{example:transitive:Courant} and that the generalized metric is given by \eqref{eq:generalized:metric:transitive}. Let $e = Z + z + \zeta \in \Gamma(E)$ be such that
\begin{equation*}
\mathrm{div} = \mathrm{div}^{V_+} - \langle e, \cdot \rangle.
\end{equation*}
Define $\zeta_\pm = \frac{1}{2}(Z^\flat \pm \zeta) \in \Omega^1(M)$. Then, for any generalized connection $D \in \mathcal{D}^0(V_+, \mathrm{div})$, the \emph{generalized Ricci tensor} is given by
\begin{equation*}
\begin{split}
\mathrm{GRic}^+_D(a_-, b_+) &= i_Y i_X \pp{ \mathrm{Ric}^+ + F_\theta \circ F_\theta + \nabla^+ \zeta_+ }
%\\&\hspace{2cm}
- i_Y \left\langle d^{\theta *} F_\theta - F_\theta \iprod H + {\zeta_+} \iprod F_\theta, r \right\rangle_\mathfrak k,
\\
\GRic^-_D(b_+,a_-)&=
i_X i_Y \pp{ \mathrm{Ric}^- + F_\theta \circ F_\theta - \nabla^- \zeta_- +\bracks{F_\theta,z}_\fk}
\\&\hspace{4cm}
- i_Y \left\langle d^{\theta *} F_\theta - F_\theta \iprod H - {\zeta_-} \iprod F_\theta-d^\theta z, r \right\rangle_\mathfrak k,
\end{split}
\end{equation*}
where $a_- = X + r - X^\flat \in V_-$, $b_+ = Y + Y^\flat \in V_+$, and $\mathrm{Ric}^\pm$ is the Ricci curvature of the Bismut/Hull connections $\nabla^\pm$ defined by \eqref{eq:bismut:connection}.
\end{theorem}

The theorem above shows, in particular, that the generalized Ricci curvature associated with any connection $D \in \mathcal{D}^0(V_+, \mathrm{div})$ depends only on the pair $(V_+, \mathrm{div})$, which we call a \emph{generalized pair}, and is therefore independent of the choice of torsion-free generalized connection. Thus, the generalized Ricci curvature is an intrinsic invariant of $(V_+, \mathrm{div})$ and may be computed using the explicit connection described below. We henceforth refer to $\GRic$ as the \emph{generalized Ricci tensor of the pair} $(V_+,\dv)$ and denote it by $\GRic_{V_+,\dv}$ or $\GRic(V_+,\dv)$.

\medskip
A particular case of interest is when the section $e\in \Gamma(E)$ determining the divergence has only the 1-form component, i.e., $e \in T^*M \subset E = TM \oplus \adjointbundle{P} \oplus T^*M$. In this case, $\zeta_+ = \tfrac{1}{2} e$, and 
\begin{equation}
\label{eq:divergence:by:1:form}
\mathrm{div} = \mathrm{div}^{V_+} - 2 \langle \zeta_+, \cdot \rangle.
\end{equation}

\begin{remark}
\label{remark:Gric+=Gric-}
Let $(V_+,\dv)$ be a generalized pair, with $\dv$ as in Theorem~\ref{theorem:explicit:formula:generalized:ricci:curvature}, and write $e=Z+z+\zeta$. 
Then the two generalized Ricci tensors carry the same information; more precisely, \cite[Lemmas~2.5 and~2.7]{garciafernandez2024pluriclosedflowhullstrominger},
\[
    \GRic^+(a_-,b_+)=\GRic^-(b_+,a_-),
\]
if and only if
\begin{equation}
\label{eq:compatible:pair}
    \mcal L_Zg=0,
    \qquad
    d^\theta z=-i_ZF_\theta,
    \qquad
    d\zeta=i_ZH-2\bracks{F_\theta,z}_\fk.
\end{equation}
For a divergence of the form \eqref{eq:divergence:by:1:form}, we have $e=2\zeta_+=\zeta$, the conditions \eqref{eq:compatible:pair} reduce to $d\zeta=0$. Throughout this article, however, we do not assume that $\zeta$ is closed. Accordingly, we use the term \emph{generalized Ricci-flat} to mean that $\GRic^+=0$. When $d\zeta=0$, the equality above shows that $\GRic^+=0$ is equivalent to $\GRic^-=0$, and hence both tensors vanish.
\end{remark}

\paragraph{Explicit construction of $D\in \mcal D^0(V_+,\dv)$}
\label{parag:explicit:cnstruction:generalized:connection}

When a generalized connection $D$ is compatible with the generalized metric, it decomposes into four components:
\[
D^\pm_\mp:\Gamma(V_\pm)\to \Gamma(V_\mp^*\otimes V_\pm).
\]
Using the isometry $V_+\isomorphic TM$ and the anti-isometry $V_-\isomorphic TM\oplus \adjointbundleP$, write the divergence as
\[
\dv=\dv^{V_+}-\bracks{e,\cdot},
\qquad e=Z+z+\zeta\in\Gamma(TM\oplus\adjointbundleP\oplus T^*M),
\]
and set $\zeta_\pm=\tfrac12\pp{Z^\flat\pm\zeta}$. An explicit generalized connection $D\in \mcal D^0(V_+,\dv)$ is then given by
\[
    \pp{D^+_+}_X =\nabla^{1/3}_X+\frac1{n-1}\pp{X^\flat\wedge\zeta_+}^\sharp,
    \qquad 
    \pp{D^+_-}_{X+r}=
    \nabla^+_X-\bracks{F_\theta\diamond\cdot{},r}^\sharp,
    \]\[
    \pp{D^-_+}_X=\begin{bmatrix}
        \nabla^-_X&-\bracks{i_XF_\theta,\cdot}^\sharp
        \\
        -i_XF_\theta&\nabla^\theta_X
    \end{bmatrix},
    \]
    \begin{align*}
    (D^-_-)_{X+r}&=
    \begin{bmatrix}
        \nabla^{-1/3}_X-{\frac{1}{3}}\bracks{ F_\theta,r}^\sharp
        &
        -\frac{2}{3}\bracks{i_XF_\theta,\cdot}^\sharp
        \\
        -\frac{2}{3}i_XF_\theta
        &
        \nabla^\theta_X-\frac13[r,\cdot]
    \end{bmatrix}
    \\&\hspace{2cm}
    +\frac{1}{\dim(\fk)+n-1}\begin{bmatrix}
        -\pp{X^\flat\wedge\zeta_-}^\sharp
        &
        r^\flat\otimes \zeta_-^\sharp-z^\flat\otimes X
        \\
        \zeta_-\otimes r-X^\flat\otimes z
        &
        r^\flat\otimes z-z^\flat\otimes r
    \end{bmatrix},
\end{align*}
Here, $r^\flat$ and $z^\flat$ are defined using the pairing on $\adjointbundleP$, and $\nabla^{\pm1/3}=\nabla^g\pm\tfrac16 g^{-1}H$ denote the metric connections with skew torsion $\pm\tfrac13H\in\Omega^3(M)$; see \cite[\S 5.2]{Garcia-Fernandez2017}.

\medskip
An important property of the mixed operators $D^+_-$ and $D^-_+$ is that they do not depend on the choice of generalized connection $D\in \mcal D^0(V_+)$; see \cite[Lemma 3.2]{Garcia-Fernandez2019}. In fact, they are the natural operators given by
\[
D_{e_-}e_+=[e_-,e_+]_+, \qquad 
D_{e_+}e_-=[e_+,e_-]_-.
\]
These operators will appear later in the generalized Ricci tensor for 4-form geometries.

\subsection{Generalized Ricci flatness and coupled equations}
\label{sec:ricci:curvature:4:forms}
This section contains two of the main results of the paper, establishing complementary links between 4-form geometries and generalized Ricci flatness. In \ref{par:gravitino:ricci:flatness}, we introduce the gravitino equation and the generalized pair naturally induced by a 4-form geometry $(M,G,\psi)$. We then prove Theorem~\ref{theorem:ricci:flatness:simple:group}, which shows that, under an appropriate algebraic condition on the flux operator $\bH_\psi^2$, the gravitino equation forces the generalized Ricci tensor to vanish. In \ref{par:coupled:instantons}, we turn to the coupled instanton equation, relate it to the gravitino equation, and recover results from \cite{daSilva2024a}. Finally, Theorem~\ref{theorem:ricci:flatness:coupled:equations} expresses the generalized Ricci tensor of a solution to the coupled instanton equation in terms of the flux 3-form and the gravitino term.

\paragraph{Generalized Ricci flatness and the gravitino condition}
\label{par:gravitino:ricci:flatness}

The main result of this paper is that certain 4-form geometries $(M,G,\psi)$ determine generalized Ricci-flat metrics whenever $\theta$ is a $G$-instanton and $\Hol(\nabla^+)\subset G$. These conditions are encoded in a natural equation called the \emph{gravitino equation}.

\begin{definition}[\cite{daSilva2024a}]
\label{def:gravitino:equation}
    We say that a tensor $\xi\in \Gamma(\mcal T^{p,q}(TM))$ satisfies the \emph{gravitino equation} if $D^+_-\xi=0$ (see \ref{parag:explicit:cnstruction:generalized:connection}). We refer to $D^+_-\xi\in \Gamma(V_-^*\otimes \mcal T^{p,q})$ as the \emph{gravitino term} of $\xi$.
\end{definition}

The gravitino equation $D^+_-\xi=0$ is equivalent to the tensor being parallel with respect to the Bismut connection and being annihilated by the curvature 2-form of $\theta$ under the diamond operator; see \ref{parag:explicit:cnstruction:generalized:connection}:
    \begin{equation}
    \label{eq:gravitino:equation}
    D^+_-\xi=0
    \quad\iff\quad
    \nabla^+\xi=0 \qandq F_\theta\diamond\xi=0.
    \end{equation}

\begin{definition}
\label{def:generalized:pair:4:form}
Let $
(P,K,\theta,\pdi{\cdot,\cdot}_{\fk})
\to
(M,G,\psi)
$ (or $(M,N(G),\psi)$)
be the string algebroid $E$ induced by a 4-form geometry for which the Lie algebra $\fg$ is an eigenspace of the flux operator $\vet H_\psi^2$ with eigenvalue $b\neq 0$. Then $E$ carries an induced pair
\begin{equation}
\label{eq:generalized:induced:pair}
\pp{V_+^{G,\psi},\dv^{G,\psi}}
\end{equation}
consisting of a generalized metric and a divergence operator. The generalized metric $V_+^{G,\psi}$ is given by \eqref{eq:generalized:metric:transitive}, using the metric induced by the $G$-structure, whereas the divergence operator is defined in terms of the Lee form by
\begin{equation}
\label{eq:divergence:4:form:geometry}
\dv^{G,\psi}
=
\dv^{V_+^{G,\psi}}
-
2\bracks{-\frac{1}{b}\zeta_\psi,\cdot}.
\end{equation}
We refer to this pair as the \emph{generalized pair induced by the 4-form geometry}.
\end{definition}

\begin{theorem}[Theorem~\ref{theorem:ricci:flatness:simple:group}]
\label{theorem:ricci:flatness:simple:group:texto}
    Let $E: (P, K; \theta, \pdi{\cdot,\cdot}_\fk)\to (M, G,
    \psi )$ be the string algebroid induced by a 4-form geometry for which $\vet H_\psi^2|_\fg=b\cdot \id_\fg$ with $b\neq 0$. If $G$ is the stabiliser of $\psi$ and a multi-tensor $\xi=(\xi_1,\dots,\xi_m)$, and both $\psi$ and $\xi$ satisfy the gravitino equation, 
    \[
    G = \stab_{\SO(n)}(\psi,\xi)\qandq 
    D^+_-\psi=0=D^+_-\xi, 
    \]
    then the generalized Ricci tensor of the induced pair \eqref{eq:generalized:induced:pair} vanishes:
$$\GRic\pp{V_+^{G,\psi},\dv^{G,\psi}}=0.$$
\end{theorem} 
{
\begin{proof}
By Theorem~\ref{theorem:explicit:formula:generalized:ricci:curvature}, it remains only to verify the Yang--Mills equation~\eqref{eq:YM} and the Ricci--Bismut equation~\eqref{eq:RB}. Since $\psi$ satisfies the gravitino equation \eqref{eq:gravitino:equation}, we have $
\nabla^+\psi=0$ and $
F_\theta\diamond\psi=0$.
Hence, applying Propositions~\ref{prop:yang:mills} and~\ref{prop:ricci:bismut} to $\fg$ as an eigenspace, the Yang--Mills and Ricci--Bismut terms are given by
\begin{align*}
d^{\theta *}F_\theta-F_\theta\iprod H_\psi
-\tfrac{1}{b}\zeta_\psi\iprod F_\theta
&=
\mathfrak A_b\bigl(\nabla^{\theta,+}F_\theta\bigr),
\\
\Ric^++F_\theta\circ F_\theta-\tfrac{1}{b}\nabla^+\zeta_\psi
&=
\fA_b\bigl(\fR^+-\bracks{F_\theta\otimes F_\theta}\bigr).
\end{align*}
Since $\nabla^+\xi=0=\nabla^+\psi$ and the $G$-structure is determined by $\pp{\xi,\psi}$, the connection $\nabla^+$ preserves the $G$-structure. Thus its curvature is $\fg$-valued:
\[
R_{\nabla^+}\in\Omega^2\otimes\fg
\quad \Longrightarrow \quad
\fR^+\in\fg
\quad\Longrightarrow\quad
\fA_b(\fR^+)=0.
\]
Moreover, the gravitino conditions, together with the fact that $G$ is determined by $(\psi,\xi)$, imply the instanton condition, $F_\theta\in\fg$, i.e., $\proj_{\Lg^\perp}(F_\theta)=0$.
Since $\nabla^+$ preserves the $G$-decomposition, it follows that
\[
\proj_{\Lg^\perp}\bigl(\nabla^{\theta,+}F_\theta\bigr)
=
\nabla^{\theta,+}\bigl(\proj_{\Lg^\perp}F_\theta\bigr)
=0
\quad\Longrightarrow\quad
\mathfrak A_b\bigl(\nabla^{\theta,+}F_\theta\bigr)=0.
\]
Likewise, the instanton condition gives $
\fA_b\bigl(\bracks{F_\theta\otimes F_\theta}\bigr)=0$. 
Hence, both the Yang--Mills and Ricci--Bismut equations are satisfied. The conclusion now follows from Theorem~\ref{theorem:explicit:formula:generalized:ricci:curvature}, yielding the generalized Ricci flatness condition.
\end{proof}
}

    Theorem~\ref{theorem:ricci:flatness:simple:group} is stated using a tensorial description of the $G$-structure, $G=\stab_{\SO(n)}(\psi,\xi)$, together with the gravitino equation for $\psi$ and $\xi$. Alternatively, the same conclusion follows directly from $\theta$ being a $G$-instanton and $\rmm{Hol}(\nabla^+)\subset G$. We adopt the spinorial description because it provides a more natural treatment of structures with structure groups $\SU(m)$ and $\sp(k)$ in \S\ref{sec:examples:4-form:geometry}.

\begin{corollary}
\label{corollary:GRIC:dimensao:quatro}
Let $M^4$ be a Calabi--Yau surface, and let $(P,K;A,\pdi{\cdot,\cdot}_\fk)$ be a principal $K$-bundle endowed with an anti-self-dual instanton $A$. If $\bracks{F_A\wedge F_A}_\fk=0$, then the string algebroid $E:(P, K; A, \pdi{\cdot,\cdot}_\fk)\to (M^4,\SU(2),\vol_M)$ induced by the 4-form geometry is generalized Ricci-flat.
\end{corollary}
\begin{proof}
    Since $M$ is Calabi--Yau, i.e., torsion-free as an $\SU(2)$-structure manifold, we have $H_{\vol_M}=0$ and $\Ric^g=0$. The string algebroid structure is well-defined: 
    $$
    \bracks{F_A\wedge F_A}_\fk=0=dH_{\vol_M}.
    $$
    Now, following Example~\ref{example:dimensao:quatro}, we have $\zeta_{\vol_M}=0$ and, since $A$ is an instanton, the \emph{Yang--Mills equation} is satisfied:
\[
F_A\in \fs\fu(2)\Rightarrow
*F_A=-F_A\Rightarrow
d^{A*}F_A =- \star d^A F_A =0.
\]
    Generalized Ricci flatness in this case is equivalent to 
\[
\Ric^g+F_A\circ F_A=0\qandq
d^{A*}F_A=0.
\]
    Finally, by the proof of Proposition~\ref{prop:ricci:bismut}, $-2F^\alpha{}_{\mu i}F_{\alpha \mu j}=\pp{dH_{\vol_M}}_i\iprod\psi_j=0$. Thus, $F_A\circ F_A =0$, and the left-hand side above vanishes identically.
\end{proof}

\paragraph{The coupled instanton equations}
\label{par:coupled:instantons}

In this subsection, we introduce a weaker version of the gravitino equation, proposed by de la Ossa et al.\ \cite{delaOssa2018a} in the physics literature and reinterpreted by Garcia-Fernandez and Molina \cite{Garcia-Fernandez2025} in the context of generalized geometry. The resulting equations were also investigated from a spinorial perspective in \cite{Garcia-Fernandez2017,Garcia-Fernandez2019,daSilva2024a}.

\begin{definition}[Coupled instanton]
\label{def:coupled:instanton}
    Let $E\to(M,N(G),\psi)$ be a string algebroid over a 4-form geometry, where $G\subset \SO(n)$ is a closed subgroup. We say that $E$ satisfies the \emph{coupled $G$-instanton equation} if $D^-_+$ is a $G$-instanton in the sense of remark~\ref{remark:G-instanton:N(G)-structure}.
\end{definition}

\begin{lemma}[{\cite[Lemma 4.7]{Garcia-Fernandez2025}}]
\label{lemma:eq:curvature:coupled}
    Under the identifications $V_+\isomorphic TM$ and $V_-\isomorphic TM\oplus \adjointbundle{P}$, we have
 \begin{equation*}
     F_{D^-_+}=\begin{bmatrix}
         \mathfrak R^+-\bracks{F_\theta\otimes F_\theta} & -\nabla^{\theta,+}F_\theta^\dagger 
         \\
         \nabla^{\theta,+}F_\theta & [F_\theta,{}\cdot{}]+F_\theta\diamond\bracks{\cdot,\cdot }^{-1} F_\theta 
     \end{bmatrix},
 \end{equation*}
    where $\mathfrak R^+\in  \Omega^2(M,\Omega^2)$ is given by \eqref{eq:definition:R^+}.
\end{lemma}

\begin{proposition}
\label{prop:gravitino=>coupled}
    Let $E:(P, K, \theta, \pdi{\cdot,\cdot}_\fk) \to (M, H)$ be a string algebroid with generalized metric $V_+$ and $\xi\in \Gamma(\mcal T^{p,q})$ a tensor on $M$ satisfying $D^+_-\xi=0$, then $F_{D^-_+}\diamond \xi=0$.
\end{proposition}
\begin{proof}
    The gravitino equation implies $\mathfrak R^+\diamond \xi=0$ by the holonomy principle. Similarly, the condition $F_\theta\diamond\xi=0$ implies $\bracks{\pts{F_\theta\diamond \xi}\otimes F_\theta}=0$. Therefore, by the Leibniz rule, $$\nabla^{\theta,+}F_\theta\diamond\xi=0.$$
    Clearly, $[F_\theta\diamond\xi,{}\cdot{}]=0$ and $F_\theta\diamond (F_\theta\diamond\xi) =0$ also hold, so the remaining entries of $F_{D^-_+}$ vanish under the diamond action on $\xi$. The claim follows from Lemma~\ref{lemma:eq:curvature:coupled}.
\end{proof}

The result above shows that, under the hypotheses of Theorem~\ref{theorem:ricci:flatness:simple:group}, the coupled instanton equations are satisfied and generalized Ricci flatness follows. We conclude this section by addressing the following natural question: given a solution to the coupled instanton equation, which conditions guarantee generalized Ricci flatness?

\begin{theorem}[Theorem~\ref{theorem:ricci:flatness:coupled:equations}]
\label{theorem:ricci:flatness:coupled:equations:texto}
    Let $E:(P, K; \theta, \pdi{\cdot,\cdot}_\fk)\to (M, N(G),
    \psi )$ be the string algebroid induced by a 4-form geometry for which
    \[
        \vet H_\psi^2|_{\fg}=b\cdot \id_{\fg}, \qquad b\neq 0.
    \]    
    Suppose that the \emph{coupled $G$-instanton equation} holds. Then the generalized Ricci tensor of the induced pair \eqref{eq:generalized:induced:pair} coincides with a scalar multiple of the contraction of the flux 3-form with the gravitino term:
    \[ \GRic^+\pp{V_+^{G,\psi},\dv^{G,\psi}}=-\frac1bH_\psi\iprod D^+_-\psi.
    \]
\end{theorem}

\begin{proof}
    Using the expressions for the Yang--Mills and Ricci--Bismut terms in Propositions~\ref{prop:yang:mills} and \ref{prop:ricci:bismut} applied to $\fg$ as an eigenspace, we obtain
    \begin{equation}
    \label{eq:YM+RB:coupled:temp}
    \begin{aligned}
        d^{\theta*}F_\theta-F_\theta\iprod H_\psi
        -\tfrac1b\zeta_\psi\iprod F_\theta 
        &=-\tfrac1b
        H_\psi\iprod \pp{F_\theta\diamond\psi}
        +
        \mathfrak A_b(\nabla^{\theta,+}F_\theta),
        \\
        \Ric^++F_\theta\circ F_\theta-\tfrac1b\nabla^+\zeta_\psi &=-\tfrac1bH_\psi\iprod\nabla^+\psi +\mathfrak A_b(\mathfrak R^+-\bracks{F_\theta\otimes F_\theta}).
    \end{aligned}
    \end{equation}
    Since the system satisfies the coupled $G$-instanton equation, we have
    \[
    \nabla^{\theta,+}F_\theta\in {\mathfrak g}
    \otimes \pp{T^*M\otimes \adjointbundleP}
     \qandq 
     \mathfrak R^+-\bracks{F_\theta\otimes F_\theta}\in {\mathfrak g}\otimes \Omega^2(M),
    \]
    and hence the $\mathfrak A_b$-components vanish in \eqref{eq:YM+RB:coupled:temp}.
    Finally, Theorem~\ref{theorem:explicit:formula:generalized:ricci:curvature} yields 
    \begin{align*}
\GRic^+\pp{V_+^{ G,\psi},\dv^{ G,\psi}}&=
    \pp{\Ric^++F_\theta\circ F_\theta-\tfrac1b\nabla^+\zeta_\psi}-
\bracks{
d^{\theta*}F_\theta -F_\theta\iprod H_\psi-\tfrac1b\zeta_\psi\iprod F_\theta 
,{}\cdot {}}_\fk
\\&=-\tfrac1b
H_\psi\iprod \pp{
\nabla^+\psi -
\bracks{F_\theta\diamond\psi,{}\cdot{}}_\fk
}=-\tfrac1bH_\psi\iprod D^+_-\psi
.\qedhere
    \end{align*}
\end{proof}

\smallskip

It is worth noting that, unlike Theorem~\ref{theorem:ricci:flatness:simple:group}, the conclusion of Theorem~\ref{theorem:ricci:flatness:coupled:equations} does not require $G$ to be realised as the stabiliser of a multi-tensor.

\begin{remark}
\label{eq:remark:discussion:Ric+gravitino+coupled}
    In view of Lemma~\ref{lemma:eq:curvature:coupled}, the coupled $G$-instanton condition is equivalent to
 \[
    \mathfrak R^+-\bracks{F_\theta\otimes F_\theta}\in \fg
    , 
        \qquad 
    \nabla^{\theta,+}F_\theta\in \fg
    , 
    \qquad 
[F_\theta,{}\cdot{}]+F_\theta\diamond F_\theta\in \fg
.
    \]
    The first two conditions arise, respectively, in the Yang--Mills and Ricci--Bismut equations, whereas the last condition does not seem to enter directly into the computation of the generalized Ricci tensor. Moreover, even the first two conditions do not appear in their entirety. Indeed, under the hypothesis that $\vet H_\psi^2|_{\fg}=b\ \id_{\fg}$, we have
    \begin{equation}
    \GRic^+\pp{V_+^{G,\psi},\dv^{G,\psi}}=-\frac1bH_\psi\iprod D^+_-\psi+\fA_b\pp{F^{TM}_{D^-_+}},    
    \end{equation}
    where $F^{TM}_{D^-_+}=\fR^+-\bracks{F_\theta\otimes F_\theta}+\nabla^{\theta,+}F_\theta\in \Omega^2(M,T^*M\otimes V_-)$ denotes the component of the curvature $F_{D^-_+}=F_{D^-_+}^{TM}\oplus F_{D^-_+}^{\adjointbundleP}$ taking values in $T^*M\subset V_-^*\isomorphic T^*M\oplus \adjointbundleP^*$. 
\end{remark}

\section{Examples of 4-form geometries}\label{sec:examples:4-form:geometry}

In this section, we apply the framework of 4-form geometries to the main classes of geometric structures admitting invariant 4-forms: almost Hermitian and almost special Hermitian structures, with structure groups $\U(m)$ and $\SU(m)$, respectively; \(\rG_2\)- and \(\Spin(7)\)-structures; and almost quaternionic Hermitian and almost hyper-Hermitian structures, with groups $\Sp(k)\Sp(1)$ and $\Sp(k)$, respectively. In each subsection, we begin by recalling the basics of the \(G\)-structure under consideration and showing that it admits a natural 4-form geometry. We then recall the decomposition of differential forms into irreducible \(G\)-submodules and use Theorem~\ref{theorem:flux:3:form:expression} to compute the associated flux operator, flux 3-form, and Lee form, recovering several formulas already known in the literature. Next, we explain how these structures satisfy the hypotheses of Theorem~\ref{theorem:ricci:flatness:simple:group} and thereby yield generalized Ricci-flat metrics. The main conclusion is that, for the groups
\[
    G = \SU(m),\  \rG_2,\  \Spin(7),\  \Sp(k),
\]
the associated string algebroid carries a generalized Ricci-flat metric, provided that the Bismut connection \(\nabla^+\) is compatible with the underlying \(G\)-structure and \(\theta\) is a \(G\)-instanton satisfying the heterotic Bianchi identity; see Theorem~\ref{theorem:ricci:flatness:compilation}. Finally, we present explicit examples of \(G\)-structures endowed with generalized Ricci-flat metrics. We discuss the torsion-free regime in all these cases, nearly Kähler and nearly parallel structures on the 6- and 7-spheres, lifts obtained by considering Cartesian products \(M \times \bb R\) or \(M \times \bb S^1\), and structures with closed and parallel torsion, satisfying \(dH=0\) and \(\nabla^+H=0\). In particular, in \ref{parag:lifting:G2:to:spin7} we study a lift of a $\rG_2$-structure to a $\Spin(7)$-structure using the language of 4-form geometries; this can be viewed as a generalisation of \cite[Theorem~5.1]{Ivanov2005}. In \ref{parag:su(4):structures}, we discuss \(\SU(4)\)-structures equipped with an alternative 4-form, namely the real part $\Psi_+$ of the complex volume form, instead of the usual choice \(\psi=\frac12\omega^2\). This alternative does not yield a suitable 4-form geometry for our construction of special generalized geometries and therefore illustrates the necessity of the hypotheses in our result.

\subsection{\texorpdfstring{$\U(m)$}{U(m)} and \texorpdfstring{$\SU(m)$}{SU(m)}-structures}
\label{subsec:U(m):and:SU(m)}

We focus on $m\ge 3$, since the lower-dimensional cases have atypical module decompositions. When $m=1$, there are no non-zero 4-forms. When $m=2$, the forms $d\omega$ and $d^*\psi$ do not encode the same information; moreover, this case has already been treated separately; see Example~\ref{example:dimensao:quatro} and Corollary~\ref{corollary:GRIC:dimensao:quatro}. Accordingly, let $M^{2m}$ be a smooth manifold with $m\ge 3$.

A \emph{$\U(m)$-structure} on $M$ is a pair $(g,\omega)$, where $g$ is a Riemannian metric on $M$ and $\omega\in \Omega^2(M)$ is pointwise modelled on
    \begin{equation*}
    \omega_0 = e^{12} + e^{34}+ \cdots +e^{2m-1}\wedge e^{2m}\in \Lambda^2(\bb R^{2m})^* , 
\end{equation*}
for an orthonormal frame $\{e_1, \cdots, e_{2m}\}$ of $\R^{2m}$ with associated metric $g_0=e^1\otimes e^1+\cdots+e^{2m}\otimes e^{2m}$. Alternatively, a $\U(m)$-structure is a pair $(g,J)$, where $J$ is a $g$-orthogonal almost complex structure; the two definitions are equivalent by setting 
\(
\omega({}\cdot{},{}\cdot{})=g(J{}\cdot{},{}\cdot{})
\). A $\U(m)$-structure induces an orientation \(\vol_M:=\frac{1}{m!}\omega^m\), and the group $\U(m)\subset \SO(2m)$ can be viewed as the $\SO(2m)$-stabiliser
\begin{equation*}
    \U(m ) =\Stab(\omega_0)\cap\Stab(g_0)
    = \Stab{}_{\SO(2m)}(\omega_0).
\end{equation*}
Thus, we consider the 4-form $\psi= \tfrac{1}{2} \omega \wedge \omega\in \Omega^4(M)$ (the factor $\frac{1}{2}$ will be useful later), which is pointwise modelled on
\begin{equation*}
\psi_0 
=\sum_{i<j}^me^{2i-1,2i}\wedge e^{2j-1,2j}  \in \Lambda^4(\bb R^{2m})^*.
\end{equation*}
It follows that $\U(m)\subset \Stab_{\SO(2m)}(\psi_0)$. Thus, $(M,\U(m),\psi)$ defines a 4-form geometry; see Definition~\ref{def:4:form:geometry}.

\begin{remark}
    In the case of $\U(m)$-structures with $m\ge 3$, $d^*\psi$ contains the same information as $d\omega$, since
\begin{align*}
d^*\psi &=-* d* \frac{\omega^2}{2!}=-* d \frac{\omega^{m-2}}{(m-2)!}=-\frac{1}{(m-3)!}* \pts{d\omega \wedge \omega^{m-3}}
\end{align*}
    and the assertion follows from the fact that the Lefschetz operator wedging with $\omega$ is injective on 3-forms for $m\ge 4$. The assertion for $m=3$ is immediate, since $d^*\psi=-\star d\omega $.
\end{remark}

An \emph{$\SU(m)$-structure} on $M^{2m}$ consists of a triple $(g, \omega, \Psi)$, where $(g,\omega)$ is a $\U(m)$-structure and $\Psi\in \Omega^{(m,0)}(M)$ is a complex volume form pointwise modelled on
 \begin{equation}\label{eq:model:Psi}
\Psi_0 = z^1 \wedge \cdots \wedge z^m =
\sum_{S \subset [m]} i^{\#S} e^{j_1} \wedge \cdots \wedge e^{j_m}, 
\qquad 
j_k = \begin{cases}
    2k & k \in S, \\
    2k-1 & k \notin S , 
\end{cases}
\end{equation}
with $z^k = e^{2k-1} + i e^{2k}$. We denote by $\Psi_+$ and $\Psi_-$ the real and imaginary parts of $\Psi=\Psi_+ + i\Psi_-$, respectively. We refer the reader to \cite{Chiossi2002,Salamon1989} for details.

\begin{remark}
The complex volume form $\Psi$ determines the Riemannian volume of $M$ through the normalisation condition
\begin{equation*}
    \Psi\wedge\fecho{\Psi} 
    = \dfrac{1}{m!}i^{m}(-1)^{{m(m+1)}/{2}}\ 2^m\omega^{m}
    =i^{m}(-1)^{{m(m+1)}/{2}}\ 2^m\vol_M .
\end{equation*}
One may also recover the metric from the forms $\omega$ and $\Psi$  via 
\[
g(X,Y)=\frac{(-1)^m}{2^{m-2}\cdot \vol_M} (X\iprod\omega )\wedge (Y\iprod\Psi_+)\wedge \Psi_{[m]},
\]
since $
    \SU(m) = \Stab(\omega_0,\Psi_0)$. Here $[m]\coloneqq (-1)^{m+1}\in \{+1, -1\}$.
\end{remark}
    
It follows from the previous remark that $\SU(m)\subset \Stab_{\SO(2m)}(\psi_0)$, where $\psi_0=\frac{1}{2}\omega_0^2$, so that $(M, \SU(m),\psi)$ defines a 4-form geometry; see Definition~\ref{def:4:form:geometry}. When $m=4$, the forms $\Psi_+$ and $\Psi_-$ define alternative 4-form geometries $(M^8,\SU(4),\Psi_\pm)$. We briefly explore $(M^8,\SU(4),\Psi_+)$ in \ref{parag:su(4):structures}.
More generally, a $G$-structure can admit more than one 4-form geometry, and it is natural to ask how different choices affect the resulting generalized geometry or whether the corresponding flux 3-forms are meaningfully related. We are currently investigating this question.

\paragraph{Decomposition of differential forms}
\label{subsec:Decomposition:forms:Un}
\label{parag:dec:forms:u:m}
Let $(M^{2m},g,\omega)$ be a $\U(m)$-structure with $m\ge 3$. Following \cite{Salamon1989,Cabrera2005}, $\Omega^k(M)$ decomposes into irreducible $\U(m)$-submodules. If a complex representation $V$ admits a real form, we denote that real form by $[V]$, so that $[V]\tensor_\R\C\iso V$. If $V$ is not self-conjugate, we denote by $\bracl{V}$ its underlying real representation, characterised by $\bracl{V}\tensor_\R\C\iso V\oplus\fecho{V}$. The spaces of 0-forms and 1-forms are irreducible as $\U(m)$-modules, whereas the 2-forms decompose as\footnote{Throughout the text, the subscript denotes the rank of the module, i.e. $\rank\pp{\Omega^k_l}=l$.}
        \begin{equation*}
        \begin{split}
            \Omega^2
            &= \Omega^2_{m(m-1)}\oplus \Omega^2_{m^2-1}\oplus \Omega^2_1 
            \iso \bracl{\Omega^{2,0}}\oplus [\Omega_0^{1,1}]\oplus \pdi{\omega},
        \end{split}
        \end{equation*}
        where $\mathfrak{su}(m)\isomorphic [\Omega^{1,1}_0]$. These components are characterised by
    \begin{align*}
    \Omega^2_1&
    =\{
    \beta\in \Omega^2:\beta\iprod \psi=(m-1)\beta 
    \}
    =\{f\omega:f\in C^\infty(M)\},
    \\
     \bracl{\Omega^{2,0}}&=
     \{
     \beta\in \Omega^2:\beta\iprod\psi=\beta 
     \}
     =\{
     \beta\in \Omega^2:J\beta=-\beta 
     \},
     \\
     [\Omega^{1,1}_0]&=
     \{
     \beta\in \Omega^2:\beta\iprod \psi=-\beta 
     \}=
     \{
     \beta\in \Omega^2
     :J\beta=\beta , \beta\iprod\omega=0
     \}.
    \end{align*}
    The space of 3-forms decomposes as
    \begin{align*}
        \Omega^3
        &= \Omega^3_{2m}\oplus \Omega^3_{\frac{1}{3}m(m-1)(m-2)}\oplus \Omega^3_{m(m+1)(m-2)} 
        \iso \Omega^1\oplus \bracl{\Omega^{3,0}}\oplus \bracl{\Omega^{2,1}_0} , 
    \end{align*}
    where the identification $\Omega^3_{2m}\iso \Omega^1$ is given by $\Omega^3_{2m}
        = \{X \iprod \psi : X \in \Omega^1\}$
     and the sum $\bracl{\Omega^{3,0}} \oplus \bracl{\Omega^{2,1}_0}$ consists of 3-forms $\gamma$ satisfying $\omega \iprod \gamma = 0$. Moreover, these decompositions remain valid as decompositions into irreducible $\SU(m)$-modules, with the exception of $m=3$, for which the 2-dimensional module $\bracl{\Omega^{3,0}}$ decomposes into two irreducible components $\Omega^3_2=\langle\Psi_+\rangle\oplus\langle\Psi_-\rangle$.

\paragraph{Flux operator and associated forms}
Consider the 4-form geometry $\pp{M^{2m},\ \U(m),\ \psi}$, where $\psi=\frac{1}{2}\omega^2$. According to the decomposition of $\Omega^3$ into irreducible $\U(m)$-submodules, the codifferential of $\psi$ decomposes as
\begin{equation}
\label{eq:torsion:forms:u:m}
d^*\psi  = \tau_1\iprod  \psi + 
\tau_3^{3,0}+\tau_3^{2,1},
\end{equation}
where $\tau_1\in \Omega^1(M)$, $\tau_3^{2,1}\in \bracl{\Omega^{2,1}_0}$ and $\tau_3^{3,0}\in \bracl{\Omega^{3,0}}$ are the \emph{torsion forms}. For $\U(m)$-structures, the Lee space $\Omega^1=\Omega^1_L$ is irreducible and the eigenvalue of $\mathfrak{su}(m)$ on the flux operator $\bH_\psi^2$ is $b=-1$, cf. \ref{subsec:Decomposition:forms:Un}.

\begin{proposition}
\label{prop:flux:operator:almost:hermitian}
    Let $(M^{2m},\U(m),\psi)$ be a 4-form geometry, where $\psi=\frac{1}{2}\omega^2$. Then the flux operator $\vet H_\psi:\Omega^3\to \Omega^3$ is an isomorphism of $\U(m)$-modules and acts on each component as
\begin{equation*}
    \vet H_\psi\big|_{\Omega^3_{2m}}=(m-2)\cdot \id_{\Omega^3_{2m}},  \qquad 
\vet H_\psi\big|_{\bracl{\Omega^{3,0}}}=3\cdot \id_{\bracl{\Omega^{3,0}}},  \qquad 
\vet H_\psi\big|_{\bracl{\Omega^{2,1}_0}}=- \id_{\bracl{\Omega^{2,1}_0}} .
\end{equation*}
In particular, the associated flux 3-form $H_\omega\in \Omega^3(M)$ and Lee form $\zeta_\omega\in \Omega^1(M)$ are given, respectively, by
\begin{equation*}
H_\omega=\frac{1}{m-2}\tau_1\iprod\psi +\frac{1}{3}\tau_3^{3,0}-\tau_3^{2,1} \qandq
        \zeta_\omega 
        =-\frac{m-1}{m-2}\tau_1 .
    \end{equation*}
\end{proposition}

\begin{proof}
For the component $\Omega^3_{2m}$, whose elements have the form $X \wedge \omega$, one has
\begin{align*}
\mathbf H_\psi(X \wedge \omega)
&= (X^\flat \wedge \omega)\iprod^2 \psi
= \frac{1}{2!\,2!} \left( X_j \omega_{ka} + X_a \omega_{jk} + X_k \omega_{aj} \right) \psi_{jkbc}\, e^{abc} \\
&= \frac{1}{4} \left( X_j \omega_{ka} + X_a \omega_{jk} + X_k \omega_{aj} \right) \big(\omega_{jk} \omega_{bc} - \omega_{jb} \omega_{kc} + \omega_{kb} \omega_{jc}\big) e^{abc}.
\end{align*}
After a straightforward computation using $\omega^{\mu \nu}\omega_{\mu\nu}=2m$ and $
    \omega^{\mu i}\omega_{\mu a} = \delta^i{}_a$, we obtain
\[
\mathbf H_\psi(X \wedge \omega)=(m - 2)\, X^\flat \wedge \omega.
\]
Next, for $\leftbrack\Omega^{3,0}\rightbrack$, consider $\gamma=\Re z^{123}=e^{135} - e^{146}  - e^{236} - e^{245}$, which is an element of this space for every $m\ge 3$. Hence
$$
\gamma\iprod^2\psi =\sum_{i<j}\gamma_{ij}\wedge \psi_{ij}=
\sum_{i<j\le 6}\gamma_{ij}\wedge \psi_{ij},
$$ 
so it is enough to perform the computation for $m=3$. The non-zero $\gamma_{ij}$ (for $i<j$) are
\[
\hspace*{\fill}
\begin{array}{@{}r@{\quad}r@{\quad}r@{\quad}r@{\quad}r@{\quad}r@{}}
\gamma_{13} = e^5, & 
\gamma_{14} = -e^6, & 
\gamma_{15} = -e^3, & 
\gamma_{16} = e^4, &
\gamma_{23} = - e^6, & 
\gamma_{24} = -e^5, \\
\gamma_{25} = e^4, & 
\gamma_{26} = e^3, &
\gamma_{35} = e^1, & 
\gamma_{36} = -e^2, & 
\gamma_{45} =- e^2, & 
\gamma_{46} = -e^1.
\end{array}
\]
The corresponding components of $\psi$ are
\[
\begin{matrix}
\psi_{13} = -e^{24}, & \psi_{14} = e^{23}, & \psi_{15} = -e^{26}, & \psi_{16} = e^{25}, \\
\psi_{23} = e^{14}, & \psi_{24} = -e^{13}, & \psi_{25} = e^{16}, & \psi_{26} = -e^{15}, \\
\psi_{35} = -e^{46}, & \psi_{36} = e^{45}, & \psi_{45} = e^{36}, & \psi_{46} = -e^{35}.
\end{matrix}
\]
Taking the wedge products of the corresponding terms and summing gives 
\(
\mathbf H_\psi(\gamma)= 3\gamma.
\)
Finally, on the component $\leftbrack\Omega^{2,1}_0\rightbrack$, consider $\tilde{\gamma} = e^{135}-e^{146}+e^{236}+e^{245}$. Proceeding analogously and changing only some signs relative to $\gamma$, the same calculation gives $\mathbf H_\psi(\tilde{\gamma}) = -\tilde{\gamma}$. This proves the result.
The expression for $H_\omega$ follows directly from Theorem~\ref{theorem:flux:3:form:expression}. For the Lee form, we use the contraction identity $ {\psi^{ijk}}_l\psi_{ijk\mu}= 6(m-1)\,\delta_{l\mu}$ for $\psi=\frac{1}{2}\omega^2$, and then compute
\begin{align*}
        \zeta_\omega 
        &= H_\omega \iprod\psi =
        \frac{1}{m-2}\pts{\tau_1\iprod \psi}\iprod\psi
        =\frac{1}{(m-2)3!}\tau_1^\mu{\psi_{\mu}}^{ijk}\psi_{ijk\nu}e^\nu \\
        &= \frac{-6(m-1)}{6(m-2)}\tau_1^\mu\delta_{\mu\nu} e^\nu
        = -\frac{m-1}{m-2}\tau_1. \qedhere
    \end{align*}
\end{proof}

Note that the flux operator depends only on the $\U(m)$-structure. Therefore, its action remains unchanged for 4-form geometries endowed with an $\SU(m)$-structure. In particular, for 
\[
(M^{6},\SU(3),\psi=\tfrac12\omega^2),
\]
although the module
\(\Omega^3_2=\bracks{\Psi_+}\oplus\bracks{\Psi_-}
\)
is reducible, it still coincides with the eigenspace corresponding to the eigenvalue $3$. Still in the case $m=3$, the authors of \cite[\S 4.1]{Alexandrov2005} expressed the flux 3-form in terms of torsion forms. Proposition~\ref{prop:flux:operator:almost:hermitian} therefore provides an alternative and more general description of this quantity in arbitrary dimensions.

\medskip

In 1943, Hwa-Chung Lee \cite{Lee_Hwa-Chung_lee_form} introduced the notion now referred to as the \emph{Lee form} in almost Hermitian geometry. This notion coincides with our definition of the Lee form for $\U(m)$-structures; see Definition~\ref{def:lee:form}. The Lee form $\vtheta_\omega\in \Omega^1(M)$ is defined by
\begin{equation*}
    \vtheta_\omega 
    \defeq Jd^*\omega
    = -d^*\omega(J\cdot{}) ,
\end{equation*}
or, equivalently, it is implicitly defined by $
    d\omega^{m-1} = \vtheta_\omega\wedge \omega^{m-1}$. Using $X\iprod\omega=JX$, we can compute $\omega\iprod d^*\psi$ in two ways and conclude that $\vtheta_\omega=\zeta_\omega$, as follows:
\begin{align*}
    \omega\iprod d^*\psi &=\omega\iprod (\tau_1\iprod \psi )=\tau_1\iprod (\omega\iprod \psi )
    =(m-1)\tau_1\iprod \omega 
    =(m-1)J\tau_1.
\\
    \omega\iprod d^*\psi &=-
    \star \pp{\omega \wedge d\star \psi }=-\frac{1}{(m-2)!}\star \pp{
    \omega\wedge d\omega^{m-2}
    }
    =-(m-2)J\vtheta_\omega .
\end{align*}

\begin{remark}
The intrinsic torsion of a $\U(m)$-structure was first studied by Gray--Hervella \cite{Gray1980}, who classified almost Hermitian structures into sixteen classes according to the presence or absence of four irreducible $\U(m)$-modules:
\[
\mcal W_1,\quad 
\mcal W_2,\quad 
\mcal W_3,\quad 
\mcal W_4.
\]
Three of these components can be identified directly in terms of the torsion forms introduced above:
$ \tau_1\in \mcal W_4$, $\tau_3^{3,0} \in \mcal W_1$ and $
\tau^{2,1}_3\in \mcal W_3$. The component $\tau_3^{3,0}$ is related to the Nijenhuis tensor $N_\omega$ by
\[
\tau_3^{3,0}(X,Y,Z)=-3\  \mathrm{Skew}(N_\omega)(JX,Y,Z).
\]
Consequently, using the fact that $\cW_1\oplus\cW_2$ is determined by $N_\omega$, we recover the classical result that there exists a Hermitian connection with skew-symmetric torsion compatible with the $\U(m)$-structure if and only if $N_\omega$ is totally skew-symmetric. In this case, the connection is unique \cite[Theorem 10.1]{Friedrich2003b}, and its torsion 3-form, which coincides with our flux 3-form $H_\omega$, is normally written as $H_B=-d^c\omega + N_\omega$, where $d^c\omega\defeq Jd\omega$.
\end{remark}

\paragraph{String algebroids over $\U(m)$-structures}
\label{parag:string:algebroids:u:m}
Finally, we discuss generalized Ricci flatness for string algebroids over Hermitian structures. For this, let us recall the notion of \emph{primitive Hermitian Yang--Mills} connections $\theta$, namely those satisfying
\[
F_\theta^{(0,2)+(2,0)}=0, \qquad 
F_\theta\wedge\omega^{m-1}=0.
\]
Note that primitive Hermitian Yang--Mills connections are precisely the $\SU(m)$-instantons. Connections that satisfy only the first condition above are called \emph{Hermitian}.

\begin{proposition}[{\cite{Garcia-Fernandez2017,Garcia-Fernandez2019,garciafernandez2024pluriclosedflowhullstrominger}}]
\label{prop:ricci:flatness:U(m)}
Let $(M^{2m},\omega,\Psi)$ be an $\SU(m)$-structure, with $m\geq 3$, whose Nijenhuis tensor $N_\omega$ is totally skew-symmetric and for which $\nabla^+\Psi=0$. Consider the string algebroid
$$
(P,K,\theta,\pdi{\cdot,\cdot}_\fk)\to \pp{M,\SU(m),\psi=\tfrac{1}{2}\omega^2}.
$$
If $\theta$ is a primitive Hermitian Yang--Mills connection, then the generalized Ricci tensor of the pair \eqref{eq:generalized:induced:pair} vanishes,
    \[
    \GRic^+\pp{V_+^{\SU(m),\psi}, \dv^{\SU(m),\psi}}=0,
    \]
    where the divergence is 
    \begin{equation}
    \label{eq:divergence:u:m}
    \dv^{\SU(m),\psi}=\dv^{V_+^{\SU(m),\psi}}-2\bracks{-\frac{m-1}{m-2}\tau_1,\cdot}.
    \end{equation}
\end{proposition}

\begin{proof}
As mentioned above, a connection $\theta$ is primitive Hermitian Yang--Mills if and only if it is an $\SU(m)$-instanton. Indeed, $F_\theta^{(2,0)+(0,2)}=0$ is equivalent to $F_\theta\in\bracks{\omega}\oplus \mathfrak{su}(m)$, while $F_\theta\wedge\omega^{m-1}=0$ is equivalent to $F_\theta\in \bracks{\omega}^\perp$. Since $\SU(m)=\Stab_{\SO(2m)}(\psi)\cap\Stab_{\SO(2m)}(\Psi)$, it follows that
\begin{equation*}
    F_\theta\in\su(m)
    \quad\Longleftrightarrow\quad
    F_\theta\diamond\psi=0
    \qandq
    F_\theta\diamond\Psi=0 .
\end{equation*}
The condition that $N_\omega$ be totally skew-symmetric implies $\nabla^+\omega=0$ \cite[Theorem 10.1]{Friedrich2003b}; consequently, $\nabla^+\psi=0$.
Together, $\nabla^+\psi=0$ and $F_\theta\diamond\psi=0$ are equivalent to $D^+_-\psi=0$. Similarly, $\nabla^+\Psi=0$ and $F_\theta\diamond\Psi=0$ are equivalent to $D^+_-\Psi=0$. Thus, the 4-form geometry $(M,\SU(m),\psi)$, for which $b=-1$, satisfies the hypotheses of Theorem~\ref{theorem:ricci:flatness:simple:group}. 
\end{proof}

In general, string algebroids over arbitrary $\U(m)$-structures do not yield examples of generalized Ricci-flat metrics even under the hypothesis $D^+_-\omega=0$ \cite{Garcia-Fernandez2025}. The hypothesis of Theorem~\ref{theorem:ricci:flatness:simple:group} fails because $\fu(m)=\fs\fu(m)\oplus \bracks{\omega}$ is not an eigenspace of the flux operator $\vet H_\psi^2$; see \ref{parag:dec:forms:u:m}. However, under suitable assumptions, $\U(m)$-structures that satisfy the coupled
$\SU(m)$-instanton equation (which make sense since $N(\SU(m))=\U(m)$ inside $\SO(2m)$, see \ref{parag:instantons}) give rise to generalized Ricci-flat metrics.

\begin{proposition}[{\cite[Proposition~5.6]{Garcia-Fernandez2025}}]
\label{prop:coupled=>GRic=0}
    Let $(M^{2m},g,\omega)$ be a Hermitian manifold, with $m\geq 3$, and
    consider a string algebroid
    \[
        (P,K,\theta,\pdi{\cdot,\cdot}_\fk)
        \to
        \pp{M,\U(m),\psi=\tfrac{1}{2}\omega^2}
    \]
    satisfying the coupled $\SU(m)$-instanton condition. If $\theta$ is
    a Hermitian connection, then the generalized Ricci tensor vanishes, \(\GRic^+\pp{V_+^{g,\omega},\dv^{g,\omega}}=0\), where the divergence $\dv^{g,\omega}$ is given by \eqref{eq:divergence:u:m}.
\end{proposition}

\begin{proof}
    By Theorem~\ref{theorem:ricci:flatness:coupled:equations}, it is
    enough to show that
    \(
        H_\omega\iprod D^+_-\psi=0.
    \)
    In fact, in the present setting, one has $D^+_-\psi=0$. Indeed,
    since $(g,\omega)$ is Hermitian, $N_\omega=0$, and the Bismut
    connection $\nabla^+$ preserves the Hermitian structure; hence, $\nabla^+\psi=0$. On the other hand, the assumption that $\theta$ is Hermitian gives
    \[
        F_\theta\in\fu(m)
        \qquad\Longrightarrow\qquad
        F_\theta\diamond\psi=0.\qedhere
    \]
\end{proof}

\paragraph{Explicit examples}

We now present some explicit examples of generalized Ricci-flat metrics on string algebroids over Hermitian manifolds.

\begin{example}[Kähler manifolds]
\label{example:kahler}
    If $(M^{2m},g,\omega)$ is a Kähler manifold, i.e., $\nabla^g\omega=0$, then, in particular, $H_\omega=0$. We can therefore consider $P=M\times \{1\}$, the trivial principal bundle with trivial structure group $K=\{1\}$. In this case, the condition that the string algebroid $TM\oplus T^*M$ (usually referred to as an \emph{exact Courant algebroid} when there is no $\adjointbundleP$ component) is generalized Ricci-flat with Riemannian divergence $\dv=\dv^{V_+}$ reduces to
    \[
    \Ric^g=0,
    \]
    so the metric is Ricci-flat Kähler and hence locally Calabi--Yau; see \cite[Proposition~6.1.1]{Joyce2000}.
    \end{example}

\begin{remark}
    Consider the string algebroid
    $E:(P,K;\theta,\bracks{\cdot,\cdot}_\fk)\to
    (M^{2m},\U(m),\psi)$ defined in
    \S\ref{subsec:U(m):and:SU(m)}. Suppose that $M$ is Kähler and that
    $\bracks{\cdot,\cdot}_\fk$ is definite. Since $H_\omega=0$, the heterotic
    Bianchi identity reduces to
    \begin{equation*}
        0=\bracks{F_\theta\wedge F_\theta}_\fk.
    \end{equation*}
    If $\theta$ is an $\SU(m)$-instanton, then
    $F_\theta\iprod\psi=-F_\theta$, or equivalently
    $F_\theta\wedge\star\psi=-\star F_\theta$. Therefore,
    \[
    0=
    \bracks{F_\theta\wedge F_\theta}_\fk\wedge\star\psi
    =
    -\bracks{F_\theta\wedge\star F_\theta}_\fk
    =-\bracks{F_\theta\iprod F_\theta}_\fk\ \vol_M.
    \]
    The definiteness of the pairing implies that $F_\theta=0$. Hence, every
    $\SU(m)$-instanton satisfying the heterotic Bianchi identity is flat in the Kähler case.
    The same argument applies to torsion-free $\rG_2$-, $\Spin(7)$-, and
    hyper-K\"ahler structures, with their corresponding instanton equations.
\end{remark}

Let $(M^6,g,\omega)$ be a strict nearly-Kähler manifold. Its $\U(3)$-structure satisfies $\nabla^g_XJ(X)=0$ for every vector field $X$. It carries an $\SU(3)$-structure \cite[Corollary 10.6]{Friedrich2003b} characterised by $d\omega=3\lambda\Psi_+$, where the non-zero constant $\lambda$ is used for normalisation. In this setting, we have \cite[Corollary 10.3]{Friedrich2003b}
\[
    d\Psi_-=-2\lambda\,\omega\wedge\omega,
    \qquad
    H_\omega=-\lambda\Psi_-=\frac14N_\omega,
    \qquad
    \nabla^+H_\omega=0.
\]
In particular, the Lee form vanishes, $\zeta_\omega=0$. Furthermore,
using
\(
    (\Psi_-)_{abi}(\Psi_-)^{ab}{}_{j}=4g_{ij}
\), we obtain
\[
    H_\omega^2
    :=
    H_{abi}H^{ab}{}_{j}\,e^i\otimes e^j
    =
    \lambda^2(\Psi_-)_{abi}(\Psi_-)^{ab}{}_{j}\,e^i\otimes e^j
    =
    4\lambda^2g.
\]
Using 
\(
    R^+=R^g+d^{\nabla^g}H_\omega+[H_\omega,H_\omega],
\)
and 
\(
    \nabla^+_XH_\omega
    =
    \nabla^g_XH_\omega+\frac12\, i_XH_\omega\diamond H_\omega,
\)
together with the fact that $\nabla^+H_\omega=0$ in the nearly-Kähler case, we obtain \cite[\S6]{Ivanov2005}
\[
R^+_{ijkl}
=
R^g_{ijkl}
+\frac12H_{ij\mu}H_{kl}{}^\mu
+\frac14H_{jk\mu}H_{il}{}^\mu
+\frac14H_{ki\mu}H_{jl}{}^\mu.
\]
Moreover, $\nabla^+$ is an $\SU(3)$-instanton, since $
    R^+\in \Sigma^2(\mathfrak{su}(3))$. It is well known that nearly-Kähler 6-manifolds are also Einstein and satisfy
\[
    \Ric^g=\frac54H_\omega^2=5\lambda^2g,
    \qquad
    \rmm{Scal}^g=30\lambda^2.
\]
There are four homogeneous strict nearly-Kähler 6-manifolds: the sphere $\bb S^6$, $\bb S^3\times \bb S^3$, $\bb CP^3$ and $F(1,2)=\SU(3)/\bb T^2$. Among these homogeneous examples, $\bb S^6$ is the only one with constant sectional curvature.

\begin{example}[Nearly-Kähler 6-sphere, { \cite[\S6.1]{Ivanov2005} }, {\cite{friedrich2005nearlykaehler_sphere}}]
\label{example:sphereS6}
   Let $(\bb S^6, g, J)$ be the usual nearly-Kähler structure on the 6-sphere, as described above, with constant sectional curvature $\rmm{sec}=1$. Then 
   \(R^g_{ijkl}={g_{jk}g_{il}-g_{ik}g_{jl}}\). Using the expression for $R^+$ on nearly-Kähler manifolds above, together with $dH_\omega=H_\omega\diamond H_\omega$, we obtain
\[
\rmmtr\pp{R^+\wedge R^+}=-\frac32dH_\omega.
\]
Let $P=\rmm{Fr}_{\SO(6)}(TM)\to M$ be the orthonormal frame bundle of $M$, viewed as a principal $\SO(6)$-bundle, and equip its Lie algebra $\fk=\fs\fo(6)$ with the pairing
\[
\bracks{A, B}_{\fk}:=-\frac23\rmmtr\pp{A\circ B}.
\]
Moreover, consider the induced Bismut connection $\nabla^+$ on $P$, which by construction satisfies $dH_\omega=\bracks{R^+\wedge R^+}_\fk$. Thus, $(H_\omega,\nabla^+)$ induces a string algebroid structure on the bundle $\transitive$. Since $\nabla^+\omega=0=\nabla^+\Psi_+$ and $\nabla^+$ is an $\SU(3)$-instanton, Proposition~\ref{prop:ricci:flatness:U(m)} yields a generalized Ricci-flat metric with divergence $\dv=\dv^{V_+}$.
\end{example}

\begin{example}[{\cite[Example~7.12]{kennon2026}}]
Let $G=\SU(2)\times\SU(2)\times\U(1)$, and let
\(\{e^1,\ldots,e^7\}\) be a left-invariant coframe satisfying
\begin{align*}
    de^1=e^{23}, \quad 
    de^2=e^{31}, \quad 
    de^3=e^{12},& \\
    de^4=e^{56}, \quad 
    de^5=e^{64}, \quad 
    de^6=e^{45} \qandq&
    de^7=0.
\end{align*}
Let \(\{e_1,\ldots,e_7\}\) be the dual frame and set $V=e_4-e_3$.
The one-parameter subgroup generated by $V$ defines a free circle action and
hence a principal $\U(1)$-bundle
\[
P=G\to M^6=\bb S^3\times\bb S^2\times\bb S^1.
\]
The 1-form \(\eta=\tfrac1{\sqrt2}V^\flat=\tfrac1{\sqrt2}(e^4-e^3)\)
defines a $\U(1)$-connection on this bundle. Consider the $\SU(3)$-structure on
$M^6$ given by
\begin{align*}
    \omega &=e^{16}+e^{25}-e^{37}+e^{15}-e^{26}-e^{47},\\
    \Psi_+&=e^{127} -e^{146}-e^{245}+e^{567}-e^{136}-e^{235}.
\end{align*}
A direct computation gives
\begin{align*}
H_\omega&=\frac12\pp{e^{123}+e^{124}+e^{356}+e^{456}}
\\
dH_\omega&=e^{1256}
=-\frac12(e^{56}-e^{12})\wedge(e^{56}-e^{12})
=-d\eta\wedge d\eta.
\end{align*}
Moreover, $\nabla^+\omega=0=\nabla^+\Psi_+$. Equip
$\fu(1)\cong\bb R$ with the pairing $\bracks{x,y}_{\fu(1)}=-xy$.
The identity above shows that $(H_\omega,\eta)$ satisfies the heterotic
Bianchi identity and therefore defines a string algebroid
\[
E=TM\oplus\underline{\bb R}\oplus T^*M.
\]
It remains to verify the instanton condition. First,
\begin{align*}
\bracks{d\eta,\omega}
&=\frac1{\sqrt2}
\bracks{e^{56}-e^{12},e^{16}+e^{25}-e^{37}+e^{15}-e^{26}-e^{47}}
=0,
\end{align*}
and hence $
d\eta\in[\Lambda^{1,1}_0]\oplus\bracl{\Lambda^{2,0}}$. Furthermore,
\begin{align*}
d\eta\iprod\Psi_+
&=\frac1{\sqrt2}(e^{56}-e^{12})\iprod
\pp{e^{127}-e^{146}-e^{245}+e^{567}-e^{136}-e^{235}} = \frac1{\sqrt2}(e^7-e^7)=0.
\end{align*}
Therefore, $d\eta\in\bracks{\Lambda^1\iprod\Psi_+}^\perp =[\Lambda^{1,1}_0]\oplus\bracks{\omega}$. By intersecting $d\eta\in[\Lambda^{1,1}_0]=\fs\fu(3)$, so $\eta$ is an $\SU(3)$-instanton. Proposition~\ref{prop:ricci:flatness:U(m)} now implies
\[
\GRic^+\pp{V_+^{\SU(3),\psi},\dv^{V_+^{\SU(3),\psi}}}=0.
\]
The divergence is the Riemannian one because $d\omega^2=2d\omega\wedge\omega=0$, and hence the Lee form
$\zeta_\omega=0$. 
\end{example}

Solutions satisfying the hypotheses of
Proposition~\ref{prop:ricci:flatness:U(m)} are called solutions of the \emph{Hull--Strominger system}. They were first studied in the physics literature in the context of supergravity, and were later introduced into the mathematical literature in the seminal paper of Li--Yau \cite{LiYau2006}. In their setting, $P$ is the bundle of split frames, so that
\[
\adjointbundle{P}
=
\rmm{End}(TM)\oplus \rmm{End}(E'),
\]
where $E'\to M$ is an auxiliary vector bundle and $M$ is typically a complex threefold. Moreover, $\theta$ is the direct sum of two Hermitian connections: a connection $\nabla$ on $TM$ and a connection $A$ on $E'$. The pairing $\bracks{\cdot,\cdot}_\fk$ is given by the difference between the traces, so that, for a constant $\alpha'$, the heterotic Bianchi identity takes the form
\[
dd^c\omega = \frac{\alpha'}{4} \pp{ \rmmtr\pp{ R_\nabla\wedge R_\nabla} - \rmmtr\pp{F_A\wedge F_A} }.
\]

Example~\ref{example:sphereS6} above is simpler than the physically motivated construction, since there is no auxiliary vector bundle. The original example of Ivanov--Ivanov considered two bundles, with $\nabla^-$ appearing in the second component, but it does not satisfy the Hull--Strominger system in the strict sense, as the almost complex structure is not integrable; see \cite[\S6]{Ivanov2005}. We therefore adapted the example to obtain a generalized Ricci-flat solution.

Examples of the Hull--Strominger system involving both bundles, and hence closer to the physical setting, can be found in Li--Yau's paper \cite[\S5]{LiYau2006}, which uses techniques from complex and algebraic geometry, as well as in the lecture notes by Garcia-Fernández \cite{Garcia-Fernandez2016}. Solutions with torus symmetry over K3 surfaces and orbifolds were constructed by Fu--Yau \cite{Fu2008} and Fino--Grantcharov--Vezzoni \cite{Fino2021}, respectively.

\paragraph{\texorpdfstring{$\SU(4)$}{SU(4)}-structures}
\label{parag:su(4):structures}
Here, we present a 4-form geometry whose flux operator has a non-trivial kernel and whose decomposition of $\Omega^3$ contains two isomorphic irreducible components, namely $(M^8,\SU(4), \Psi_+)$. Recall that an $\SU(4)$-structure in eight dimensions is a tuple $(M^8,g,\omega,\Psi)$ whose model tensors are given by
\begin{equation*}
\begin{aligned}
\omega &= e^{12}+e^{34}+e^{56}+e^{78},
\\\psi&=e^{1234}+e^{1256}+e^{1278}+e^{3456}+e^{3478}+e^{5678}
, \\
\Psi_{+} &= e^{1357} - e^{1368} - e^{1458} - e^{1467} - e^{2358} - e^{2367} - e^{2457} + e^{2468}, \\
\Psi_{-} &= 
e^{1358} + e^{1367} + e^{1457} + e^{2357}
- e^{1468} - e^{2368} -e^{2458} - e^{2467}.
\end{aligned}
\end{equation*}
Here, $\psi=\frac{1}{2}\omega^2$ and $\Psi = \Psi_++i\Psi_-$.
The space of 2- and 3-forms decomposes into irreducible $\SU(4)$-submodules as
\begin{align*}
\Omega^2&=\leftbrack\Omega^{2,0}\rightbrack\oplus \mathfrak{su}(4)\oplus \langle \omega\rangle 
=
\Omega^2_{12}\oplus \Omega^2_{15}\oplus \Omega^2_1
\\
\Omega^3&=
\Omega^1 \oplus 
\leftbrack \Omega^{3,0} \rightbrack
\oplus 
\leftbrack \Omega^{2,1}_0 \rightbrack
=
\Omega^3_8\oplus \Omega^3_{8'}\oplus \Omega^3_{40}.
\end{align*}
In particular, $\leftbrack\Omega^{3,0}\rightbrack\iso \Omega^1$, so we use a ``\emph{prime}'' to distinguish it from $\Omega^3_8$, which is also isomorphic to $\Omega^1$. These spaces are characterised by
\begin{align*}
    \Omega^3_8&=\{X\iprod \psi:X\in \Omega^1\}=\{Y\wedge\omega:Y\in \Omega^1\},
    \\
    \Omega^3_{8'}&= \{X\iprod\Psi_+:X\in \Omega^1\}=\{X\iprod\Psi_-:X\in \Omega^1\},\\
    \Omega^3_{40}&=
    \{\gamma\in \Omega^3:\omega\iprod\gamma=0, \gamma\iprod \Psi_+=0\}.
\end{align*}

Consider the 4-form geometry $(M^8,\SU(4),\Psi_+= \Re(\Psi))$. Using the decomposition of $\Omega^3$ into irreducible $\SU(4)$-submodules described in \ref{subsec:Decomposition:forms:Un}, the codifferential of $\Psi_+$ decomposes as
\begin{equation*}
d^*\Psi_+  = \sigma_1\iprod  \psi + 
\sigma_1'\iprod\Psi_+ +\sigma_3^{2,1},
\end{equation*}
where $\sigma_1,\sigma_1'\in \Omega^1(M)$ and $\sigma_3^{2,1}\in \bracl{\Omega^{2,1}_0}$. For $\SU(4)$-structures, the Lee space $\Omega^1_L=\Omega^1$ is irreducible. Moreover, a direct computation shows that the restriction of the flux operator $\bH_{\Psi_+}^2$ to $\mathfrak{su}(4)$ has eigenvalue $b=0$. Therefore, $(M^8,\SU(4),\Psi_+)$ does not satisfy the hypotheses of Theorem~\ref{theorem:ricci:flatness:simple:group}.

\begin{proposition}
\label{prop:flux:operator:SU(4)}
    Let $(M^8,\SU(4),\Psi_+=\Re(\Psi))$ be a 4-form geometry. Then the flux operator $\bH_+ = \vet H_{\Psi_+}:\Omega^3\to \Omega^3$ acts on the irreducible components by
\begin{equation*}
    \vet H_+\big|_{\Omega^3_{8}}= 3\cdot \Phi,  \qquad 
\vet H_+\big|_{\Omega^3_{8'}}=4\cdot \Phi^{-1},  \qquad 
\vet H_+\big|_{\Omega^3_{40}}= 0 \cdot\id_{\Omega^3_{40}} ,
\end{equation*}
where $\Phi : \Omega^3_8\to \Omega^3_{8'}$ is the natural isomorphism given by $\Phi(X\iprod\psi)\defeq X\iprod\Psi_+$ for $X\in\Omega^1(M)$. In particular, the associated flux form $H_+\in \Omega^3(M)$ and Lee form $\zeta_+\in \Omega^1(M)$ are given, respectively, by
\begin{equation*}
H_+= \frac{1}{3}\sigma_1'\iprod\psi + \frac{1}{4}\sigma_1\iprod\Psi_+ \qandq
        \zeta_+ 
        =-\sigma_1 .
    \end{equation*}
\end{proposition}

\begin{proof}
Consider
\begin{align*}
    \gamma_1
    &= e^1\iprod\psi
    = e^{234}+e^{256}+e^{278}
    \in \Omega^3_8,
    \\
    \gamma_2
    &= e^1\iprod\Psi_+
    = e^{357}-e^{368}-e^{458}-e^{467}
    \in \Omega^3_{8'}.
\end{align*}
A direct computation gives
\begin{align*}
    \mathbf H_+(\gamma_1)
    &= 3\cdot(e^{357}- e^{368} - e^{458} -e^{467}) 
    = 3 \cdot \pts{e^1\iprod\Psi_+}
    = 3\cdot\Phi(\gamma_1).
\\
    \mathbf H_+({\gamma}_2)
    &= 4\cdot(e^{234} + e^{256} + e^{278}) 
    = 4 \cdot\pts{e^1\iprod\psi}
    = 4\cdot\Phi^{-1}({\gamma}_2).
\end{align*}
Next, consider $\gamma=\fR(z^1\wedge z^2\wedge \overline z^3)=e^{135}+e^{146}+e^{236}-e^{245}\in\Omega^3_{40}=\bracl{\Omega_0^{2,1}}$. A direct computation again gives $\mathbf H_+(\gamma)=0$. Following the notation of Theorem~\ref{theorem:flux:3:form:expression}, we have $V_8=\Omega^3_8\oplus\Omega^3_{8'}$, and
 \begin{equation*}
     \bh_8=\begin{bmatrix}
         0 & 4\cdot\Phi^{-1}\\
         3\cdot\Phi & 0
     \end{bmatrix}
     \quad \Longrightarrow\quad
     \bh_8^{-1}=\begin{bmatrix}
         0 & \frac{1}{3}\cdot\Phi^{-1}\\
         \frac{1}{4}\cdot\Phi & 0
     \end{bmatrix}.
 \end{equation*}
Hence, the expression for $H_+$ follows directly from Theorem~\ref{theorem:flux:3:form:expression}. To obtain the expression for $\zeta_+$, we use Remark~\ref{remark:Lee:form:2:irreps}. In this case,
\begin{equation*}
    \psi=\frac{1}{2}\omega^2,\quad
    \hat{\psi}=\Psi_+, \quad
    \tau^1=\sigma_1, \quad
    \hat{\tau}^1=\sigma_1', \quad
    a=d=0,\quad
    b=4,\quad
    c=3,
\end{equation*}
and one can check from the expressions for $\psi$ and $\Psi_+$ given at the beginning of this section that
\begin{equation*}
    (X\iprod\psi)\iprod\Psi_+ =0 \qandq (X\iprod\Psi_+)\iprod\Psi_+=-4\cdot X,\quad \forall X\in \Omega^1. \qedhere
\end{equation*}
\end{proof}

\subsection{\texorpdfstring{$\rG_2$}{G2}-structures}

    A \emph{$\rG_2$-structure} on $M^7$ corresponds to a 3-form $\vphi\in\Omega^3(M)$ pointwise modelled\footnote{Our convention follows \cite{delaOssa2018a}.} on
    \begin{equation} 
    \label{eq:model:associative:3:form}
        \vphi_0 = e^{127} + e^{347} + e^{567} + e^{135} - e^{146} - e^{236} - e^{245} , 
    \end{equation}
    where $\{e^1,\dots, e^7\}$ is an orthonormal frame of $\R^7$. A 7-manifold admits a $\rG_2$-structure if and only if it is both orientable and spinnable. Since $\rG_2\subset\SO(7)$, $\vphi$ induces a metric $g$ and an orientation $\vol_M$ through the nonlinear algebraic relation
\begin{equation*}
    g(X,Y)
    = \dfrac{1}{6 \vol_M}\pts{X\iprod\vphi}\wedge\pts{Y\iprod\vphi}\wedge{\vphi},
    \qforq X,Y\in \fX(M). 
\end{equation*}
In the frame $\{e_j\}$ such that $\varphi$ takes the form \eqref{eq:model:associative:3:form}, $g=e^1\otimes e^1+\cdots+e^7\otimes e^7$ and $\vol_M=e^{1234567}$. In particular, with the Hodge star operator $*$ induced by $\varphi$ we define the co-associative form $\psi\defeq *\vphi\in \Omega^4(M)$, which is pointwise modelled on
\begin{equation}\label{eq:model:coassoc:4:form}
    \psi_0 = e^{3456} + e^{1256} + e^{1234} - e^{2467} + e^{2357} + e^{1457} + e^{1367} .
\end{equation}
Under our conventions $|\vphi|^2 = |\psi|^2=7$, or equivalently $\vphi_0\wedge\psi_0 = 7\vol_0$.
The group $\rG_2\subset \SO(7)$ can be viewed as the $\SO(7)$-stabiliser
\begin{equation*}
    \rG_2=\Stab_{\SO(7)}(\psi_0).
\end{equation*}
So $(M^7,\rG_2,\psi=\ast\varphi)$ is a \emph{coassociative} 4-form geometry, cf. Definition~\ref{def:4:form:geometry}.

\paragraph{Decomposition of differential forms}
\label{parag:dec:forms:G2}
  Let $(M^7,\vphi)$ be a $\rG_2$-structure. Then the space of 1-forms $\Omega^1$ is irreducible. The space  of 2-forms decomposes as
        \begin{equation*}
            \Omega^2 = \Omega^2_7\oplus \Omega^2_{14} , 
        \end{equation*}
        where $\fg_2\iso \Omega^2_{14}$. These components are characterised by
        \begin{align*}
            \Omega^2_7 
            &= \{X\iprod\vphi:X\in TM \}
            = \{\beta\in \Omega^2:\pts{\beta\iprod\vphi}\iprod\vphi = 3\beta \}
            = \{\beta\in \Omega^2 :\beta\iprod\psi=2\beta \} ,   \\
            \Omega^2_{14}
            &= \{\beta\in \Omega^2 :\beta\wedge\psi=0\}
            = \{\beta\in \Omega^2:\beta\iprod\vphi=0 \}
            = \{ \beta\in \Omega^2:\beta\iprod\psi=-\beta \} . 
        \end{align*}
        The space of 3-forms decomposes as
        \begin{equation*}
            \Omega^3 = \Omega^3_1\oplus\Omega^3_7\oplus\Omega^3_{27}\iso \Omega^0\oplus\Omega^1\oplus \Sigma^2_0(\R^7) , 
        \end{equation*}
        whose components can be explicitly characterised as
        \begin{align*}
            \Omega^3_1 
            &= \{f \vphi:f\in C^{\infty}(M)\} , \\
            \Omega^3_7
            &=\{X\iprod\psi:X\in \Omega^1\} , \\
            \Omega^3_{27}
            &= \{\gamma\in\Omega^3 :\gamma\wedge\vphi=0,  \gamma\wedge\psi=0 \} 
            =\chaves{h\diamond \varphi:h\in \Sigma_0^2T^*M}.
        \end{align*}

    Due to the decompositions of $\Omega^4\iso \Omega^3$ and $\Omega^5\isomorphic \Omega^2$, there are unique \emph{torsion forms}
    \begin{equation}
    \label{eq:G2:torsion:form:modules}
        \tau_0\in \Omega^0,\quad \tau_1\in \Omega^1,\quad
        \tau_2\in \Omega^2_{14},\quad \tau_3\in \Omega^3_{27},
    \end{equation}
    satisfying
    \begin{equation}
    \label{eq:torsion:forms:G2}
    \begin{aligned}
        d\vphi
        &= \tau_0 \psi + 3\tau_1\wedge\vphi + *\tau_3 \qandq 
        d\psi
        = 4\tau_1\wedge\psi +\tau_2\wedge\vphi .
    \end{aligned}
    \end{equation}
The fact that $\tau_1$ appears in both equations is due to the existence of a unique 7-dimensional irreducible component in the intrinsic torsion, see \cite[Proposition 1]{Bryant2003} or \cite[Theorem 2.23]{Karigiannis2008a}.

\paragraph{Flux operator, flux 3-form and the Lee form}
Consider the co-associative 4-form geometry $(M^7,\rG_2,\psi)$. In view of \eqref{eq:torsion:forms:G2}, the codifferential of $\psi$ decomposes as
\begin{equation}\label{eq:codif:psi:G2}
    d^*\psi = \tau_0\vphi - 3\tau_1\iprod\psi + \tau_3.
\end{equation}
For $\rG_2$-structures, the Lee space $\Omega^1=\Omega^1_L$ is irreducible and the eigenvalue of $\fg_2$ on the flux operator on 2-forms is $b=-1$, cf. \ref{parag:dec:forms:G2}.

\begin{proposition}
\label{prop:flux:operator:G2}
    Let $(M^7,\rG_2, \psi)$ be a co-associative 4-form geometry. Then the flux operator $\bH_\psi : \Omega^3\to\Omega^3$ is an isomorphism of $\rG_2$-modules and acts on each component as
    \begin{equation*}
\vet H_\psi\big|_{\Omega^3_1}=6\cdot \id_{\Omega^3_1},  \qquad 
\vet H_\psi\big|_{\Omega^3_7}=3\cdot \id_{\Omega^3_7} ,  \qquad 
\vet H_\psi\big|_{\Omega^3_{27}}=- \id_{\Omega^3_{27}} .
\end{equation*}
In particular, the associated flux 3-form $H_\vphi\in\Omega^3(M)$ and Lee form $\zeta_\vphi\in\Omega^1(M)$ are, respectively,
\begin{equation*}
H_\vphi = \frac{1}{6}\tau_0 \varphi - \tau_1 \iprod \psi - \tau_3 
\qandq
\zeta_\vphi 
        =4\tau_1 .
\end{equation*}
\end{proposition}

\begin{proof}
For the component $\Omega^3_7=\Omega^1\iprod\psi$, we consider $\gamma=X\iprod\psi$. Then, using the contraction identities for $\rG_2$-structures, which can be found in \cite[Appendix]{delaOssa2018a}, we obtain
\begin{align*}
\mathbf H_\psi(X \iprod \psi) &= \frac{1}{2!2!} X^i \psi_{i}{}^{\mu\nu}{}_{j} \psi_{\mu\nu ab}\, e^{jab} 
= \frac{1}{4} X^i \big( 4 \delta_i^a \delta_j^b - 4 \delta_i^b \delta_j^a + 2 \psi_{ijab} \big) e^{jab} 
= \frac{1}{2} X^i \psi_{ijab} e^{jab} 
\\&= 3\,X \iprod \psi.
\end{align*} 
Next, for $\Omega^3_1$, take $\gamma=\varphi$. Then 
\[
\mathbf H_\psi(\varphi)=\frac{1}{2! 1! 2!} \varphi^{\mu\nu}{}_{i} \psi_{\mu\nu ab} e^{iab}
=\frac{1}{4}\cdot 4 \varphi_{iab} e^{iab}=6 \varphi.
\]
Finally, let $\gamma = S \diamond \varphi=\tfrac{1}{2}S_{ai}\varphi_{ajk}e^{ijk}\in \Omega^3_{27}$, i.e.
$\gamma_{ijk}=S_\mu^i \varphi_{\mu jk} + S_\mu^k \varphi_{\mu ij} + S_\mu^j \varphi_{\mu ki}$, 
then
\begin{align*}
\mathbf H_\psi(S \diamond \varphi) &= \frac{1}{2\cdot 2} \big( S_\mu^i \varphi_{\mu jk} + S_\mu^k \varphi_{\mu ij} + S_\mu^j \varphi_{\mu ki} \big) \psi_{ijab}  e^{kab} \\
&= -\frac{1}{2} \left( S_j^i \varphi_{abj} + S_j^b \varphi_{iaj} + S_j^a \varphi_{ibj} \right) e^{iab} + 2 S \diamond \varphi = -3 S \diamond \varphi + 2 S \diamond \varphi = - S \diamond \varphi .
\end{align*}
The expression for $H_\psi$ follows from Theorem~\ref{theorem:flux:3:form:expression} together with \eqref{eq:codif:psi:G2}. For the Lee form, we compute
\begin{align*}
\zeta_\psi
&=H_\psi \iprod \psi 
 = -\pts{\tau_1 \iprod \psi }\iprod \psi = -\frac{1}{1!3!1!} \tau_1^\mu \psi_{\mu}{^{ijk}} \psi_{ijk\nu} e^\nu 
= \frac{1}{6} \tau_1^\mu \cdot 24 \delta_{\mu\nu} e^\nu = 4\tau_1.\qedhere
\end{align*}
\end{proof}

Following \cite[\S 2]{Cabrera1994}, Cabrera et al. define the 1-form $\vtheta_\varphi$ associated with a $\rG_2$-structure $(M^7,\varphi)$ by
\begin{equation*}
    \vtheta_\vphi
    \defeq -\dfrac{1}{3}*\pts{*d\vphi\wedge\vphi}
    = -\dfrac{1}{3}*\pts{d^*\vphi\wedge\psi}
    = -\dfrac{1}{3}\pts{d^*\vphi}\iprod\vphi .
\end{equation*}
The authors called it the \emph{Lee form}. Since $\vtheta_\vphi = \zeta_\vphi$, their notion coincides with ours in Definition~\ref{def:lee:form}.

\begin{remark}
    The classification of $\rG_2$-structures by intrinsic torsion type was initiated in \cite{Fernandez1982}. The intrinsic torsion $\Gamma$ of a $\rG_2$-structure takes values pointwise in\footnote{The notation $\mcal W$ was introduced by Gray--Fernandez \cite{Fernandez1982}. Although we do not use it elsewhere in this text, we include it here for completeness.}
\begin{align*}
    \R^7\tensor \fg_2^\perp 
    &\iso \Lambda^3\oplus\Lambda^2_{14}
    \iso \Lambda^3_1\oplus\Lambda^3_7\oplus\Lambda^3_{27}\oplus\Lambda^2_{14}
    \iso \cW_1\oplus\cW_4\oplus\cW_3\oplus\cW_2.
\end{align*}
The components of $\Gamma$ can be characterised by the torsion forms in \eqref{eq:G2:torsion:form:modules}; see \cite{Bryant2003}.
The case of \emph{integrable $\rG_2$-structures}, with $\tau_2=0$, plays a special role, since they admit a unique metric-compatible connection with totally skew-symmetric torsion $T=H_\varphi$ \cite[Theorem~4.7]{Friedrich2003b}. Proposition~\ref{prop:flux:operator:G2} provides an alternative way of recovering the expression for $H_\varphi$ given by Friedrich--Ivanov \cite[Theorem 4.8]{Friedrich2003b}.
\end{remark}

\paragraph{String algebroids over \texorpdfstring{$\rG_2$}{G2}-structures}
Finally, we discuss generalized Ricci flatness for $\rG_2$-structures. Proposition~\ref{prop:string:alg:G2} was previously proved in \cite[Theorem~4.9]{daSilva2024a}. Here, we recover it as an application of Theorem~\ref{theorem:ricci:flatness:simple:group}, which is a general result about 4-form geometries.

\begin{proposition}[{\cite[Theorems 3.8 and 4.9]{daSilva2024a}}]
\label{prop:string:alg:G2}
\label{prop:ricci:flatness:G2}
    Let $(P,K,\theta,\pdi{\cdot,\cdot}_\fk)\to (M^7,\rG_2,\psi=\star\varphi)$ be the co-associative string algebroid induced by an integrable $\rG_2$-structure ($\tau_2=0$). If $\theta$ is a $\rG_2$-instanton, then the generalized Ricci tensor of the pair \eqref{eq:generalized:induced:pair} vanishes, 
    \[
    \GRic^+\pp{V_+^{\rG_2,\psi}, \dv^{\rG_2,\psi}}=0,
    \]
    where the divergence is
    \[
        \dv^{\rG_2,\psi}=\dv^{V_+^{\rG_2,\psi}}-2\bracks{4\tau_1,\cdot}.
    \]
\end{proposition}

\begin{proof}
    Since the $\rG_2$-structure is integrable, we have $\nabla^+\psi = 0$ \cite{Friedrich2003b}. Moreover, since $\rG_2=\Stab(\psi)$ and $\theta$ is a $\rG_2$-instanton, we have
    \begin{equation*}
        F_\theta\in \fg_2
        \quad \Longleftrightarrow\quad
        F_\theta\diamond\psi = 0 .
    \end{equation*}
    Together, $\nabla^+\psi=0$ and $F_\theta \diamond\psi=0$ are equivalent to $D^+_-\psi=0$.
    Thus, the 4-form geometry $(M^7,\rG_2,\psi)$, for which $b=-1$, satisfies the hypotheses of Theorem~\ref{theorem:ricci:flatness:simple:group}.
\end{proof}

\paragraph{Explicit examples}
The generalized Ricci-flat metrics satisfying the hypotheses of Proposition~\ref{prop:ricci:flatness:G2} are usually referred to as solutions of the \emph{heterotic $\rg_2$-system}. As in the Hull--Strominger system, the $K$-bundle $P$ under consideration usually decomposes into a part associated with $TM$ and another part associated with an auxiliary vector bundle $E'$. This system is also motivated by 10-dimensional string theories; cf. \cite{delaOssa2018,delaOssa2018a}.

There are many solutions in the literature. For example, de la Ossa--Galdeano constructed families of examples over the squashed 7-sphere and the squashed Allof-Wallach space \cite{delaOssa2021}; Clarke--Garcia-Fernández--Tipler constructed solutions over the total space of $\bb T^3$-bundles over hyper-Kähler manifolds \cite{Clarke2022}; Fino--Grantcharov--Medel constructed solutions on total spaces of fibrations over K3 surfaces \cite{fino2026newsolutionsG2hullstrominger}; Lotay--Sá Earp constructed approximate solutions over contact Calabi--Yau manifolds \cite{Lotay2023}; and Galdeano--Stecker found examples over 3-$(\alpha,\delta)$-Sasaki manifolds \cite{Galdeano2024}. Ivanov--Ivanov studied how solutions of the Hull--Strominger system can be extended to the heterotic $\rG_2$-system via Cartesian products $M^6\times \bb R$ or $M^6\times \bb S^1$. Moroianu--Raffero--Vezzoni studied the heterotic $\rG_2$-system on 2-step nilmanifolds endowed with principal torus bundles, discussing the existence of solutions for all possible isomorphism classes of 7-dimensional 2-step nilpotent Lie algebras and providing examples with $\tau_1=0$
both when $\tau_0=0$ and when $\tau_0\neq 0$. In \cite{daSilva2024a}, the first and second authors and their collaborators constructed examples of the heterotic $\rG_2$-system in order to obtain solutions to the coupled $\rG_2$-instanton equations. More recently, del Barco--Fowdar--Moreno classified nilmanifolds admitting solutions to the heterotic $\rG_2$-system \cite{delBarco2026}. We now present some simple examples.

\begin{example}[Torsion-free $\rG_2$-structures]
    If $(M^{7},\varphi)$ is a torsion-free $\rG_2$-manifold, so that $d\varphi=0$ and $d\psi=0$, then in particular $H_\psi=0$. Considering $P=M\times \{1\}$, the condition that the exact Courant algebroid $TM\oplus T^*M$ be generalized Ricci-flat with Riemannian divergence $\dv=\dv^{V_+}$ reduces to
    \(
    \Ric^g=0
    \),
    which is indeed the case \cite{bonan_holonomie_g2_spin7,alekseevskii_holonomy_groups_1968}; see also \cite[Theorem 11.8]{Salamon1989}.
\end{example}

A $\rG_2$-structure $(M^7,\varphi)$ is called \emph{nearly parallel} if there exists a constant $\lambda\in \bb R$ such that $d\varphi=\lambda\psi$. This condition means that all torsion forms except $\tau_0$ vanish.\footnote{In this case, $d\varphi=\tau_0\psi$ and $d\psi=0$. Hence $0=d^2\varphi=d\tau_0\wedge\psi$, which implies $d\tau_0\iprod\varphi=0$. Consequently, $\tau_0$ is constant, provided that $M$ is connected.} Thus,

\[
H_\vphi=\frac{1}{6}\lambda\cdot \varphi
\quad \Rightarrow\quad
dH_\vphi=\frac{1}{6}\lambda^2\cdot \psi.
\]
Moreover, $\nabla^+H_\psi=\frac16\lambda\nabla^+\varphi=0$. The Bismut connection $\nabla^+$ is an instanton since $R^+\in \Sigma^2(\fg_2)$, and nearly parallel $\rG_2$-structures are Einstein \cite{Friedrich1997}, with
\[
\Ric^g=\frac38\lambda^2g, \qquad 
\rmm{Scal}^g=\frac{21}8\lambda^2.
\]
The round 7-sphere $\bb S^7$ has a nearly-parallel $\rG_2$-structure with $\lambda=4$ and constant sectional curvature. There is also a metric squashing of the 7-sphere which is still nearly-parallel but does not have constant sectional curvature.

\begin{example}[Nearly parallel round 7-sphere {\cite[\S6]{Ivanov2005}}]
    Let $(\bb S^7, \vphi)$ be the standard $\rG_2$-structure on the 7-sphere defined by $\varphi_x=x\iprod \Omega_0$, where $\Omega_0$ is the standard $\spin(7)$-structure on $\bb R^8\supset \bb S^7$. As described above, this $\rG_2$-structure induces an Einstein metric with constant sectional curvature $\rmm{sec}=1$. Similarly to Example~\ref{example:sphereS6}, we obtain
\[
\rmmtr\pp{R^+\wedge R^+}=-\frac83dH_\vphi.
\]
    Considering again $P=\rmm{Fr}_{\SO(7)}(TM)\to M$ with Lie algebra pairing \(\bracks{A, B}:=-\frac38\rmmtr\pp{A\circ B}\), we obtain $dH_\vphi=\bracks{R^+\wedge R^+}_\fk$. Therefore, by Proposition~\ref{prop:ricci:flatness:G2}, the string algebroid $\transitive$ admits a generalized Ricci-flat metric with respect to the divergence $\dv=\dv^{V_+}$.
\end{example}

\begin{example}[Cartesian product]
\label{example:cartesian:product:SU3->G2}
    Let $(\tilde M^6,\omega,\Psi)$ be an $\SU(3)$-structure, and let $P\to \tilde M$ be a principal bundle endowed with a connection $\theta$. Assume that the hypotheses of Proposition~\ref{prop:ricci:flatness:U(m)} are satisfied, so that the corresponding generalized metric on $T\tilde M\oplus \adjointbundleP\oplus T^*\tilde M$ is Ricci-flat. Define $M^7:=\tilde M\times \bb R$ (or $M^7:=\tilde M\times \bb S^1$) and endow it with the $\rG_2$-structure
\[
\varphi:=\eta \wedge\omega+\Psi_+,
\qquad
\psi:=\starseven\varphi=\frac12\omega^2-\eta \wedge\Psi_-,
\]
where $\eta =dt$ denotes the standard one-form on $\bb R$ or $\bb S^1$. We use the product metric $g_7=g_6+\eta \otimes \eta $ and the orientation $\vol_7=\eta \wedge\vol_6$. Consider the pullback bundle $\overline P:=\proj_1^*P\to M^7$ with the pullback connection $\overline\theta:=\proj_1^*\theta$. With the above conventions, one can prove (see, e.g., \cite[\S5.1]{fino2025}) that the torsion form of the $\rG_2$-structure is the pullback of the torsion form of the original $\SU(3)$-structure, namely
\[
H_\varphi=\proj_1^*H_\omega.
\]
Hence, the heterotic Bianchi identity is preserved under pullback, since $dH_\varphi=\proj_1^*dH_\omega$ and $F_{\overline\theta}=\proj_1^*F_\theta$. Therefore, the characteristic connection of the $\rG_2$-structure is the product extension of the Bismut connection of the $\SU(3)$-structure; in particular, $\nabla^+\varphi=0$.
Moreover, since $\theta$ is an $\SU(3)$-instanton and $\mathfrak{su}(3)\subset\mathfrak g_2\subset\Lambda^2(\bb R^7)^*$, the pullback connection $\overline\theta$ is a $\rG_2$-instanton. Consequently, the pullback construction produces a generalized Ricci-flat metric on the $\rG_2$-structure $(M^7,\varphi)$.
\end{example}

\subsection{\texorpdfstring{$\Spin(7)$}{Spin(7)}-structures}

    A $\Spin(7)$-structure on $(M^8,g)$ corresponds to a 4-form $\Omega\in \Omega^4(M)$, called the \emph{Cayley form}, which is pointwise modelled on
\begin{equation}\label{eq:model:spin(7):structure}
    \begin{split}
\Omega_0=
e^0\wedge\vphi_0 +\psi_0 & =
e^{0127} + e^{0347} + e^{0567} + e^{0135} - e^{0146} - e^{0236} - e^{0245} \\&+ e^{3456} + e^{1256} + e^{1234} - e^{2467} + e^{2357} + e^{1457} + e^{1367} , 
\end{split}
\end{equation}
where $\{e_0,\dots,e_7\}$ is an orthonormal frame of $\R^8$, $\vphi_0$ is the canonical associative 3-form \eqref{eq:model:associative:3:form}, and $\psi_0$ is the canonical co-associative 4-form \eqref{eq:model:coassoc:4:form}.

Since $\Spin(7)\subset \SO(8)$, the form $\Omega$ induces a metric $g$ and an orientation $\vol_M$ in a nonlinear way \cite{Karigiannis2008, Fadel2024}. More precisely, for $v\in TM$,
\begin{equation*}
     (g(v,v))^2=-\frac{7^3}{6^{\frac{7}{3}}}\frac{(\det\ B_{ij}(v))^{\frac 13}}{A(v)^3},
     \qquad 
     \begin{matrix}
    B_{ij}(v)  =(\Omega(v,e_i)\wedge \Omega(v,e_j)\wedge \Omega(v))_{1234567},\\
    A(v)=(\Omega(v)\wedge \Omega)_{1234567}.
    \hspace{2.46cm}{}
\end{matrix}
\end{equation*}
The metric and the orientation determine a Hodge star operator $*$, under which $\Omega=\ast\Omega$ is self-dual. With our conventions, we have $|\Omega|^2=14$. In the frame $\{e_j\}$ such that $\Omega$ takes the form \eqref{eq:model:spin(7):structure}, we have $g=e^0\otimes e^0+\cdots+e^7\otimes e^7$ and $\vol_M=e^{01234567}$. 
Thus, a $\Spin(7)$-structure $(M,\Omega)$ gives rise to a \emph{Cayley} 4-form geometry $(M,\Spin(7),\Omega)$.

\begin{remark}
        An 8-manifold admits a $\Spin(7)$-structure if and only if it is orientable, spinnable, and
\begin{equation*}
    p_1^2(M) - 4p_2(M) + 8\chi(M) = 0 , 
\end{equation*}
    for an appropriate choice of orientation, where $p_i(M)$ denotes the $i$-th Pontryagin class and $\chi(M)$ denotes the Euler characteristic; see \cite[Theorem~10.7]{Lawson1989}.
\end{remark}

\paragraph{Decomposition of differential forms}
\label{sec:decomposition:forms:spin7}

Let $(M^8,\Omega)$ be a $\Spin(7)$-structure. Following \cite[\S~2]{Ivanov2004}, the spaces $\Omega^k(M)$ decompose into irreducible $\Spin(7)$-submodules. The spaces of $0$-forms and 1-forms are irreducible as $\Spin(7)$-modules. The space of 2-forms decomposes as $\Omega^2 =  \Omega^2_{7} \oplus \Omega^2_{21} $, where $\mathfrak{spin}(7)\iso\Omega^2_{21}$. Moreover, these components are characterised by 
        \begin{equation*}
        \begin{aligned}
            \Omega^2_7
            &= \{\beta\in\Omega^2:\beta\iprod\Omega = 3\beta\},
            \\
            \Omega^2_{21}
            &= \{\beta\in\Omega^2:\beta\iprod\Omega = -\beta\}=\chaves{\beta\in \Omega^2:\beta\diamond\Omega=0}.
        \end{aligned}
        \end{equation*}
        The space of 3-forms decomposes as $\Omega^3= \Omega^3_8\oplus \Omega^3_{48}$, where
        \begin{align*}
            \Omega^3_8
            &= \{X\iprod\Omega\in \Omega^3:X\in TM\} \qandq 
            \Omega^3_{48}
            = \{\gamma\in \Omega^3:\gamma\wedge\Omega=0\}.
        \end{align*}

Due to the decomposition of $\Omega^5\iso \Omega^3$, there are unique \emph{torsion forms} 
$$\tau_1\in \Omega^1
\qandq \tau_3\in \Omega^3_{48}$$
satisfying
\begin{equation}\label{eq:torsion:forms:Spin(7)}
    d\Omega = \tau_1\wedge\Omega + *\tau_3.
\end{equation}

\paragraph{Flux operator, flux 3-form and the Lee form}
Consider the Cayley 4-form geometry $(M^8,\spin(7),\Omega)$. From \eqref{eq:torsion:forms:Spin(7)}, the codifferential of $\Omega$ decomposes as
\begin{equation}
\label{eq:codif:Omega:spin7}
d^* \Omega = -\tau_1 \iprod \Omega + \tau_3.
\end{equation}
For $\spin(7)$-structures, the Lee space $\Omega^1_L=\Omega^1(M)$ is irreducible, and the eigenvalue of $\mathfrak{spin}(7)$ on the flux operator on 2-forms is $b=-1$, cf. \ref{sec:decomposition:forms:spin7}.

\begin{proposition}
\label{prop:flux:operator:spin7}
    Let $(M^8,\Spin(7), \Omega)$ be a Cayley 4-form geometry. Then the flux operator $\vet H_\Omega:\Omega^3\to \Omega^3$ is an isomorphism of $\Spin(7)$-modules and acts on each component as
\begin{equation*}
\vet H_\Omega\big|_{\Omega^3_8}=6\cdot \id_{\Omega^3_8} ,  \qquad 
\vet H_\Omega\big|_{\Omega^3_{48}}=- \id_{\Omega^3_{48}}.
\end{equation*}
In particular, the associated flux 3-form $H_\Omega\in\Omega^3(M)$ and Lee form $\zeta_\Omega\in \Omega^1(M)$ are, respectively,
\begin{equation*}
H_\Omega = -\frac{1}{6}\tau_1 \iprod \Omega - \tau_3
\qandq
\zeta_\Omega
        =\frac{7}{6}\tau_1 .
\end{equation*}
\end{proposition}

\begin{proof}
For the vector component $\Omega^3_8=\{X\iprod \Omega:\ X\in TM\}$, we use the contraction identities for $\Spin(7)$-structures from \cite[Appendix~A]{Karigiannis2008} to obtain
\begin{align*}
    \mathbf H_\Omega(X\iprod\Omega)&
     = \frac{1}{2!2!}X^i\,{\Omega_i}{}^{\mu \nu}{}_{j}\,\Omega_{\mu \nu ab}\,e^{j ab}=\frac{1}{4}X^i\big(
    6\,\delta_i{}_{a}\delta_j{}_{b}
-6\,\delta_i{}_{b}\delta_j{}_{a}
+ 4\,\Omega_{ijab}
\big)e^{jab }
=6\, X\iprod \Omega.
\end{align*}
For $\Omega^3_{48}(M)$, consider $\gamma=e^{023}-e^{057}$. Direct inspection shows that $\gamma\wedge\Omega=0$, hence $\gamma\in\Omega^3_{48}$, and one then computes $\mathbf H_\Omega(\gamma)=-\gamma$, which yields the expression for $H_\Omega$. For the Lee form, we have
\[
\zeta_\Omega=H_\Omega\iprod \Omega 
=-\frac{1}{6\cdot 3!}
        \tau_1^\mu \Omega_{\mu}{^{ijk}} \Omega_{ijk\nu} e^\nu 
= \frac{1}{36} \tau_1^\mu \cdot 42 \delta_{\mu\nu} e^\nu
= \frac{7}{6}\tau_1.\qedhere
\]
\end{proof}

\medskip

The Lee form $\vtheta_{\Omega}$ of a $\Spin(7)$-structure was defined in \cite{Ivanov2004}, following ideas from \cite{Cabrera1995}, as the 1-form
\begin{equation}\label{eq:Lee:form:spin7}
    \vtheta_\Omega
    \defeq -\dfrac{1}{7}*\pts{*d\Omega\wedge\Omega}
    = \dfrac{1}{7}*\pts{d^*\Omega\wedge\Omega}
    =\dfrac{1}{7}\pts{d^*\Omega}\iprod\Omega .
\end{equation}
One can check that $\vtheta_\Omega = \tau_1$. Thus, our notion of Lee form in Definition~\ref{def:lee:form} differs from the literature by a constant multiple, $\zeta_\Omega= \frac{7}{6}\vtheta_\Omega$.

\begin{remark}
    The classification of different types of $\Spin(7)$-structures via intrinsic torsion was first considered in \cite{Fernandez1986} and subsequently developed in \cite{Cabrera1995}. The intrinsic torsion $\Gamma$ of a $\Spin(7)$-structure takes values pointwise in\footnote{The  $\mcal W$ notation is due to Fernandez \cite{Fernandez1986} and will not be used in this text.}
\begin{equation*}
    \R^8\tensor\mathfrak{spin}(7)^\perp
    \iso
    \Lambda^3
    = \Lambda^3_8\oplus\Lambda^3_{48}
    \iso \cW_2\oplus\cW_1.
\end{equation*}
The components of $\Gamma$ are characterised by the torsion forms $\tau_1\in\Omega^1$ and $\tau_3\in \Omega^3_{48}$ in \eqref{eq:torsion:forms:Spin(7)}.
In \cite[Theorem~1.1]{Ivanov2004}, Ivanov showed that $\Spin(7)$-structures play a special role among geometries with skew-symmetric torsion, since every such structure admits a unique metric-compatible connection $\nabla^+$ with totally skew-symmetric torsion $H_\Omega$. Proposition~\ref{prop:flux:operator:spin7} therefore recovers the results of \cite{Ivanov2004}, where a direct approach based on solving a linear system of equations was used. Further proofs using spinors were given by Lucía Martin-Mérchan in \cite{Martin-Merchan2021,Martin-Merchan2020}, and we have provided an alternative proof.
\end{remark}

\paragraph{String algebroids over \texorpdfstring{$\spin(7)$}{Spin(7)}-structures}
Since the Bismut connection $\nabla^+$ is always compatible with the $\Spin(7)$-structure, obtaining a generalized Ricci-flat metric requires only a $\Spin(7)$-instanton satisfying the heterotic Bianchi identity.

\begin{proposition}\label{prop:string:alg:Spin7}
\label{prop:ricci:flatness:Spin7}
Let $(P,K,\theta,\pdi{\cdot,\cdot}_\fk)\to (M^8,\spin(7),\Omega)$ be the Cayley string algebroid induced by a $\spin(7)$-structure. If $\theta$ is a $\spin(7)$-instanton, then the generalized Ricci tensor of the pair \eqref{eq:generalized:induced:pair} vanishes, 
   \[
    \GRic^+\pp{V_+^{\spin(7),\Omega}, \dv^{\spin(7),\Omega}}=0,
    \]
    where the divergence is
    \[
        \dv^{\spin(7),\Omega}=\dv^{V_+^{\Spin(7),\Omega}}-2\bracks{\tfrac76 \tau_1, {} \cdot {} }.
    \]
\end{proposition}

\begin{proof}
    Since every $\spin(7)$-structure admits a Bismut connection $\nabla^+$, we have $\nabla^+\Omega=0$ \cite{Friedrich2002,Ivanov2004}. Moreover, since $\Spin(7) = \Stab_{\SO(8)}(\Omega_0)$ and $\theta$ is a $\spin(7)$-instanton, we have
    \begin{equation*}
        F_\theta\in\mathfrak{spin}(7)
        \quad \Longleftrightarrow\quad
        F_\theta\diamond\Omega = 0.
    \end{equation*}
    Together, $\nabla^+\Omega=0$ and $F_\theta\diamond\Omega=0$ are equivalent to $D^+_-\Omega=0$. Thus, the 4-form geometry $(M^8,\Spin(7),\Omega)$, for which $b=-1$, satisfies the hypotheses of Theorem~\ref{theorem:ricci:flatness:simple:group}.
\end{proof}

\paragraph{\texorpdfstring{$\Spin(7)$}{Spin(7)}-structures induced by \texorpdfstring{$\rG_2$}{G2}-structures via the flux operator}\label{parag:lifting:G2:to:spin7}
Now we use the 4-form formalism to describe the relationship between the flux 3-forms and Lee forms of a $\rG_2$-structure on $N^7$ and the induced $\Spin(7)$-structure on the product $M^8=\R\times N^7$. Our result can be considered a generalisation of \cite[Theorem~5.1]{Ivanov2005}, with the difference that we start with a generic $\rG_2$-structure, not necessarily integrable.

\smallskip

Let $(N^7,\vphi)$ be a manifold endowed with a $\rG_2$-structure. We can then consider a $\Spin(7)$-structure on $M^8\defeq \R\times N^7$ given by
\begin{equation}\label{eq:Spin7:induced:G2:4:form}
    \Omega\defeq \eta\wedge\vphi + \psi , 
\end{equation}
where $\eta=dt$ is the coordinate 1-form in the direction of $\bb R$, with $d\eta=0$ and $|\eta|=1$. Here, $\starseven$ and $\stareight$ denote the Hodge star operators in dimensions 7 and 8, respectively. In particular, we have
\[
g_M=g_N+\eta\otimes\eta, \qquad 
\vol_M=\eta\wedge \vol_N.
\]

\begin{lemma}\label{lemma:G2:to:Spin7:forms}
For $M^8=\bb R\times N^7$ with the $\spin(7)$-structure constructed as above, the space of $\spin(7)$-components of $\Omega^3$ decomposes with respect to the associated $\rG_2$-structure as
    \[
\Omega^3_8(M)=\Omega^3_{\overline 7^8}(M)\oplus \Omega^3_1(N), \quad 
\Omega^3_{48}(M)=\Omega^3_{27}(N)\oplus \pp{\eta\wedge \Omega^2_{14}\pp{N}}\oplus \Omega^3_{\overline 7^\perp}(M).
\]
Here $\Omega^3_{\overline{7}^8}(M)$ and $\Omega^3_{\overline{7}^\perp}(M)$ are isomorphic 7-dimensional $\rG_2$-modules identified as
\begin{align*}
    \Omega^3_{\overline 7^8}(M)&=\chaves{
    -\eta\wedge (X^7\iprod\varphi)+X^7\iprod\psi:X^7\in \Omega^1(N)
    },
    \\
    \Omega^3_{\overline 7^\perp}(M)&=\chaves{
    4\eta\wedge (X^7\iprod\varphi)+3X^7\iprod\psi:X^7\in \Omega^1(N)
    }.
\end{align*}
Moreover, for $X^7\in \Omega^1(N)\iso\Omega^3_7(N)$:
\begin{equation}
    \proj_{\Omega^3_8}\pts{X^7\iprod\psi}
    = \dfrac{4}{7}X^7\iprod\Omega
    \qandq
    \proj_{\Omega^3_{48}}\pts{X^7\iprod\psi}
    = \dfrac{4}{7}\eta\wedge\pts{X^7\iprod\vphi} + \dfrac{3}{7}X^7\iprod\psi .
\end{equation}
\end{lemma}

\begin{proof}
    For $\Omega^3_8(M)$, we have elements of the form $X^8\iprod\Omega$ for arbitrary $X^8=X^7+f\xi\in \Omega^1(M)$, with $X^7\in \Omega^1(N)$, $f\in \mcal C^\infty(M)$ and $\xi:=\eta^\flat$ unit vector field along $\bb R$. Then
    \begin{align*}
        X^8\iprod\Omega&=
        \pp{X^7+f\xi}\iprod (\eta\wedge\varphi+\psi)
        =-\eta\wedge \pp{X^7\iprod\varphi }+X^7\iprod\psi +f\varphi .
    \end{align*}
    Since $-\eta\wedge (X^7\iprod\varphi)+X^7\iprod\psi$ is an arbitrary element of $\Omega^3_{\overline 7^8}(M)$, while $f\varphi$ is an arbitrary element of $\Omega^3_1(N)$, the decomposition of $\Omega^3_8(M)$ follows. By a dimension count, it remains to check that each component in the expression for $\Omega^3_{48}(M)$ above is contained in that space. Let $\beta\in \Omega^2_{14}(N)$. Then $\beta\wedge \psi=0$, and thus
    \begin{align*}
        \pp{\eta\wedge \beta} \wedge\Omega&=
        \eta\wedge \beta\wedge (\eta\wedge \varphi+\psi)=\eta\wedge (\beta\wedge\psi)=0.
    \end{align*}
Hence, $\eta\wedge\Omega^2_{14}(N)\subset \Omega^3_{48}(M)$. Now let $\gamma\in \Omega^3_{27}(N)$. Then $\gamma\wedge\varphi=0=\gamma\wedge\psi$, and hence $\gamma\wedge\Omega=0$. Finally, using $(X^7\iprod\psi)\wedge \psi=0$, we obtain
\begin{align*}
\pp{
    4\eta\wedge (X^7\iprod\varphi)+3X^7\iprod
\psi }\wedge\Omega 
&=\eta\wedge \pp{
-3(X^7\iprod\psi)\wedge\varphi
+4(X^7\iprod\varphi)\wedge \psi
}
\\&=
\eta\wedge \pp{
3\starseven((X^7\iprod\psi)\iprod\psi )
+4\starseven((X^7\iprod\varphi)\iprod\varphi )
}
\\&=(-12+12)\eta\wedge \starseven\ X^7=0.
\end{align*}
For the projections, let $X\in\Omega^1(N)$. Then
\begin{align*}
    \pts{X\iprod\psi}\iprod\Omega
    &= \pts{X\iprod\psi}\iprod\pts{\vphi\wedge\eta + \psi} 
    = \pts{X\iprod\psi}\iprod\psi
    = -4X \\
    \Longrightarrow    \pts{\pts{X\iprod\psi}\iprod\Omega}\iprod\Omega
    &= -4X\iprod\pts{\eta\wedge\vphi + \psi}
    = 4\eta\wedge\pts{X\iprod\vphi} - 4X\iprod\psi .
\end{align*}
Therefore, 
\begin{equation*}
    \proj_{\Omega^3_8}\pts{X\iprod\psi}
    = -\dfrac{(-4)}{7}\pts{ -\eta\wedge\pts{X\iprod\vphi} + X\iprod\psi }
    = \dfrac{4}{7}X\iprod\Omega.
\end{equation*}
The result now follows from
\[
\proj_{\Omega^3_{48}}\pts{X\iprod\psi}
    = X\iprod\psi - \proj_{\Omega^3_8}\pts{X\iprod\psi}.\qedhere
\]
\end{proof}

Let us compute the torsion 3-form $H_\Omega$ of the $\Spin(7)$-structure in terms of the initial $\rG_2$-structure. The method is as follows: by the definitions of $\Omega$ and its associated flux operator, we have
\begin{equation*}
    \vet{H}_\Omega(H_\Omega)
    =d^{\stareight}\Omega
    = d^{\stareight}\pts{\eta\wedge\vphi} + d^{\stareight}\psi , 
\end{equation*}
and we know that the space of 3-forms on $M^8$ decomposes into irreducible components as $\Omega^3 = \Omega^3_8\oplus\Omega^3_{48}$, cf. \ref{sec:decomposition:forms:spin7}. Thus, it remains to identify the components of $d^{\stareight}\pts{\eta\wedge\vphi}$ and $d^{\stareight}\psi$ in each of these irreducible spaces. Once this is done, we invert the flux operator and obtain the expression for $H_\Omega$.

\begin{proposition}\label{prop:Spin7:induced:by:G2}
    Let $(M^8, \Spin(7), \Omega)$ be a 4-form geometry, induced by the 4-form geometry $(N^7,\rG_2,\psi)$ as in  \eqref{eq:Spin7:induced:G2:4:form}. Then the flux 3-form and the Lee form of $(M,\Spin(7),\Omega)$ are given by
    \begin{equation*}
        H_\Omega = H_\vphi + \eta\wedge\tau_2
        \qandq
        \zeta_\Omega = \zeta_\vphi - \frac{7}{6}\tau_0\eta,
    \end{equation*}
    where $H_\vphi$ and $\zeta_\vphi$ are the flux 3-form and the Lee form of $(N,\rG_2,\psi)$.
\end{proposition}

\begin{proof}
Using the expressions for $d\vphi$ and $d\psi$ in \eqref{eq:torsion:forms:G2}, we obtain
\begin{align*}
    \stareight d\psi
    &= -\eta\wedge\starseven d\psi
    = -\eta\wedge\pts{4\starseven\pts{\tau_1\wedge\psi} + \starseven\pts{\tau_2\wedge\vphi}}
    = -\eta\wedge\pts{4\tau_1\iprod\vphi - \tau_2}, 
\\
    \stareight\pts{\eta\wedge d\vphi}
    &= \starseven d\vphi
    = \starseven\pts{ \tau_0\psi + 3\tau_1\wedge\vphi + \starseven\tau_3 }
    = \tau_0\vphi - 3\tau_1\iprod\psi + \tau_3 .
\end{align*}
Therefore,
\begin{align*}
    d^{\stareight}\Omega
    &= -\pts{\stareight d\stareight\pts{\eta\wedge\vphi} + \stareight d\stareight\psi}
    = -\pts{\stareight d\psi - \stareight\pts{\eta\wedge d\vphi}}\\
    &= \eta\wedge 4\pts{\tau_1\iprod\vphi} - \eta\wedge\tau_2 + \tau_0\vphi - 3\tau_1\iprod\psi + \tau_3 \\
    &= 4\pts{\eta\wedge\pts{\tau_1\iprod\vphi} - \tau_1\iprod\psi} + \tau_1\iprod\psi + \tau_0\vphi + \tau_3 - \eta\wedge\tau_2.
\end{align*}
Noting that $\tau_1\iprod\Omega
    = -\eta\wedge\pts{\tau_1\iprod\vphi} + \tau_1\iprod\psi \in \Omega^3_8 $, we conclude
\begin{equation*}
    d^{\stareight}\Omega
    = -4\tau_1\iprod\Omega + \tau_1\iprod\psi + \tau_0\vphi + \tau_3 - \eta\wedge\tau_2 .
\end{equation*}
From Lemma~\ref{lemma:G2:to:Spin7:forms}, we have $\tau_0\vphi \in \Omega^3_8$, $\tau_3\in\Omega^3_{48}$, and $\eta\wedge\tau_2\in \Omega^3_{48}$. Thus, the only term that must be projected is $\tau_1\iprod\psi$, whose components are given by Lemma~\ref{lemma:G2:to:Spin7:forms}. Therefore,
    \begin{align*}
    d^{\stareight}\Omega
    &= -4\tau_1\iprod\Omega + \tau_1\iprod\psi + \tau_0\vphi + \tau_3 - \eta\wedge\tau_2 \\
    &= -4\tau_1\iprod\Omega + \proj_8(\tau_1\iprod\psi) + \proj_{48}(\tau_1\iprod\psi) + \tau_0\vphi + \tau_3 - \eta\wedge\tau_2 \\
    &= -4\tau_1\iprod\Omega + \dfrac{4}{7}\tau_1\iprod\Omega + \dfrac{4}{7}\eta\wedge\pts{\tau_1\iprod\vphi} + \dfrac{3}{7}\tau_1\iprod\psi + \tau_0\vphi + \tau_3 - \eta\wedge\tau_2.
\end{align*}
Inverting the flux operator $\vet{H}_\Omega$, we obtain
\begin{align*}
    H_{\Omega}
    &= \vet{H}_{\Omega}^{-1}\pts{d^{\stareight}\Omega}
    = -\dfrac{4}{6}\tau_1\iprod\Omega 
    + \dfrac{4}{6\cdot 7}\tau_1\iprod\Omega 
    - \dfrac{4}{7}\pts{\tau_1\iprod\vphi}\wedge\eta
    -\dfrac{3}{7}\tau_1\iprod\psi
    +\dfrac{1}{6}\tau_0\vphi 
    -\tau_3
    +\eta\wedge \tau_2 \\
    &= \pts{\dfrac{2}{3} - \dfrac{2}{21} - \dfrac{4}{7}}\eta\wedge\pts{\tau_1\iprod\vphi}
    +\pts{-\dfrac{2}{3} + \dfrac{2}{21} - \dfrac{3}{7}}\tau_1\iprod\psi 
    +\dfrac{1}{6}\tau_0\vphi 
    -\tau_3
    +\eta\wedge\tau_2\\
    &= -\tau_1\iprod\psi 
    +\dfrac{1}{6}\tau_0\vphi 
    -\tau_3
    +\eta\wedge \tau_2 = H_\vphi +\eta\wedge\tau_2 .
\end{align*}
Now, using the characterisation of $\Omega^2_{14}(N)$ and $\Omega^3_{27}(N)$ in \ref{parag:dec:forms:G2}, we compute the associated Lee form
\begin{align*}
    \zeta_\Omega
    &= H_\Omega\iprod\Omega
    =\stareight\pts{H_\Omega\wedge\Omega}
    =\stareight\pts{\pts{H_\vphi+\tau_2\wedge\eta}\wedge\pts{\eta\wedge\vphi+\psi}} 
    = \stareight\pts{H_\vphi\wedge\eta\wedge\vphi + H_\vphi\wedge\psi}\\
    &= \stareight\pts{-\eta\wedge H_\vphi\wedge(\starseven\psi) + \pts{\frac{1}{6}\tau_0\vphi - (\tau_1\iprod\psi) - \tau_3}\wedge\psi} \\
    &= \stareight\pts{-\eta\wedge H_\vphi\wedge(\starseven\psi) +\frac{1}{6}\tau_0\vphi\wedge\psi}
    = -\pts{H_\vphi\wedge(\starseven\psi)} +\frac{7}{6}\tau_0\stareight\vol_N\\
    &= H_\vphi\iprod\psi - \frac{7}{6}\tau_0\eta= \zeta_\vphi - \frac{7}{6}\tau_0\eta.\qedhere
\end{align*}
\end{proof}

The computations in this section suggest a method for comparing the flux 3-forms associated with lifts of $G$-structures. One could also carry out the computations for $\rG_2$-structures induced by $\SU(3)$-structures and recover the results of \cite{fino2025}, or study $\Spin(7)$-structures induced by $\SU(4)$- and $\Sp(2)$-structures. This remains an open question that we are currently investigating.

\paragraph{Explicit examples}
The construction of solutions to the heterotic $\spin(7)$-system presents some difficulties, and the literature is relatively scarce. Such a solution consists of a $\Spin(7)$-structure $(M^8,\Omega)$ together with a principal $K$-bundle $P\to M$ endowed with a connection $\theta$ satisfying the hypotheses of Proposition~\ref{prop:ricci:flatness:Spin7}. Examples can be found in \cite{Ivanov2005, Fernandez2009}. For instance, the argument used for the nearly Kähler 6-sphere and the nearly parallel 7-sphere does not extend to the 8-sphere, which does not admit a $\spin(7)$-structure. Moreover, ``nearly parallel'' $\spin(7)$-structures\footnote{In the sense that $\nabla^g_X\Omega(X)=0$ for all $X$.} are actually torsion-free, $\nabla^g\Omega=0$, so there is greater rigidity in the $\spin(7)$ setting; see \cite[\S2]{gray_weak_holonomy} and \cite{alexandrov2004weakholonomy}.

\begin{example}[Torsion-free $\spin(7)$-structures]
    Let $(M^{8},\Omega)$ be a $\spin(7)$-manifold, i.e., $d\Omega=0$, so in particular $H_\Omega =0$. As before, consider $P=M\times \{1\}$. Then the condition that the exact Courant algebroid $TM\oplus T^*M$ be generalized Ricci-flat with Riemannian divergence $\dv=\dv^{V_+}$ reduces to the condition that the metric induced by $\Omega$ be Ricci-flat. This condition is always satisfied for $\Spin(7)$-manifolds \cite{bonan_holonomie_g2_spin7,alekseevskii_holonomy_groups_1968}; see also \cite[Corollary 12.6]{Salamon1989}.
\end{example}

\begin{example}[Cartesian product]
\label{example:cartesian:product:G2->Spin7}
    Let $(N^7,\vphi)$ be an integrable $\rG_2$-structure, and let $P\to N$ be a principal $K$-bundle endowed with a $\rG_2$-instanton $\theta$ satisfying the heterotic Bianchi identity \eqref{eq: hBi4}. By Proposition~\ref{prop:ricci:flatness:G2}, we then have a generalized Ricci-flat metric on $TN\oplus \adjointbundleP\oplus T^*N$ with divergence $\dv=\dv^{V_+}-2\bracks{4\tau_1,\cdot}$. Define $M^8:=N\times \bb R$ (or $M^8:=N\times \bb S^1$) and endow it with the $\spin(7)$-structure
\[
\Omega:=\eta\wedge\varphi+\psi,
\]
as in \ref{parag:lifting:G2:to:spin7}. Consider the pullback bundle $\overline P:=\proj_1^*P\to M^8$ with the pullback connection $\overline\theta:=\proj_1^*\theta$. Since $\tau_2=0$, it follows from Proposition~\ref{prop:Spin7:induced:by:G2} that
\[
H_\Omega=\proj_1^*H_\vphi.
\]
Hence, the heterotic Bianchi identity is satisfied. Moreover, since $\theta$ is an $\rg_2$-instanton and $\mathfrak g_2 \subset\mathfrak{spin}(7)$, $\overline\theta$ is a $\spin(7)$-instanton. Consequently, the pullback construction produces a generalized Ricci-flat metric on the $\spin(7)$-structure $(M^8,\Omega)$.
\end{example}

\begin{example}
    Let $(\tilde M^6,\omega,\Psi)$ be an $\SU(3)$-structure with Nijenhuis tensor $N_\omega\in \Omega^3$ and $\nabla^+\Psi=0$, and let $P\to \tilde M$ be a principal $K$-bundle endowed with an $\SU(3)$-instanton $\theta$. Assume also that $(H_\omega,\theta)$ satisfies the heterotic Bianchi identity, so that the hypotheses of Proposition~\ref{prop:ricci:flatness:U(m)} are satisfied. We can then define an integrable $\rG_2$-structure on $\tilde M\times \bb R$ or $\tilde M\times \bb S^1$ as in Example~\ref{example:cartesian:product:SU3->G2}. As in Example~\ref{example:cartesian:product:G2->Spin7}, we can define a $\spin(7)$-structure on $M:=\tilde M\times \bb R^2$, $\tilde M\times \bb T^2$, or $\tilde M\times \bb R\times \bb S^1$ via
    \[
    \Omega=dt\wedge \pp{ds\wedge \omega +\Psi_+}+ \frac12\omega\wedge\omega - ds\wedge \Psi_-,
    \]
    where $dt$ and $ds$ are the standard 1-forms in the directions of the flat factors. By applying the preceding two product constructions, the generalized Ricci-flat metric on $T\tilde M\oplus \adjointbundleP\oplus T^*\tilde M$ induces a generalized Ricci-flat metric on $TM\oplus \proj_1^*\pp{\adjointbundleP}\oplus T^*M$ over $(M^8,\Omega)$.
\end{example}

\subsection{\texorpdfstring{$\sp(k)\sp(1)$}{Sp(k)Sp(1)} and \texorpdfstring{$\Sp(k)$}{Sp(k)}-structures}

Let $n=4k$ denote the dimension of the manifold $M$, and suppose that $k\ge 2$. An \emph{almost quaternionic Hermitian} structure on a Riemannian manifold \((M^{4k},g)\) is given by a rank-three subbundle \(\mathcal Q \subset \End(TM)\) that is locally spanned by almost complex structures \(I,J,K\) satisfying the quaternionic relation $I\circ J=K$.
If \(\mathcal Q\) is globally trivial, so that one can choose \(I,J,K\) globally on \(M\), then the structure is called \emph{almost hyper-Hermitian}. The existence of such a subbundle \(\mathcal Q\) is equivalent to the existence of a distinguished global form $\kappa \in \Omega^4(M)$
called the \emph{Kraines form}, whose pointwise stabiliser is the group \(\Sp(k)\Sp(1)\). Moreover, the triviality of \(\mathcal Q\) is equivalent to the existence of three globally defined Hermitian \(2\)-forms \(\omega_I,\omega_J,\omega_K\), associated with the almost complex structures \(I,J,K\), and corresponds to the reduction of the structure group from $\sp(k)\sp(1)$ to $\sp(k)$.

\begin{definition}
    An \emph{$\sp(k)\sp(1)$-structure} on $M^{4k}$ is a pair $(g,\kappa)$, where $g$ is a Riemannian metric on $M$ and $\kappa\in \Omega^4(M)$, called the \emph{Kraines form}, which is pointwise modelled on 
    \begin{equation}
    \label{eq:canonical:kraines:4:form}
    \kappa_0=\frac13 \pp{\psi_I+\psi_J+\psi_K}
    \in \Omega^4(\bb R^{4k})^* , 
\end{equation}
where $\psi_A=\frac12 \omega_A\wedge\omega_A$ for each $A\in\{I,J,K\}=\{I_1,I_2,I_3\}$. The pointwise models are given by the expression below, with $(p,q,r)$ a cyclic permutation of $(1,2,3)$:
\[
\omega_{I_p} = \sum_{\mu = 0}^{k - 1}
e^{1+4\mu} \wedge e^{(p+1)+4\mu}
+e^{ ( q + 1 ) + 4 \mu } \wedge e^{(r + 1)+4 \mu },
\]
for an orthonormal basis $\{e_j\}$ of $\bb R^{4k}$. Note that $\omega_A$ is, pointwise, the Hermitian 2-form related to the almost complex structure $A$ and $|\kappa|^2=k(2k+1)/3$. On the other hand, an $\sp(k)$-structure is given by three globally defined 2-forms $(\omega_I,\omega_J,\omega_K)$ modelled as above.
\end{definition}

An $\Sp(k)$-structure naturally defines an $\SU(2k)$-structure by taking $(\omega, \Psi)$ to be
\[
\omega = \omega_I, \qquad 
\Psi=\frac1{k!}\pp{ \omega_J + i \omega_K }^k\in \Omega^{(2k,0)}(M).
\]
Consequently, it defines a metric $g$ and an $\Sp(k)\Sp(1)$-structure via \eqref{eq:canonical:kraines:4:form}.

The group $\Sp(k)\Sp(1)\subset \SO(4k)$ is the $\SO(4k)$-stabiliser of $\kappa_0$, while $\Sp(k)$ is the intersection of the three $\SO(4k)$-stabilisers of the forms $\omega_A$, as shown in \cite[\S 2]{Cabrera2007}. Note that characteristic 4-forms other than $\kappa$ could also be chosen.

\paragraph{Decomposition of differential forms}
\label{parag:dec:forms:sp(k)}
    Let $E\cong\bb C^{2k}$ and $H\cong\bb C^2$ denote the standard complex representations of $\Sp(k)$ and $\Sp(1)$, respectively. Representations of $\sp(k)\sp(1)$ occurring in the exterior algebra are described by tensor products constructed from $E$ and $H$. This framework is known as the \emph{$E/H$-formalism}; cf. \cite[\S 1]{salamon_quaternionic_kahler_1982}, \cite[\S 9]{Salamon1989} and \cite[\S 2.2]{cabrera2003hermitianstructuresquaternionicgeometries}. In particular, we have the irreducible representation
\[
\Lambda^1=[E\otimes_\bb CH]\eqcolon [EH],
\]
where brackets denote the underlying real representation. 
For 2-forms, we use the decomposition $$\Lambda^2(A\otimes B)\isomorphic(\Lambda^2A\otimes \Sigma^2B)\oplus(\Sigma^2A\otimes \Lambda^2B),$$
which gives
\begin{align*}
    \Lambda^2& =\Lambda^2[EH]\isomorphic
    [\Sigma^2E\otimes \Lambda^2H]\oplus [\Sigma^2H\otimes \Lambda^2E]\iso [\Sigma^2E]\oplus [\Sigma^2H \Lambda^2 E]
    \\&
    \iso [\Sigma^2E]\oplus [\Sigma^2H]\oplus[\Sigma^2H\Lambda^2_0E],
\end{align*}
where $\Lambda^rE=\Lambda^r_0E\oplus (\Lambda^{r-2}E\wedge\omega_I)$.

By the descriptions above, if $(M^{4k},g,\kappa)$ is a $\sp(k)\sp(1)$-structure, then the space $\Omega^1$ is irreducible. The space of 2-forms decomposes as
        \begin{equation*}
        \begin{split}
            \Omega^2
            &=
    \Omega^2_{k(2k+1)}\oplus 
    \Omega^2_3\oplus 
    \Omega^2_{3(2k+1)(k-1)}
    \isomorphic
    [\Sigma^2E]\oplus [\Sigma^2H]\oplus[\Sigma^2H\Lambda^2_0E]
    \\&\isomorphic 
    \mathfrak{sp}(k)\oplus \mathfrak{sp}(1)\oplus 
    [\Sigma^2H\Lambda^2_0E]
            ,
        \end{split}
        \end{equation*}
        where $\fs\fp(1)$ is locally spanned by $\{\omega_I,\omega_J,\omega_K\}$ and we have the characterisations
    \begin{align*}
    \Omega^2_3&
    =\bracks{\omega_I, \omega_J,\omega_K}=\{
    \beta\in \Omega^2:3\beta\iprod \kappa=(2k+1)\beta 
    \},
    \\
    \fs\fp(k)&=
    \chaves{\beta\in \Omega^2:\beta\diamond\kappa=0, \bracks{\beta,\omega_A}=0}
    =\chaves{\beta\in \Omega^2:\beta\iprod\kappa=-\beta}
    \\
    \Omega^2_{3(2k+1)(k-1)}&
    =
    \chaves{\beta\in \Omega^2:3\beta\iprod\kappa=\beta}.
    \end{align*}    
Finally, the space of 3-forms decomposes as
    \begin{align*}
        \Omega^3
        &= \Omega^3_{4k}\oplus {\Omega^3_{8k}\oplus \Omega^3_{\frac{16}3k(k^2-1)}\oplus\Omega^3_{\frac83k(k-2)(2k+1)}  }
        \\&\isomorphic 
        [EH]\oplus[ E \Sigma^3 H]\oplus [KH]\oplus [\Lambda^3_0 E \Sigma^3 H].
	\end{align*}
In particular, $[EH]^\perp = \chaves{\gamma \in  \Omega^3:\gamma \iprod \kappa = 0}$. The Lee space is $\Omega^1_L=\Omega^1$, and its image under contraction with $\kappa$ is $\Omega^1_L\iprod\kappa=\Omega^3_{4k}$. We also have the relation $\rmm{span}(\Omega^1\wedge \Omega^2_3)=\Omega^3_{4k}\oplus \Omega^3_{8k}$, where $K:= \pp{E\otimes \Lambda^2_0E}/E$ is an irreducible module.

\paragraph{Flux operator, flux 3-form and the Lee form}
With respect to the decomposition of $\Omega^3$ into irreducible $\sp(k)\sp(1)$-submodules given in \ref{parag:dec:forms:sp(k)}, the codifferential of the Kraines form decomposes into \emph{torsion forms}:
\begin{equation}
\label{eq:torsion:forms:sp:k:sp:1}
    d^*\kappa  = \tau_1\iprod  \kappa + 
    \tau_3^{8k}+\tau_3^{K}+\tau_3^{G},
\end{equation}
where $$\tau_1\in \Omega^1(M), \quad \tau_3^{8k}\in\Omega^3_{8k}, \quad \tau_3^K\in [KH], \quad\tau_3^{G}\in [\Lambda^3_0E\Sigma^3H].$$ 
Note that the eigenvalue of $\mathfrak{sp}(k)$ under the flux operator on 2-forms is $b=-1$, cf. \ref{parag:dec:forms:sp(k)}.

\begin{proposition}
\label{prop:flux:operator:spksp1}
    On the 4-form geometry $(M^{4k},\sp(k)\sp(1),\kappa)$, the flux operator $\vet H_\kappa:\Omega^3\to \Omega^3$ is an isomorphism of $\sp(k)\sp(1)$-modules and acts on each component as
\begin{equation*}
    \vet H_\kappa\big|_{\Omega^3_{4k}}=\frac{2k-2}3 \id,  \quad 
    \vet H_\kappa\big|_{\Omega^3_{8k}}=\frac{2k+1}{3}\id, \quad 
\vet H_\kappa\big|_{[\Lambda^3_0E\Sigma^3H]}=\id,  \quad 
\vet H_\kappa\big|_{[KH]}=- \id .
\end{equation*}
    In particular, the associated flux 3-form $H_\kappa\in \Omega^3(M)$ and Lee form $\zeta_\kappa\in \Omega^1(M)$ are given, respectively, by
\begin{equation*}
H_\kappa=\frac{3}{2k-2}\tau_1\iprod\kappa +\frac{3}{2k+1}\tau_3^{8k}
-\tau^K_3
+\tau_3^{G} \qandq
        \zeta_\kappa 
        =-\frac{2k+1}{2k-2}\tau_1 .
    \end{equation*}
\end{proposition}

\begin{proof}
We begin with the component $\Omega^3_{4k}=\Omega^1\iprod\kappa$. An arbitrary element is of the form $\gamma=X\iprod\kappa$, and hence
\begin{align*}
    \mathbf H_\kappa(X \iprod \kappa) &= \frac{1}{2!2!} X^i \kappa_{i}{}^{\mu\nu}{}_{j} \kappa_{\mu\nu ab}\, e^{jab} 
=
\frac14 X^i
\pp{
\frac23\bigl(\delta_{ia}\delta_{jb}-\delta_{ib}\delta_{ja}\bigr)
+\frac{4k-4}9
\kappa_{ijab}
}e^{jab}
\\&=
\frac{2k-2}{3}\cdot \frac1{3!}X^i\kappa_{ijab}e^{jab}
=\frac{2k-2}{3}X\iprod\kappa.
\end{align*} 
For the component $[E\Sigma^3H]=\Omega^3_{8k}$, we first consider the 3-form $e^1\wedge \omega_I\in \Omega^3_{4k}\oplus \Omega^3_{8k}$ and subtract its projection onto $\Omega^3_{4k}$:
\[
\gamma_1:=e^1\wedge \omega_I-\proj_{4k}(e^1\wedge \omega_I)=
e^1\wedge \omega_I+\frac3{2k+1}\pp{(e^1\wedge\omega_I)\iprod\kappa}\iprod\kappa\in \Omega^3_{8k}.
\]
A direct computation then gives  $\gamma_1\iprod^2\kappa =\tfrac13(2k+1) \gamma_1$. Similarly, for the 3-forms
$$
\gamma_{2}= e^{167}-e^{158} \qandq
\gamma_3 = -e^{159}+ e^{16(10)}+e^{25(10)}+e^{269},
$$
we find $\gamma_2\iprod^2\kappa=-\gamma_2$ and $\gamma_3\iprod^2\kappa=\gamma_3$. By invariance, $\gamma_2$ and $\gamma_3$ lie in the two summands $[KH]$ and $[\Lambda^3_0E\Sigma^3H]$, in some order. To determine the order, observe that $[\Lambda^3_0E\Sigma^3H]$ vanishes when $k=2$, whereas $\gamma_3$ involves $e^{\pp{10}}$ and therefore cannot exist in eight dimensions. Hence $\gamma_3\in [\Lambda^3_0E\Sigma^3H]$. The expression for $H_\kappa$ follows from Theorem~\ref{theorem:flux:3:form:expression}, together with \eqref{eq:torsion:forms:sp:k:sp:1}. For the Lee form, we have
\begin{align*}
\zeta_\kappa
&=H_\kappa \iprod \kappa
 = \frac{3}{2k-2}\pts{\tau_1 \iprod \kappa }\iprod \kappa = \frac{1}{4k-4} \tau_1^\mu \kappa_{\mu}{^{ijk}} \kappa_{ijk\nu} e^\nu 
=- \frac{1}{4k-4} \tau_1^\mu \cdot 2(2k+1)\delta_{\mu\nu}e^\nu \\&= -\frac{2k+1}{2k-2}\tau_1.\qedhere
\end{align*}
\end{proof}

\paragraph{String algebroids over $\sp(k)\sp(1)$-structures} \label{parag:string:alg:sp(k)sp(1)}
As in the Hermitian case, generalized Ricci flatness does not hold for $\sp(k)\sp(1)$-structures. However, if we consider the reduction to $\sp(k)$, similarly to the reduction to $\SU(m)$ in the Hermitian case, generalized Ricci flatness is recovered.

\begin{proposition}
\label{prop:ricci:flatness:sp(k)}
    Let $(M^{4k},\omega_I,\omega_J,\omega_K)$ be an $\sp(k)$-structure and suppose that $\nabla^+\omega_A=0$. Let
    $(P,K,\theta,\pdi{\cdot,\cdot}_\fk)\to(M,\sp(k),\kappa)$
    be the Kraines string algebroid induced by $\kappa=\tfrac13\pp{\psi_I+\psi_J+\psi_K}$. If $\theta$ is an $\sp(k)$-instanton, then the generalized Ricci tensor of the pair \eqref{eq:generalized:induced:pair} vanishes,
    \[    
        \GRic\pp{V_+^{\sp(k), \kappa}, \dv^{\sp(k),\kappa}}=0,
    \]
    where the divergence is
    \begin{equation}
    \label{eq:divergence:spk}
        \dv^{\sp(k),\kappa}=\dv^{V_+^{\sp(k),\kappa}}-2\bracks{-\frac{2k+1}{2k-2}\tau_1,\cdot}.
    \end{equation}
\end{proposition}

\begin{proof}
The conditions $\nabla^+\omega_A=0$ and the assumption that $\theta$ is an $\sp(k)$-instanton imply \(D^+_-\omega_A=0=D^+_-\kappa\). Since \(\Sp(k)=\stab(\omega_A)\), the 4-form geometry $(M^{4k},\sp(k),\kappa)$, for which $b=-1$, satisfies the hypotheses of Theorem~\ref{theorem:ricci:flatness:simple:group}.
\end{proof}

Explicit examples of solutions to the \emph{hyper-Hull--Strominger system} are still scarce in the literature, although constructions and qualitative descriptions of $\sp(k)$-instantons have recently been obtained \cite{Alonso2025,Madnick2024}. The only example considered here is provided by hyper-Kähler manifolds.

\begin{example}[Hyper-Kähler manifolds]
    Let $(M^{4k},\omega_I,\omega_J,\omega_K)$ be a hyper-Kähler manifold, i.e., $\nabla^g\omega_A=0$, so in particular $H_\kappa =0$. Considering $P=M\times \{1\}$, the condition that the exact Courant algebroid $TM\oplus T^*M$ be generalized Ricci-flat with Riemannian divergence $\dv=\dv^{V_+}$ is equivalent to $\Ric^g=0$, which is always satisfied for hyper-Kähler manifolds \cite[Corollary 9.4]{Salamon1989}.
\end{example}

\begin{example}[Coupled $\sp(k)$-instantons]
Analogously to Proposition~\ref{prop:coupled=>GRic=0}, let
$(M^{4k},g,\kappa)$, with $k\ge2$, be an $\Sp(k)\Sp(1)$-structure
satisfying $\nabla^+\kappa=0$, and consider a string algebroid
\[
(P,K,\theta,\bracks{\cdot,\cdot}_\fk)\to
(M,\Sp(k)\Sp(1),\kappa)
\]
satisfying the \emph{coupled} $\Sp(k)$-instanton condition, see \ref{parag:instantons} noting that $N(\sp(k))=\sp(k)\sp(1)$ inside $\SO(4k)$, and \ref{par:coupled:instantons}. If $\theta$ is a
quaternionic-Hermitian connection, in the sense that
\[
F_\theta\in\pp{\fs\fp(k)\oplus\fs\fp(1)}\otimes \adjointbundleP\quad \Rightarrow\quad F_\theta\diamond\kappa=0,
\]
then we have $D^+_-\kappa=0$, and since $\vet H^2_\kappa|_{\fs\fp(k)}=-\id_{\fs\fp(k)}$, by Theorem~\ref{theorem:ricci:flatness:coupled:equations}, the generalized Ricci tensor vanishes, \(\GRic^+\ppp{V_+^{\sp(k),\kappa},\dv^{\sp(k),\kappa}}=0\).
\end{example}

Recently, manifolds endowed with certain $G$-structures satisfying $dH_\psi=0$ and $\nabla^+H_\psi=0$ (in particular, $\nabla^+\zeta_\psi=0$) have been classified; see \cite{barbaro2026pluriclosedmanifoldsparallelbismut} for the Hermitian case, and \cite{papadopoulos_2026_classification_rigidityspecialexceptionalgeometries,moroianu_2026_geometriesparallelskewsymmetricclosed} for a more general treatment. For $\SU(m)$-, $\rG_2$-, $\Spin(7)$-, and $\Sp(k)$-structures, we can construct generalized Ricci-flat metrics on the exact string algebroid $TM\oplus T^*M$ with divergence $\dv=\dv^{V_+}$ as particular cases of Theorem~\ref{theorem:ricci:flatness:compilation}. In this setting, the generalized Ricci-flat condition is equivalent to $\Ric^+=0$, namely,
\[
\Ric^g-\frac14 H^2_\psi=0,
\qquad
d^*H_\psi=0.
\]
In particular, $H_\psi$ is harmonic. Moreover, since $\nabla^+H_\psi=0$ implies that $|H_\psi|^2$ is constant, taking the trace of the first equation shows that these metrics have constant scalar curvature, $d(\rmm{Scal}^g)=0$.

\medskip
For $\Sp(k)\Sp(1)$-structures satisfying $dH_\kappa=0$ and $\nabla^+H_\kappa=0$, the curvature of $\nabla^+$ satisfies the first Bianchi identity $R^+_{i[jkl]}=0$. Hence, by a direct analogue of Ishihara's argument \cite{Ishihara1974}, one obtains
\[
\Ric^+=\lambda g
\]
for some constant $\lambda\in \bb R$. Consequently, in this setting, the exact Courant algebroid $TM\oplus T^*M$ satisfies an analogue of the Einstein condition:
\[
\GRic^+\pp{V_+,\dv}=\lambda g,
\qquad\text{equivalently}\qquad
\Ric^g-\frac14 H_\kappa^2=\lambda g .
\]
Taking the trace gives
\(
\rmm{Scal}^g-\frac14\tr_g(H_\kappa^2)=n\lambda
\), and hence $\rmm{Scal}^g=n\lambda+\frac32|H_\kappa|^2$. In particular, since $\nabla^+H_\kappa=0$ implies that $|H_\kappa|^2$ is constant, the Riemannian scalar curvature is constant.

\smallskip
The investigation of the case $dH_\kappa=0$, where we have the exact Courant algebroid structure but not necessarily with parallel torsion, is currently in progress. In particular, we would like to determine whether there exists a notion of generalized Einstein metrics induced by $\sp(k)\sp(1)$-structures.

\phantomsection
\section*{Afterword}
\addcontentsline{toc}{section}{Afterword}

Our framework suggests several natural directions for further investigation. First, we have seen that a single $G$-structure may admit more than one distinguished 4-form geometry. This raises the question of how different choices of invariant 4-forms affect the associated constructions. In particular, it would be natural to understand whether the corresponding flux 3-forms are related, and to what extent the resulting string algebroid structures and generalized Ricci flatness conditions depend on this choice. This question already appears in the case of $\SU(4)$-structures discussed in \ref{parag:su(4):structures} and remains an open problem that we are currently investigating.

Another direction concerns the behaviour of the flux 3-form under geometric constructions relating different special structures. The calculations in \ref{parag:lifting:G2:to:spin7}, where a $\Spin(7)$-structure is induced from a $\rG_2$-structure, suggest a method for comparing the flux 3-form of the induced $G$-structure with that of the original $H$-structure. One could similarly carry out the corresponding computations for $\rG_2$-structures induced by $\SU(3)$-structures, recovering results such as those in \cite{fino2025}, or study $\Spin(7)$-structures induced by $\SU(4)$- and $\Sp(2)$-structures. These comparisons may clarify how the formalism of 4-form geometries behaves under natural lifts and reductions of structure groups.

The case of $\Sp(k)\Sp(1)$-structures points to a further question. As discussed in \ref{parag:string:alg:sp(k)sp(1)}, strong $\Sp(k)\Sp(1)$-structures, namely those satisfying $dH_\kappa=0$, are currently being investigated by the first author. In this setting, the exact Courant algebroid structure is available, and it is natural to ask whether there exists a suitable notion of generalized Einstein metrics induced by $\sp(k)\sp(1)$-structures.

Finally, other geometric structures also admit natural 4-form geometries. Examples include contact metric and 3-contact metric structures. Although these cases are not covered in the present work, the approach developed here may extend naturally to transverse geometries. We are currently working on 4-form geometries and generalized Ricci flatness in this broader context.

%\newpage
\appendix

\section{Conventions for contractions}
\label{app:contractions}

We follow the conventions of \cite[Appendix~A]{delaOssa2018a}.
Let $(M,g)$ be a Riemannian $n$-manifold, and let $\{e_i\}$ be a local frame with dual coframe $\{e^i\}$. We write a $k$-form $\xi\in \Omega^k(M)$ as
\begin{equation*}
    \xi 
    = \dfrac{1}{k!}\xi_{i_1\cdots i_k}e^{i_1\cdots i_k}
    = \dfrac{1}{k!}\xi_{i_1\cdots i_k}e^{i_1}\wedge\cdots\wedge e^{i_k},
\end{equation*}
where we use the Einstein summation convention and
$\xi_{i_1\cdots i_k}=\xi(e_{i_1},\ldots,e_{i_k})$.
If $M$ is oriented by the volume form $\vol_M$, the Hodge star operator is characterised by
\begin{equation*}
    * : \Omega^k(M)\to \Omega^{n-k}(M),
    \quad
    \xi\wedge*\eta = \bracks{\xi,\eta}\,\vol_M,
    \qquad \xi,\eta\in\Omega^k(M).
\end{equation*}
Using the metric, we define the contraction of a $(k+p)$-form $\xi\in \Omega^{k+p}(M)$ by a $k$-form $\phi\in \Omega^k(M)$ as the $p$-form
\begin{equation*}
    \phi\iprod\xi
    \defeq \dfrac{1}{k!\,p!}\phi^{i_1\cdots i_k}\xi_{i_1\cdots i_k\,i_{k+1}\cdots i_{k+p}}e^{i_{k+1}\cdots i_{k+p}}.
\end{equation*}

\begin{lemma}
Let $\phi\in \Omega^k(M)$ and $\xi\in \Omega^{k+p}(M)$ be forms on an oriented Riemannian $n$-manifold $M$. Then
\begin{equation*}
    \phi\iprod\xi = (-1)^{p(n-p-k)}*\pts{\phi\wedge*\xi}.
\end{equation*}
In particular,
\begin{align*}
    \phi\iprod\xi &= (-1)^{pk}*\pts{\phi\wedge*\xi}
    , \ &\textrm{if $n$ is odd},\\
    \phi\iprod\xi &= (-1)^{p(k+p)}*\pts{\phi\wedge*\xi} 
    , \ &\textrm{if $n$ is even}.
\end{align*}
\end{lemma}

\section{Some eigenvalues of the flux operator}

The results in this section were discovered during the research leading to this paper. Although they were ultimately unnecessary for the main arguments, we include them here for readers interested in the spectral behaviour of the flux operator.

\begin{definition}
    A 4-form geometry $(M, G,\psi)$ is said to be \emph{irreducible} if the Lee space $\Omega^1_L(M)$ is a non-trivial irreducible $G$-module and the subspace $\Omega^1_L\iprod\psi\subset \Omega^3$ is invariant under the flux operator:
    \[
        \vet H_\psi(\Omega^1_L\iprod\psi) \subset \Omega^1_L\iprod\psi.
    \]
\end{definition}
All the cases considered in this work are irreducible, namely those associated with $\U(m)$, $\SU(m)$, $\Sp(k)\sp(1)$, $\sp(k)$, $\rG_2$, and $\Spin(7)$.

\bigskip
By definition, $\pts{X\iprod\psi}\iprod\psi \in \Omega^1_L$ for every $X\in\Omega^1_L$, so the Lee space is invariant under the double-contraction map. Since this map is self-adjoint, Schur's lemma implies that, for an irreducible 4-form geometry, it is a scalar multiple of the identity:
\[
\pp{X\iprod\psi}\iprod\psi=\alpha\cdot X, \qquad \forall X\in \Omega^1_L(M),
\]
for some constant $\alpha\in \bb R$. For the same reason, there exists $a_L\in \bb R$ such that
\[
(X\iprod\psi)\iprod^2\psi =a_L\cdot X\iprod\psi, \qquad \forall X\in \Omega^1_L(M).
\]

\begin{lemma}
\label{lemma:computing:alpha:e:a_omega1}
Let $(M, G,\psi)$ be an irreducible 4-form geometry. Then the eigenvalue $\alpha\in \bb R$ of the double-contraction map on the Lee space and the eigenvalue $a_L\in \bb R$ of the flux operator $\vet H_\psi$ restricted to $\Omega^1_L\iprod\psi\subset \Omega^3$ are given by
\[
\alpha=-\frac{4}{L}|\psi|^2, 
\qquad 
a_L=\frac{1}{2}\bracks{\psi\iprod^2\psi,\psi}\cdot |\psi|^{-2},
\]
where $\bracks{\psi\iprod^2\psi,\psi}=\tfrac18\psi_{ij}{}^{ab}\psi_{ab}{}^{\mu\nu}\psi_{\mu\nu}{}^{ij}\in \bb R$ and $L:=\rank(\Omega^1_L)$ is the dimension of the Lee space. 
\end{lemma}
\begin{proof}
    To compute $\alpha$, let $\{e^\mu\}=\{e^1, \dots, e^L\}$ be a local orthonormal frame of $\Omega^1_L$ (where $0<L=\rank(\Omega^1_L)\le n$) and, if necessary, extend it to a local orthonormal frame of $\Omega^1$. Then, for fixed $\mu\le L$, we have
\[
(e^\mu\iprod\psi)\iprod\psi =\frac{1}{3!}\delta^{\mu a}\psi_{a}{^{ijk}}\psi_{ijkl}e^l=\alpha\delta^{\mu}_le^l\Longrightarrow
\sum_{i,j,k}\psi^{\mu ijk}\psi_{ijkl}=3!\cdot \alpha \delta^{\mu}_ l, \quad \forall l.
\]
Setting $l=\mu$ and summing over $\mu\le L$, the left-hand side of the last equality gives $-4!|\psi|^2$. Indeed, for $\mu>L$, we have $(e^\mu\iprod\psi)\iprod\psi=0$ and hence $\psi_{\mu ijk}\psi^{ijk\mu}=0$. The corresponding sum on the right-hand side is $3!\cdot\alpha\cdot L$. Therefore,
\[
\psi^{\mu ijk}\psi_{ijk \mu}=-4!|\psi|^2=3!\cdot \alpha L\Rightarrow
\alpha=-\frac{4}{L}|\psi|^2.
\]
To compute the eigenvalue $a_L$, defined by $(X\iprod\psi)\iprod^2\psi=a_L X\iprod\psi$, note that $\psi\iprod^2\psi\in \Omega^4(M)$ is a multiple of $\psi$:
\begin{align*}
    \psi\iprod^2\psi&= \frac12\psi_{ij}\wedge\psi_{ij}=\frac18\cdot e^k\wedge i_{e_k}(\psi_{ij}\wedge\psi_{ij})
    =\frac18\cdot 2e^k\wedge \psi_{ijk}\wedge\psi_{ij}
    =
    \frac14\cdot 2e^k\wedge (\psi_k\iprod^2\psi)
    \\&=
    \sum_{k\le L} \frac12a_L\cdot e^k\wedge \psi_k=
    \frac12a_L\cdot e^k\wedge \psi_k=
    2a_L\psi,
\end{align*}
and the result follows upon taking the inner product with $\psi$ on both sides.
\end{proof}

\begin{lemma}
    Let $(M, G, \psi)$ be an irreducible 4-form geometry. Then the Lee form is given by
    \begin{equation}
    \label{eq:lee:form:using:eigenvalues}
        \zeta_\psi=\frac{\alpha}{a_L}\tau^1,
    \end{equation}
    where $\alpha$ and $a_L$ are given by Lemma~\ref{lemma:computing:alpha:e:a_omega1}, and $\tau^1\in\Omega^1_L$ is defined by
    \(
        \proj_{\Omega^1_L\iprod\psi}(d^*\psi)=\tau^1\iprod\psi.
    \)
\end{lemma}

\begin{proof}
    For an irreducible 4-form geometry $(M, G,\psi)$, the flux 3-form $H_\psi$ can be written as 
\begin{equation*}
    H_\psi = \frac{1}{a_L}\proj_{\Omega^1_L\iprod\psi}(d^*\psi) + H^\perp_\psi
    = \frac{1}{a_L}\tau^1\iprod\psi+H_\psi^\perp,
\end{equation*}
where $H_\psi^\perp\in\pts{\Omega^1_L\iprod\psi}^\perp$ and $\tau^1\in \Omega^1_L$. To prove \eqref{eq:lee:form:using:eigenvalues}, note first that, by definition,
\begin{equation}\label{eq:first:expression:lee:form}
        \zeta_\psi = H_\psi\iprod\psi =  \frac{1}{a_L}(\tau^1\iprod\psi)\iprod\psi +H^\perp_\psi\iprod\psi.
    \end{equation}
We claim that $H^\perp_\psi\iprod\psi=0$. Indeed, for arbitrary $X\in \Omega^1_L$, we have
\begin{equation*}
\bracks{H_\psi^\perp\iprod\psi,X}=-\bracks{H_\psi^\perp, X\iprod\psi}=0
\quad\Longrightarrow\quad
H_\psi^\perp \iprod\psi \in (\Omega^1_L)^\perp.
    \end{equation*}
    On the other hand, \eqref{eq:first:expression:lee:form}, together with $\zeta_\psi\in\Omega^1_L$ and $(\tau^1\iprod\psi)\iprod\psi\in \Omega^1_L$, implies that $H_\psi^\perp\iprod\psi\in\Omega^1_L$. Therefore, $H_\psi^\perp\iprod\psi\in \Omega^1_L\cap\pts{\Omega^1_L}^\perp=\{0\}$. Consequently,
    \(
        \zeta_\psi = \tfrac{1}{a_L}\pts{\tau^1\iprod\psi}\iprod\psi = \frac{\alpha}{a_L}\tau^1.
    \)
\end{proof}

Consider the decomposition of the Lie algebra $\fg\subset \Omega^2$ into eigenspaces of the flux operator $\vet H_\psi^2$:
\[
\mathfrak g=\mathfrak g_1\oplus\cdots\oplus\mathfrak g_l.
\]
The next result shows that there are at most \emph{three} such eigenspaces $\fg_j$.

\begin{lemma}
\label{lemma:technical}
    Let $(M, G,\psi)$ be an irreducible 4-form geometry such that $\vet H_\psi^2|_\Lg:\Lg\to \Lg$ is a homomorphism. Then every eigenvalue $b\in \bb R$ of this map is either zero or satisfies
\begin{equation*}
    b^2-a_Lb+\alpha=0. 
    \end{equation*}
    In particular, there are at most \emph{three} distinct eigenvalues for $\vet H^2_\psi|_\mathfrak g$.
\end{lemma}

\begin{proof}
Suppose that $b\neq 0$, and let $\beta\in \fg$ be a non-zero $b$-eigenvector, i.e., $\beta\iprod\psi=b\beta$. With respect to an adapted orthonormal frame compatible with the Lee space $\Omega^1_L$, cf. Lemma~\ref{lemma:computing:alpha:e:a_omega1}, we have $\psi_\mu\iprod\psi =\alpha e_\mu$ and $\psi_\mu\iprod^2\psi=a_L\psi_\mu$ for $\mu\le L$, whereas $\psi_\mu=0$ for $\mu>L$. Thus,
\begin{align*}
    \beta\iprod\psi &=\frac12\beta_\mu\iprod\psi_\mu
    =\frac{1}{2a_{L}}\beta_\mu \iprod\pp{\psi_\mu\iprod^2\psi}
    =\frac{1}{4a_{L}}\beta_\mu\iprod\pp{\psi_{\mu\alpha\beta}\wedge \psi_{\alpha \beta}}
    \\&=
    \frac{1}{4a_{L}} \beta_{\mu\nu}\pp{\psi_{\mu\alpha\beta\nu}\psi_{\alpha\beta}-\psi_{\mu\alpha\beta}\wedge\psi_{\alpha\beta\nu}}=
    \frac{b^2}{a_{L}}\beta
    -\frac{1}{4a_{L}}\beta_{\mu\nu}\psi_{\mu\alpha\beta}\wedge\psi_{\nu\alpha\beta}.
\end{align*}
Using $e^\nu\iprod(e^i\wedge\psi_j)=\delta^\nu_i\psi_j+e^i\wedge \psi_{\nu j}$ and the fact that $\beta\in \fg$ implies $\beta\diamond\psi=0$, we obtain
\begin{align*}
    \beta_{\mu\nu} \psi_{\mu\alpha\beta} \wedge \psi_{\nu\alpha\beta}&= 2
    \beta_{\mu\nu} \bracks{\psi_{\mu i},\psi_{\nu j}}e^{ij}
    =
2\beta_{\mu\nu}\bracks{\psi_\mu,e^\nu\iprod(e^i\wedge \psi_j)-\delta_{\nu i}\psi_j}e^{ij}
\\&=
2\beta_{\mu\nu}\bracks{e^\nu\wedge \psi_{\mu},e^i\wedge \psi_j}e^{ij}
-2\beta_{\mu\nu}\bracks{\psi_\mu,\psi_j}e^{\nu j}
\\&=
-2\bracks{{\beta\diamond \psi}  ,e^i\wedge \psi_j}e^{ij}
-\frac{1}{3}\beta_{\mu\nu}\psi_{\mu abc}\psi_{jabc}e^{\nu j}.
\end{align*}
By Lemma~\ref{lemma:computing:alpha:e:a_omega1}, $\psi_{\mu abc}\psi_{jabc}=-3!\cdot \alpha\delta_{\mu j}$ for $\mu,j\le L$. Hence,
\begin{align*}
    b\beta=\beta\iprod\psi &=
    \frac{b^2}{a_{L}}\beta +\sum_{\mu,j\le L}\frac{1}{4a_{L}}\pp{-2\beta_{\mu\nu}\alpha\delta_{\mu j}}e^{\nu j}
    =\frac{b^2}{a_{L}}\beta +\sum_{\mu,j\le L}\frac{\alpha}{2a_{L}}{\beta_{j\nu}}e^{j\nu}
    .
\end{align*}
If $\mu>L$, however, then $\psi_{\mu}=0$ and \(2b\beta_{\mu a}=\beta^{pq}\psi_{pq\mu a}=0\), so $\beta_{\mu a}=0$ because $b\neq 0$. Therefore, the sum above may be taken over all indices, and we obtain $a_Lb\beta=b^2\beta+\alpha\beta$.
\end{proof}

The two solutions of the equation in Lemma~\ref{lemma:technical} are distinct and non-zero. Indeed, its discriminant is $a_L^2-4\alpha>0$, since $\alpha<0$ by Lemma~\ref{lemma:computing:alpha:e:a_omega1}. Moreover, zero cannot be a solution, as this would imply $\alpha=0$, again contradicting Lemma~\ref{lemma:computing:alpha:e:a_omega1}. This observation does not, however, characterise the eigenvalues of $\bH_\psi^2|_{\fg}$: although the equation in Lemma~\ref{lemma:technical} has two distinct non-zero solutions, $\bH_\psi^2|_{\fg}$ may have only one of them as an eigenvalue. This occurs, for instance, for $\SU(m)$-, $\rG_2$-, and $\spin(7)$-structures determined by their defining 4-forms, as described in \S\ref{sec:examples:4-form:geometry}.

%\newpage
\phantomsection
\addcontentsline{toc}{section}{References}
\begingroup
\hbadness=3000
\bibliographystyle{0config/alpha-initials}
\bibliography{0config/bib}
\endgroup

\end{document}